\documentclass[AMA,STIX1COL]{WileyNJD-v2}
\usepackage{moreverb}
\usepackage{orcidlink}

\newcommand\BibTeX{{\rmfamily B\kern-.05em \textsc{i\kern-.025em b}\kern-.08em
T\kern-.1667em\lower.7ex\hbox{E}\kern-.125emX}}

\articletype{Article Type}%

\received{<day> <Month>, <year>}
\revised{<day> <Month>, <year>}
\accepted{<day> <Month>, <year>}

\newtheorem*{conj}{Conjecture}

\providecommand{\abs}[1]{\left\lvert#1\right\rvert}
\DeclareMathOperator{\supp}{supp}

\begin{document}

\title{Periodic and chaotic dynamics in a map-based neuron model}
\author[1]{Frank Llovera Trujillo}
\author[2,3]{Justyna Signerska-Rynkowska}
\author[2]{Piotr Bart{\l}omiejczyk*}

\authormark{FRANK LLOVERA TRUJILLO \textsc{et al}}

\address[1]{\orgdiv{Doctoral School}, 
\orgname{Gda\'nsk University of Technology}, 
\orgaddress{\state{Gdańsk}, 
\country{Poland}}}

\address[2]{\orgdiv{Faculty of Applied Physics and 
Mathematics and BioTechMed Centre}, 
\orgname{Gda\'nsk University of Technology},
\orgaddress{\state{Gdańsk}, 
\country{Poland}}}

\address[3]{\orgdiv{Dioscuri Centre in Topological Data Analysis},
\orgname{Institute of Mathematics of 
the Polish Academy of Sciences}, 
\orgaddress{\state{Warsaw}, 
\country{Poland}}}

\corres{*Piotr Bart{\l}omiejczyk, 
Faculty of Applied Physics and 
Mathematics and BioTechMed Centre,
Gda\'nsk University of Technology, 
Gabriela Narutowicza 11/12,\, 80-233 Gda{\'n}sk, Poland. 
\email{piobartl@pg.edu.pl}}

\keywords{Neuronal dynamics; S-unimodal map; 
flip bifurcation; fold bifurcation; chaos.\\[2mm]
\textbf{MSC CLASSIFICATION:}\\
Primary: 37E05, 37N25; Secondary: 37G99, 92C20
}

\abstract[Abstract]{
Map-based neuron models are an important tool in modelling neural
dynamics and sometimes can be considered as an alternative to usually
computationally costlier models based on continuous or hybrid
dynamical systems. However, due to their discrete nature, rigorous
mathematical analysis might be challenging. We study a discrete model
of neuronal dynamics introduced by Chialvo [Chaos, Solitons \&
Fractals~5, 1995, 461--479]. In particular, we show that its reduced
one-dimensional version can be treated as an independent simple model
of neural activity where the input and the fixed value of the recovery
variable are parameters. This one-dimensional model still displays
very rich and varied dynamics. Using the fact that the map whose
iterates define voltage dynamics is S-unimodal, we describe in detail
both the periodic behaviour and the occurrence of different notions of
chaos, indicating corresponding regions in parameter space.
Our study is also complemented by a bifurcation analysis of 
the mentioned dynamical model.}

\maketitle

\section*{Introduction} 

In the last decades many different types 
of neuron models have been developed, including conductance-based
models (with the pioneering example of the experimentally derived
Hodgkin-Huxley model \cite{Hodgkin} or Morris-Lecar model
\cite{MorrisLecar}) and other models based on continuous dynamical
systems, i.e., differential equations (e.g. FitzHugh-Nagumo model
\cite{Fitz,Nagumo}), models based on discrete dynamical systems,
i.e., iterates of functions (so-called \emph{map-based models}), 
up to hybrid models which combine ordinary differential equations
(accounting for input integration) with discrete events (resetting
mechanism  accounting for action potential), see e.g.
\cite{BretteGerstner,izi2003,wild2,touboul2009}.

In this paper we study the following map-based neuron model,
introduced in \cite{chialvo1995}:
\begin{subequations}
\begin{align}
        x_{n+1}&=f(x_n,y_n)=x_n^2\exp{(y_n-x_n)}+k, \label{eq:Chialvo1} \\
        y_{n+1}&=g(x_n,y_n)=a y_n-b x_n +c. \label{eq:Chialvo2}
\end{align}
\end{subequations}
This model will be referred to as the 2D Chialvo model 
(or full Chialvo model). In the 2D Chialvo model
\eqref{eq:Chialvo1}--\eqref{eq:Chialvo2} $x$ is a membrane
voltage-potential (the most important dynamical variable in all 
the neuron models) and $y$ is so-called recovery variable. 
The time-constant $a\in (0,1)$, the activation-dependence 
$b\in (0,1)$ and the offset  $c>0$ are the real parameters 
connected with the recovery process. In turn, $k\geq 0$ can 
be interpreted as an additive perturbation of the voltage.  
In our analysis, the 1-dimensional subsystem
\begin{equation}
\label{eq:Chialvo1DIM}
        x_{n+1}=f(x_n,r)=x_n^2\exp{(r-x_n)}+k,        
\end{equation}
where $r\in\mathbb{R}$ is a parameter, will be referred to as 
the 1D or reduced Chialvo model. 

Other examples of map-based models are, among others,
Cazelles-Courbage-Rabinovich model \cite{cazelles2001}, 
different types of Rulkov models
\cite{rulkov2001,rulkov2002,ShilnikovRulkov2004} or
Courbage-Nekorkin-Vdovin model \cite{courbage2007}. 
In particular, the work of Zhong et al. \cite{ZhongEtAll2017} 
is devoted to the non-chaotic
Rulkov-model \cite{rulkov2002}, 
where the analysis, similarly as in our case, firstly concerns
bifurcations in the one-dimensional reduced model, i.e., 
fast subsystem describing the evolution of the membrane voltage, 
and shows that it produces fold and flip bifurcations. 
This subsystem also produces chaos as it is shown by numerical
simulations. Different examples of map-based neuron models are
discussed e.g. in review articles \cite{courbage2010,ibarz2011} 
(see also references therein).

In this work we obtain rigorous characterization of chaos and
co-dimension-one bifurcations of fixed points in the subsystem
\eqref{eq:Chialvo1DIM}, relevant for the classification of bursting
neurons. Before we start let us summarize the known results on 
Chialvo model and introduce notations. Obviously, 
the fixed points $(x_f,y_f)$ of the full model must satisfy 
\begin{align*}
    x_{f}&=x_f^2\exp(y_f-x_f)+k,\\
    y_{f}&=(c-b x_f)/(1-a).
\end{align*}
In particular, for $k=0$ the point $(x_{f_0},y_{f_0}):=(0, c/(1-a))$
is always a stable fixed point of the system as can be seen by
calculating corresponding eigenvalues which are $0$ and $a$ 
(and $a<1$ by the assumption). Further, as Chialvo observed 
\cite{chialvo1995}, when $b\ll a$, the initial part of 
the trajectory with initial condition $(x,y_{f_0})$ 
where $x\approx 0$ can be well described by the 1D Chialvo model
\eqref{eq:Chialvo1DIM} with $k=0$ and $r=y_{f_0}=c/(1-a)$. 
Therefore in this parameter regime, among others, the 1D subsystem
gives the information about initial dynamics following the small
perturbation of the resting state $x=0$ in the full model. 
A few other observations for $k>0$ (referring mainly to the situation
where the system \eqref{eq:Chialvo1}--\eqref{eq:Chialvo2} has exactly
one fixed point $(x_f,y_f)$) were made already in Chialvo's work \cite{chialvo1995},
including identification of some regions of bistability (coexistence
of stable resting state and stable periodic oscillations) and regions
of chaotic-like behaviour interspersed with subregions of periodic
behaviour. 

For more results on the full model, we also refer 
the reader to the more recent works \cite{Jing2006,NewPaperOnChialvo}. 
In particular, Jing et al. \cite{Jing2006} studied
analytically the existence and stability of the fixed points,
conditions for existence of flip, fold and Hopf bifurcations 
and chaos in the sense of Marotto's definition. Further, it
numerically reports various period doubling bifurcations and 
routes to chaos, periodic windows in transient chaotic regions 
and strange chaotic and non-chaotic attractors, 
whereas the work of Wang and Cao
\cite{NewPaperOnChialvo} illustrates numerically the existence 
of interesting structures in parameter space. 

As our work is mainly concerned with 
the 1-dimensional subsystem \eqref{eq:Chialvo1DIM} 
we have only briefly described above 
the findings on the full 2D model. 
However, we go back to the analysis 
of the full model at the end of the paper.   
Let us also observe that for $0<\varepsilon=1-a=b=c\ll1$,
the system \eqref{eq:Chialvo1}--\eqref{eq:Chialvo2} can be seen 
as a \emph{slow-fast} (discrete) system, i.e. the system of the form
\begin{subequations}
\label{eq:Chialvo}
\begin{align}
        x_{n+1}&=\tilde{f}(x_n,y_n),       \label{eq:small1} \\
        y_{n+1}&=y_n+\varepsilon\tilde{g}(x_n,y_n), \label{eq:small2}
\end{align}
\end{subequations}
where $\varepsilon\to 0$ is a small parameter. For this particular
assumptions on parameters, the separation of timescales in 
the Chialvo model \eqref{eq:Chialvo1}--\eqref{eq:Chialvo2} is
explicit. However, for many other choices of parameter values 
the timescale separation is not so explicit but still visible 
in simulations. In this case, the voltage variable $x$ and 
the recovery variable $y$, can be referred to as \emph{fast} 
and \emph{slow} variables, respectively. Consequently, 
while $x_n$  describes spiking behavior, $y_n$ acts as a slowly
changing parameter (with time-scale variation $\varepsilon \ll 1$)
that modulates the spiking dynamics. 

Models of the form \eqref{eq:small1}--\eqref{eq:small2} 
are usually analized by firstly assuming that $\varepsilon=0$ 
and treating \eqref{eq:small1} as a quazi-static 
approximation of \eqref{eq:small1}--\eqref{eq:small2}
with parameter $y_n\equiv y$. If for some values of $y$
\eqref{eq:small1} exhibits equilibrium dynamics 
(due to the existence of a stable fixed point) and for some other
values it exhibits periodic dynamics (due to the existence of 
a stable periodic orbit), then bursting in the system
\eqref{eq:small1}--\eqref{eq:small2} occurs because slowly 
varying $y_n$ acts as bifurcation parameter that makes 
the dynamics of $x$ switching between these two regimes 
(similar approach can be applied for ODE bursting systems, compare
e.g. with \cite{Rinzel1987}). Therefore bursting behaviour in 
the above neuron model is indeed directly connected with the types 
of bifurcations present in the fast system \eqref{eq:small1}
(correspondingly, in  \eqref{eq:Chialvo1} or \eqref{eq:Chialvo1DIM}).
According to this observation, in \cite{iziHoppen2004} 
the classification of bursting mappings was obtained, 
taking into account that there are only four possible
co-dimension-1 bifurcations of a stable fixed point 
in maps \eqref{eq:small1} that lead to its
loss of stability or disappearance (fold, SNIC-saddle node 
on invariant circle bifurcation and supercritical 
and subcritical flip bifurcations) and only five 
co-dimension one bifurcations of a stable
periodic orbit, which make the dynamics settle 
on a stable equilibrium afterwards 
(SNIC, homoclinic bifurcation, supercritical flip
bifurcation, fold periodic orbit bifurcation and 
subcritical flip of periodic orbits bifurcation). 
The former bifurcations correspond to transition from 
resting to spiking frequently (i.e. the onset of
bursting), the later ones to the transition back 
to resting and their combination defines the type of 
burster in a given neuron model. 
We also remark that classification of bursting types 
in ODEs neuron models was discussed e.g. in 
\cite{BertramEtAll1995} and \cite{Rinzel1987}
as well as in the recent
article  \cite{Desroches2022}  (see also review article 
\cite{izi2000} and references therein).

It is also worth pointing out that, 
contrary to most of the 2D map based-models, 
the recovery variable $y$ in Chialvo model 
enters the voltage equation \eqref{eq:Chialvo1} 
not in additive but in multiplicative way and changes 
the shape of the first return map $x_n\mapsto x_{n+1}=f(x_n,y)$ 
of the $x$ variable in a complicated nonlinear way instead 
of just shifting the $x$-nullcline upward or downward. 
Therefore understanding how the map $f_{r,k}(x):=x^2\exp{(r-x)}+k$ 
depends not only on the additive  parameter $k$ 
but also on the parameter $r$ is helpful for 
the analysis of the overall model. 
In fact $y$ variable in the Chialvo model 
can also be intended to represent fast, not slow dynamics 
(as already mentioned in the review article \cite{ibarz2011}). 
When playing with the 2D model we have often observed that 
for many parameters configurations $y$ variable quickly 
stabilized at certain value $y_s$ while the voltage variable $x$ 
still evolved in an oscillatory way. In such a case  types 
of voltage oscillations can be well explained by understanding 
the dynamics of the 1D model with corresponding value 
of parameter $r=y_s$ (we exemplify this in section \ref{sec:implications}).

The main aim of this paper is to study the 1D Chialvo
model using the methods of one-dimensional dynamics.
According to our best knowledge, 
it has never been pointed out that the map describing 
the membrane voltage evolution in the Chialvo model is 
an $S$-unimodal map and thereby there are no results on 
this model, which use the theory of $S$-unimodal maps. 
This basic observation that we make allows to draw 
relevant conclusions for this map, as well as, 
in some cases, for the whole model itself.
Another example of the occurrence of 
unimodal maps in the context of neuron dynamics 
is the adaptation map of the bidimensional 
integrate-and-fire model \cite{wild1} 
(however, this unimodal map is given 
in a very implicit way and hardly ever 
exhibits negative Schwarzian derivative).
In addition, it may be worth pointing out that
various discrete systems similar to
the 1D Chialvo model, for example the logistic family
$Q(x)=\lambda x(1-x)$, the Ricker family 
$R(x)=\lambda x\exp(-\beta x)$ and the Hassell family
$H(x)=\lambda x/(1+x)^\beta$, have been the subject of 
both theoretical and numerical intensive research study 
since the mid-1970s (see for instance 
\cite{May1976,Cohen1995,Brauer2001,
thunberg2001,Thieme+2018}). 

Let us also note that $f(z)=z^2\exp{(r-z)}+k$ 
(for the fixed complex values of parameters $r$ and $k$) 
can be seen as an example of a complex 
entire transcendental analytic  
function of finite type, i.e., having finite number 
of critical and asymptotic values. 
Since such maps exhibit, in general, different dynamical
properties than polynomials,
their complex dynamics (Julia and Fatou sets, etc.) 
has been studied extensively since 
the early 1980s (see 
\cite{misiurewicz_1981,Schleicher2010,Alhamed2022}
and the references given there).
However, the approach presented in this paper
is focused on one-dimensional methods
and, in consequence, we do not make use 
of complex dynamics theory in our paper.

The organization of the paper is as follows. 
In Section~\ref{sec:bif} we examine the existence and 
stability of fixed points and corresponding bifurcations 
in the one-dimensional Chialvo 
model~\eqref{eq:Chialvo1DIM}.
Section~\ref{sec:dyncore} is devoted to the study
of the dynamical core of the 1D Chialvo map.
Further, in Section~\ref{sec:periodic} we provide 
the description of periodic and chaotic 
behaviour of the model making use of the fact that
\eqref{eq:Chialvo1DIM} is given by an S-unimodal map. 
To ease our analysis we introduce basic concepts 
and definitions of the theory of S-unimodal maps 
in Section~\ref{sec:prel}. Finally, concluding remarks 
are presented in Discussion section at the end of the paper. 

\section{Basic definitions}\label{sec:prel}

\subsection{Unimodality and Schwarzian derivative}

We say that an interval map $f\colon I\rightarrow I$ has
\emph{negative Schwarzian derivative} if $f$ is 
of class $C^3$ and 
\[
Sf(x):= 
\frac{f'''(x)}{f'(x)}-
\frac{3}{2}\left(\frac{f''(x)}{f'(x)}\right)^2<0
\]
for all 
$x\in I \setminus \left\{ c \mid f'(c) = 0 \right\}$.
Points at which $f'(c)=0$ are called \emph{critical points}.
The Schwarzian derivative was introduced in one-dimensional
dynamics by D. Singer\cite{Singer1978}, but its origins go back
to Hermann Schwarz and 19th century complex analysis.

A continuous interval map 
$f\colon I = \left[a,b\right]\to I$ is called \emph{unimodal} 
if there is a unique maximum $c$ in the interior of $I$ such 
that $f$ is strictly increasing on $\left[a,c \right)$ and 
strictly decreasing on $\left(c,b\right]$. For simplicity, 
the term unimodal will also require that either $a$ is 
a fixed point with $b$ as its other preimage, or that
$I=\left[f^2(c),f(c)\right]$. A unimodal map with 
negative Schwarzian derivative is called \emph{S-unimodal}.

The \emph{itinerary} of $x$ under $f$ is the sequence 
$s(x)=(s_0s_1s_2\dotsc)$ where
\[
s_j=\begin{cases}
0& \text{if $f^j(x)<c$},\\
1& \text{if $f^j(x)>c$},\\
C& \text{if $f^j(x)=c$}.
\end{cases}
\]
The \emph{kneading sequence} $K(f)$ of $f$ is the itinerary
of $f(c)$, i.e., $K(f)=s\big(f(c)\big)$.

\subsection{Bifurcations}
Let $x_0$ be a fixed point for $f$. The point $x_0$ is called
\emph{hyperbolic} if $\abs{f'(x_0)}\neq1$. The number
$\mu=f'(x_0)$ is called the \emph{multiplier} of the fixed point.
Consider a one-dimensional discrete dynamical system 
depending on a parameter $a$
\begin{equation}\label{eq:system}
x\mapsto f(x,a),\quad x\in \mathbb{R},\; a\in \mathbb{R},
\end{equation}
where $f$ is smooth with respect to both $x$ and $a$.
Let $x_0$ be a hyperbolic fixed point of the system. 
While the parameter $a$ varies, the hyperbolicity condition
can be violated. The bifurcation associated with 
the appearance of the multiplier $\mu=1$ is called 
\emph{fold} or \emph{tangent} and with
the appearance of the multiplier $\mu=-1$ \emph{flip}
or \emph{period-doubling}. 
Alternatively a fold bifurcation is called 
a \emph{saddle-node} bifurcation.
Let us recall 
a classical result concerning these bifurcations
(\cite[Th. 4.1 and 4.3]{kuznetsov}).

\begin{theorem}\label{bif_theorem}
Suppose that a one-dimensional system \eqref{eq:system}
with smooth $f$, has at $a=0$ the fixed point $x_0=0$.
Let $\mu := \frac{\partial f}{\partial x}(0,0)$.
If $\mu=1$ and the following
conditions are satisfied:
\begin{enumerate}
    \item[(A1)] $\frac{\partial^2 f}{\partial x^2}(0,0)\neq 0$,
    \item[(A2)] $\frac{\partial f}{\partial a}(0,0)\neq 0$,
\end{enumerate}
then a fold bifurcation occurs at 
the fixed point $x_0=0$ for the bifurcation value $a=0$.

On the other hand, if 
$\mu = -1$ and the
following conditions are satisfied:
\begin{enumerate}
    \item[(B1)]
    $\frac{1}{2}
    \Big(\frac{\partial^2 f}{\partial x^2}(0,0)\Big)^2
    +\frac{1}{3}\frac{\partial^3 f}{\partial x^3}(0,0)
    \neq 0$,
    \item[(B2)] 
    $\frac{\partial^2 f}{\partial x\partial a}(0,0)\neq 0$,
\end{enumerate}
then a flip bifurcation occurs at the fixed point $x_0=0$ for 
the bifurcation value $a=0$.
\end{theorem}

\begin{remark}\label{remarkSupercritical}
Based on the nonlinear stability of the fixed point at 
the bifurcation point we distinguish two cases for the flip
bifurcation. If the quantity 
\[
\mathcal{Q}f(0,0):=\frac{1}{2}
    \Big(\frac{\partial^2 f}{\partial x^2}(0,0)\Big)^2
    +\frac{1}{3}\frac{\partial^3 f}{\partial x^3}(0,0)
\]    
in the (B.1) condition above is positive, the fixed point loses 
its stability at the bifurcation value and new stable $2$-periodic
orbit emerges, in which case the flip bifurcation is called
\emph{supercritical}. Otherwise, i.e. when $\mathcal{Q}f(0,0)<0$, 
the fixed point turns unstable as it coalesces with an unstable
$2$-periodic orbit and the bifurcation is called \emph{subritical}.
\end{remark}

\subsection{Attractors}

A set $\Gamma$ is called \emph{forward invariant} 
if $f(\Gamma) = \Gamma$. 
The $\omega$-\emph{limit set} (or \emph{forward limit set}) 
of $x$ is defined as
$\omega(x)=
\bigcap_{n\in\mathbb{N}}
\overline{\left\{f^k(x)\mid k>n\right\}}$,
where $\overline{A}$ stands for the closure of $A$.
Let $B(\Gamma)$ denote the 
\emph{basin of attraction} of a forward invariant set 
$\Gamma$, that is, 
$B(\Gamma) =\{x\mid \omega(x)\subset\Gamma\}$.
A forwards invariant set $\Omega$ is called 
a \emph{metric attractor} if $B(\Omega)$ satisfies:
\begin{enumerate}
    \item $B(\Omega)$ has positive Lebesgue measure;
    \item if $\Omega '$ is another forward invariant set, strictly
    contained in $\Omega$, then $B(\Omega)\setminus B(\Omega ')$ 
    has positive measure.
\end{enumerate}

We will need two well-known results concerning attractors
for unimodal maps with negative Schwarzian derivative
(for example \cite[Th. 4, Cor. 5 and Th. 6]{thunberg2001}).
Let $x$ be a periodic point of period $n$.
Its periodic orbit is called \emph{neutral}
if $\abs{(f^n)^{\prime}(x)}=1$.

\begin{theorem}\label{thm:singer}
An S-unimodal map can have at most one periodic attractor, and
it will attract the critical point. Moreover, each neutral 
periodic orbit is attracting.
\end{theorem}
By a periodic attractor in the above statement we obviously mean 
an attracting periodic orbit, i.e. a periodic orbit which is 
a forward-limit set for each point in some its neighbourhood.
 
\begin{theorem}\label{thm:attractorgen}
Let $f\colon I \rightarrow I$ be an S-unimodal map with nonflat
critical point. Then $f$ has a unique metric attractor $\Omega$, 
such that $w(x)=\Omega$ for almost all $x\in I$. 
The attractor $\Omega$ is of one of the following types:
an attracting periodic orbit,
a Cantor set of measure zero,
a finite union of intervals with a dense orbit.
In the first two cases, $\Omega = \omega(c)$.
\end{theorem}
The attractor described in \emph{(3)} will be called an
\emph{interval attractor} for brevity.
 
\section{Bifurcations in the Chialvo model}\label{sec:bif}

We start with the observation that will be crucial for our further analysis.
\begin{proposition}
The map $f(x)=x^2\exp{(r-x)}+k$ 
has negative Schwarzian derivative
for all $r,k\in \mathbb{R}$.
\end{proposition}

\begin{proof}  
By the definition of the Schwarzian derivative,
\[
Sf(x)= 
-\frac12\frac{x^4-8x^3+24x^2-24x+12}{(2x-x^2)^2}=
-\frac12\frac{(x^2-4x+2)^2+4(x-1)^2+4}{(2x-x^2)^2}.
\]
This is obviously negative for all $x$
(even for $x$ equal to $0$ or $2$ we have $Sf(x)=-\infty$),
which completes the proof. 
\end{proof}

Thus the function whose iterations define the reduced Chialvo model
\eqref{eq:Chialvo1DIM} is an S-unimodal map, when restricted to 
an appropriate invariant subinterval of $[0,\infty)$, as will be
discussed in details in the next section. However, let us firstly
examine bifurcations of fixed points of the map $f_{r,k}$, 
when $k\geq 0$ is fixed and $r$ acts as a bifurcation parameter. 

We will discuss flip and fold bifurcations of fixed points for 
the reduced Chialvo model. Loosing stability of the fixed point 
and emergence of the stable periodic orbit in case of the flip
bifurcation is relevant for the onset of bursting in 
the full 2D model.  Let us also mention that other codimension-one
bifurcations of fixed points, as well as bifurcations of periodic
orbits responsible for the onset or termination of bursting in
map-based models, are summarized in \cite{iziHoppen2004} 
(with examples of models and corresponding voltage-plots), 
which provides classification of one-dimensional and 
two-dimensional slow-fast bursting mappings with the implications 
for neuron's computational abilities.    

\subsection{Flip bifurcation}

\begin{theorem}[flip bifurcation]\label{thm:flip}
Let $k\geq 0$ be fixed. For $r_0=x_0-\ln\big(x_0(x_0-2)\big)$ 
(or equivalently $r_0=x_0+\ln\big((x_0-k)/{x_0^2}\big)$) 
the 1D Chialvo map undergoes a supercritical flip bifurcation 
at the fixed point $x_0=(k+3+\sqrt{k^2-2k+9})/{2}$.  
\end{theorem}

\begin{figure}[!htb]\vspace{-7mm}
\centering{
  \includegraphics[scale=0.325,trim= 36mm 8mm 40mm 10mm]{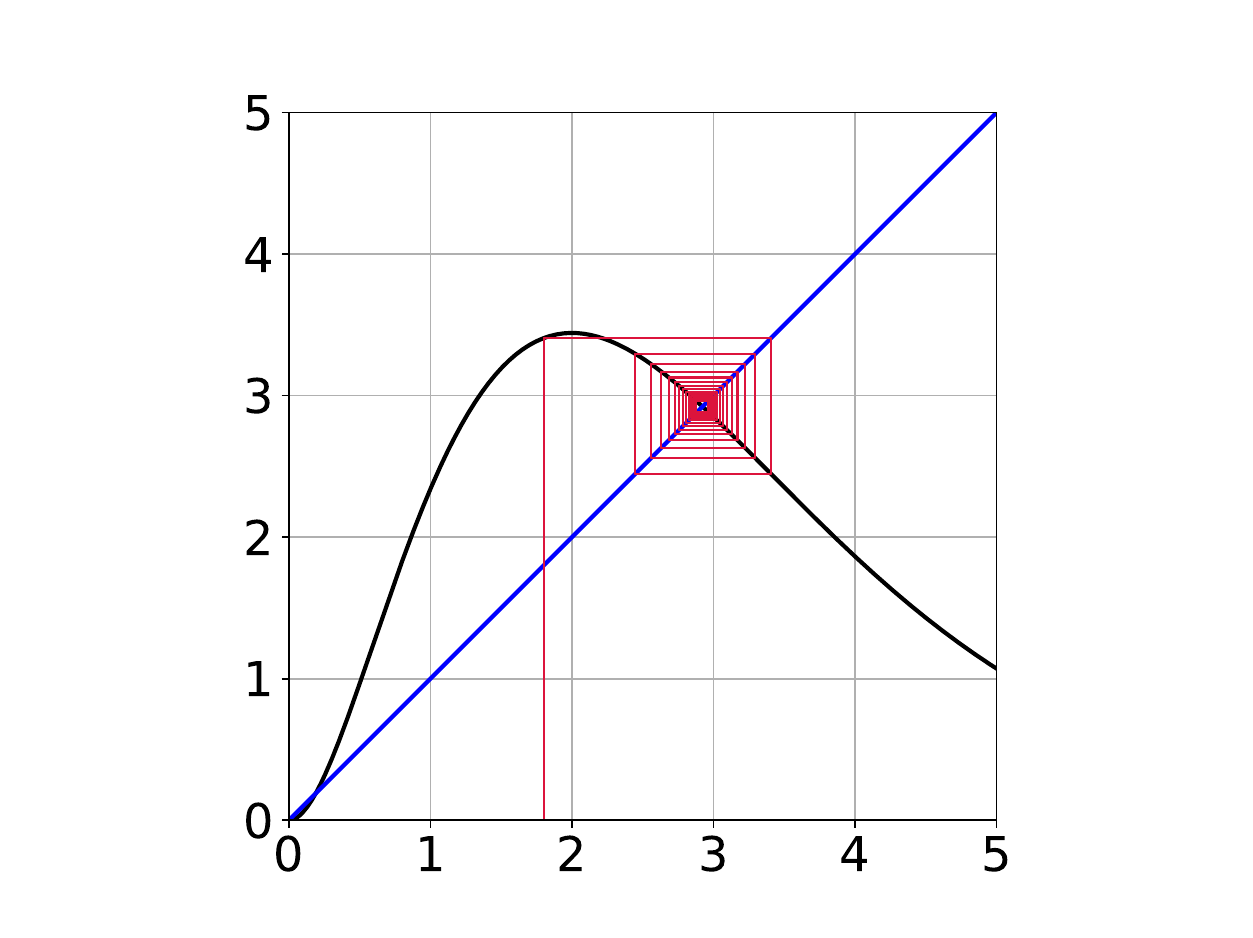}
  \includegraphics[scale=0.325,trim= 5mm 8mm 40mm 10mm]{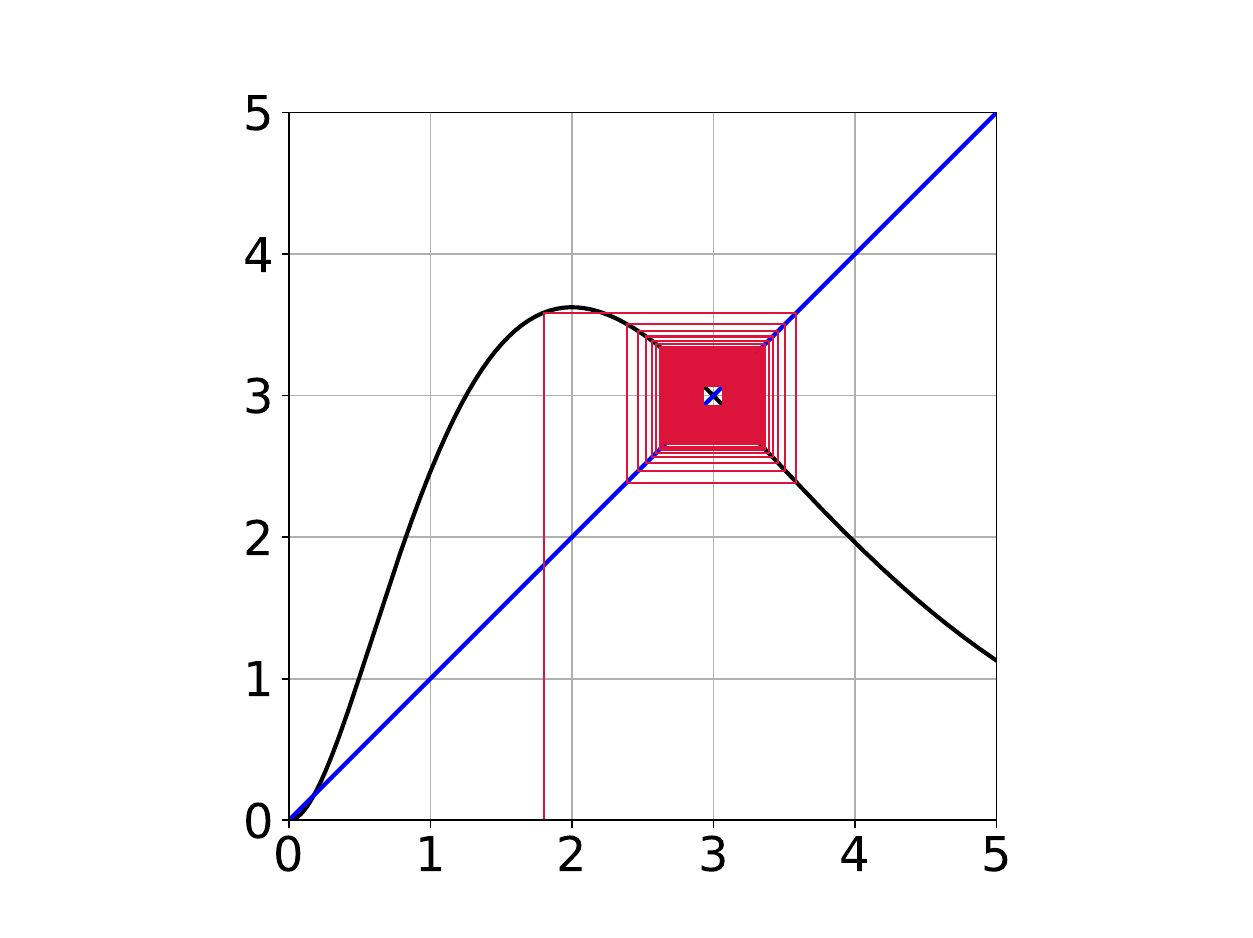}
  \includegraphics[scale=0.325,trim= 5mm 8mm 40mm 10mm]{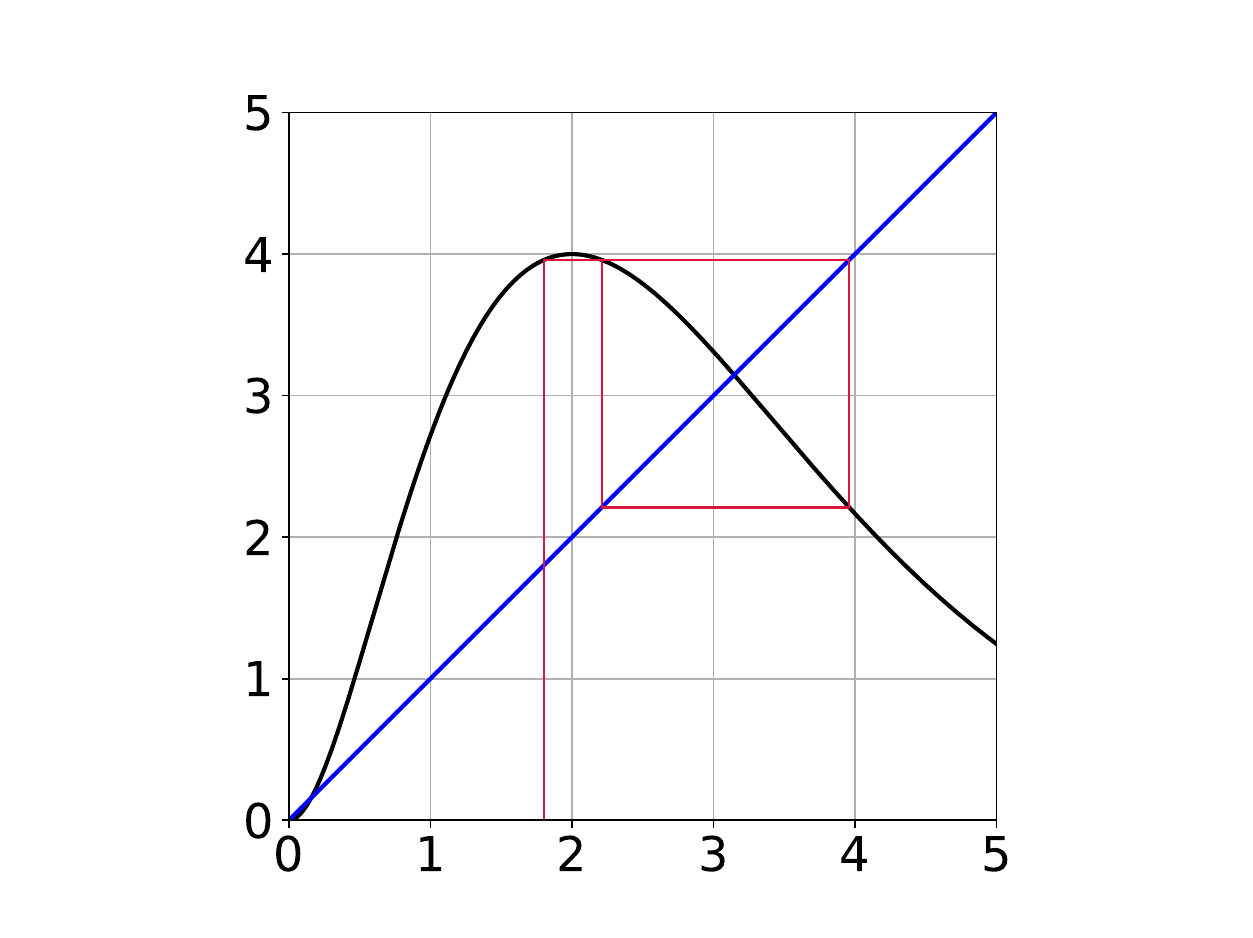}}
\caption{Flip bifurcation in the 1D Chialvo model  with $r$ as 
a bifurcation parameter ($k=0$).  Left: $r = 1.85$. Center: $r = 3 - \ln{3}$ (bifurcation value). Right: $r = 2$.}
\label{fig:flip1}
\end{figure}\vspace{-7mm}

\begin{proof}
Any fixed point $x_0$ of the map $f_{r,k}(x)=x^2\exp(r-x)+k$ 
must satisfy the fixed point condition $x_0=x_0^2\exp(r-x_0)+k$.
Simultaneously, the multiplier condition for 
the flip bifurcation gives 
\[
f^{\prime}_{r_0,k}(x_0)=x_0\exp(r_0-x_0)(2-x_0)=-1. 
\]
It follows that (when $k\geq 0$) the candidate $x_0$ for 
the flip bifurcation satisfies $x_0>\max\{2,k\}$ and
$(x_0-k)(x_0-2)/x_0=1$. The last equation, regardless of 
the value of $k$, has always two roots
$x_{0,1}=(k+3-\sqrt{k^2-2k+9})/{2}$ and
$x_{0,2}=(k+3+\sqrt{k^2-2k+9})/{2}$. 
However, $x_{0,1}$ violates the condition that $x_0>k$. 
Thus $x_0=x_{0,2}$. The formula for the parameter value $r_0$ 
follows from the multiplier condition (or equivalently 
the fixed point condition) and the condition (B.2) of 
Theorem \ref{bif_theorem} can be easily verified as direct
calculations give 
$\frac{\partial^2 f}{\partial r \partial x}(x_0,r_0)=-1\neq 0$.
Similarly, the negative value of the Schwarzian derivative 
immediately implies that the condition (B.1) is also satisfied, 
as $\mathcal{Q}f(x_0,r_0)=-Sf/{3}$, and the bifurcation 
is supercritical.  
\end{proof}

The flip bifurcation, for $k=0$, is numerically illustrated 
in Figure \ref{fig:flip1} with the help of the cobweb diagram. 

\subsection{Fold bifurcation}

\begin{theorem}[fold bifurcation]\label{thm:fold}
Let $k\in [0,3-2\sqrt{2})$ be fixed. Then the 1D Chialvo model
undergoes a fold bifurcation with $r$ as a bifurcation parameter. 
In particular, for $k=0$ and $r_0=1$ the map 
in \eqref{eq:Chialvo1DIM} displays a fold bifurcation at 
the fixed point $x_0=1$. In turn, for fixed $0<k\leq 3-2\sqrt{2}$,
there are two values of the bifurcation parameter,
$r_{0,1}=x_{0,1}-\ln((2-x_{0,1})x_{0,1})$ and
$r_{0,2}=x_{0,2}-\ln((2-x_{0,2})x_{0,2})$, 
with corresponding values of bifurcating fixed points given,
respectively, as $x_{0,1}=(k+1-\sqrt{k^2-6k+1})/{2}$ and
$x_{0,2}=(k+1+\sqrt{k^2-6k+1})/{2}$.  
\end{theorem}

\begin{figure}[!htb]\vspace{-4mm}
\centering{
  \includegraphics[scale=0.32,trim= 30mm 8mm 36mm 10mm]{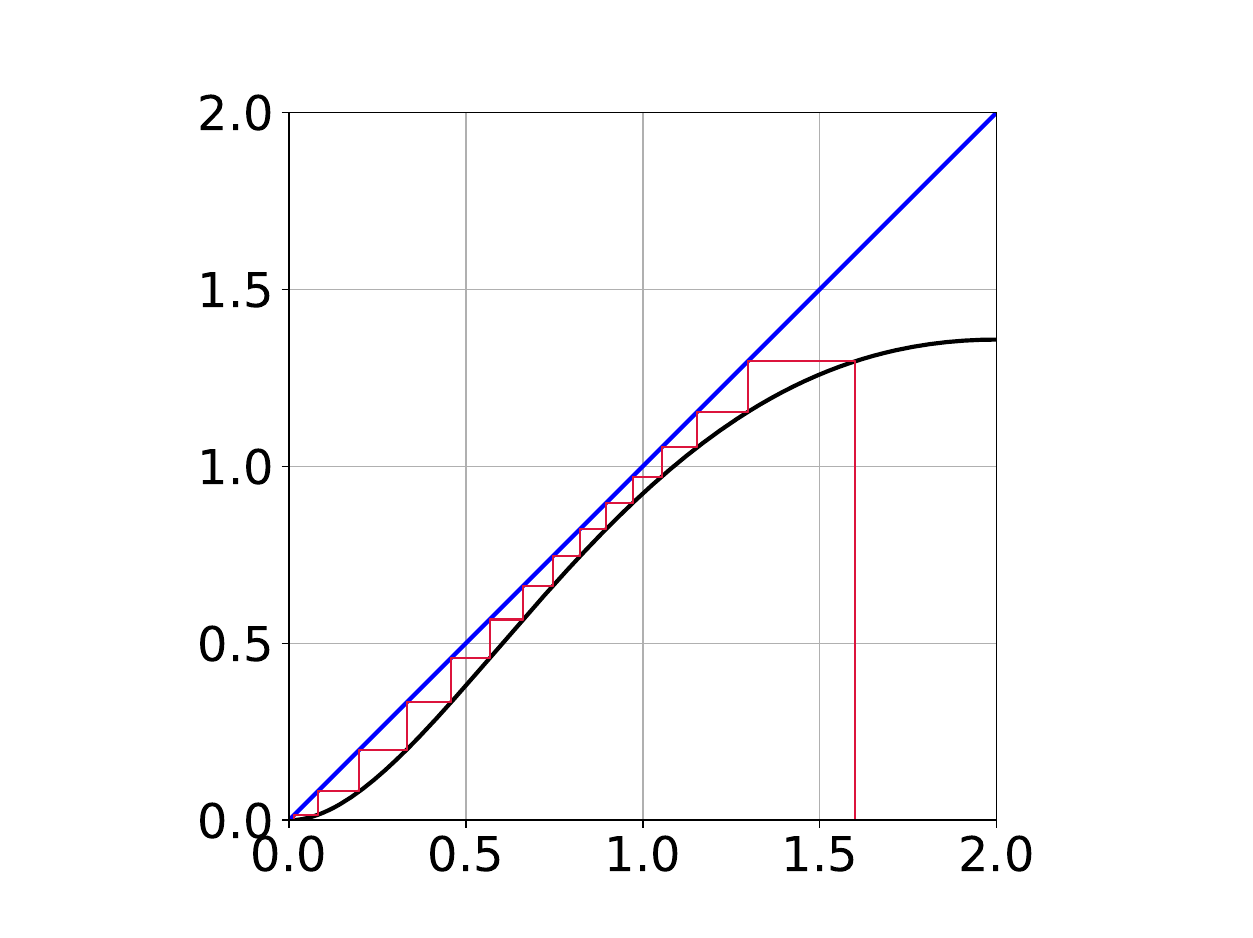}
  \includegraphics[scale=0.32,trim= 5mm 8mm 36mm 10mm]{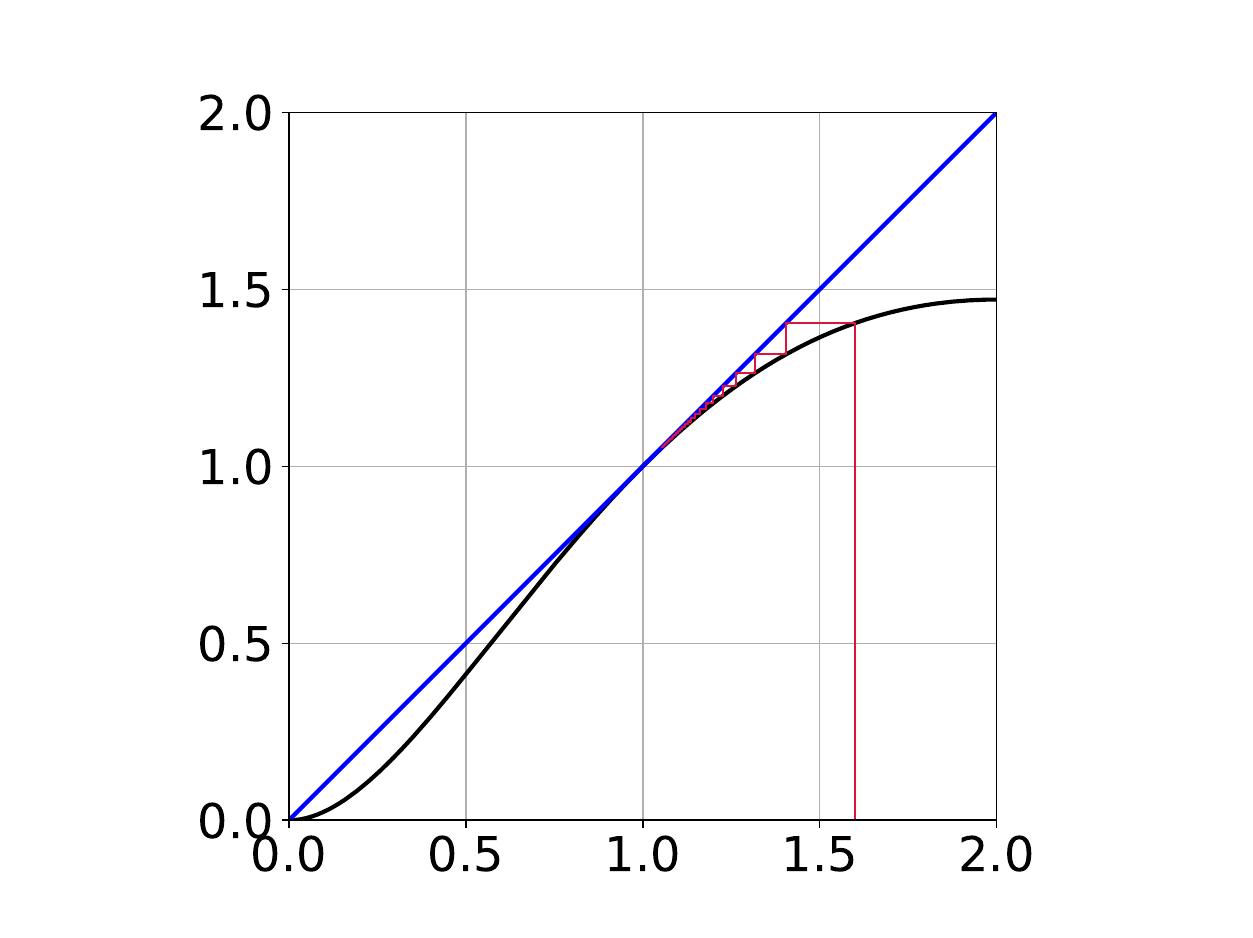}
  \includegraphics[scale=0.32,trim= 5mm 8mm 35mm 10mm]{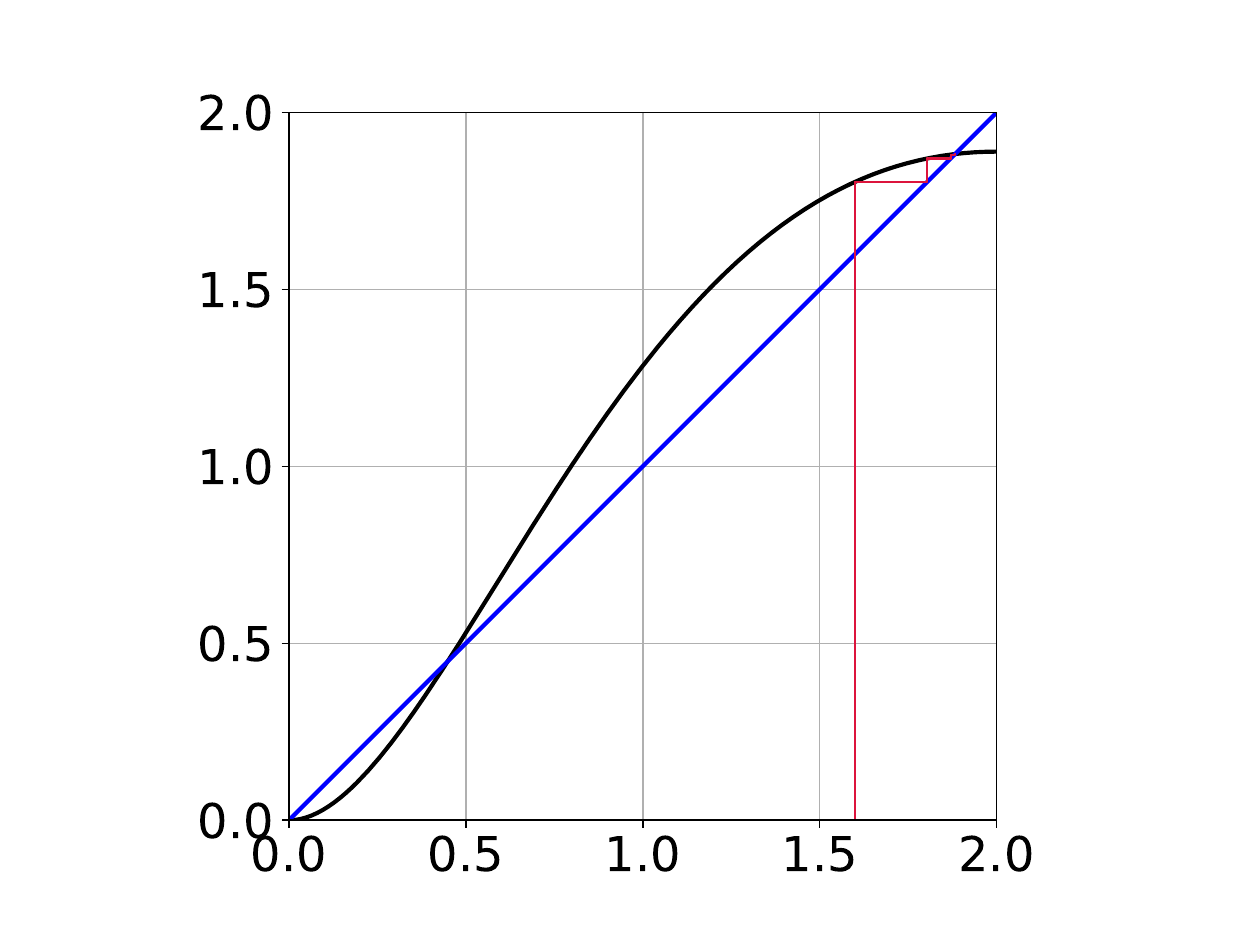}}
\caption{Fold bifurcation in the 1D Chialvo model with $r$ as 
a bifurcation parameter ($k=0$). Left: $r = 0.92$. Center: 
$r = 1$ (bifurcation value). Right: $r = 1.25$.}
\label{fig:saddle-node}
\end{figure}\vspace{-5mm}

\begin{proof}
Assume $k=0$. From the fixed point condition $x^2\exp(r-x)=x$ and 
the multiplier condition $f^{\prime}(x)=x\exp(r-x)(2-x)=1$ 
we easily obtain that the only candidate for the bifurcation point 
and parameter value is a pair $(x_0,r_0)=(1,1)$. 
It is also immediate to verify that for these values 
the conditions (A.1) and (A.2) of Theorem \ref{bif_theorem}
are satisfied.

For $k>0$ the situation is a bit more complicated. Firstly, 
the fixed point condition $x^2\exp(r-x)+k=x$ and the multiplier
condition $\mu=1$ yield that the fold bifurcation is only possible 
for $k<2$ and the candidate $x_0$ for the fixed point undergoing
bifurcation must satisfy $k<x_0<2$. Combining the multiplier 
condition with the fixed point condition gives
$(2-x_0)(x_0-k)/{x_0}=1$, or equivalently $x_0^2-(k+1)x_0+2k=0$. 
The discriminant of this equation $k^2-6k+1$ must be non-negative
which leads to further restriction on admissible values of $k$, 
namely $k\leq 3-2\sqrt{2}$. Now, the above quadratic equation has 
the roots   $x_{0,1}$ and $x_{0,2}$ as stated in the theorem, 
which both belong to the interval $(k,2)$. When $x_0=x_{0,1}$ 
or $x_0=x_{0,2}$ the corresponding values of the bifurcation 
parameter ($r_{0,1}$ and $r_{0,2}$) can be easily derived from 
the multiplier condition $x_0\exp(r-x_0)(2-x_0)=1$. 
We need to verify the conditions (A.1) and (A.2) at 
the points $(x_{0,1},r_{0,1})$ and $(x_{0,2},r_{0,2})$. 
The latter one can only be violated when $x_{0,1}=0$ or $x_{0,1}=k$
(correspondingly,  $x_{0,2}=0$ or $x_{0,2}=k$). But these two cases
have been excluded. As for (A.1), 
$\frac{\partial^2 f}{\partial x^2}(x,r)=\exp(r-x)((x-2)^2-2)$ 
and thus direct calculations show that 
$\frac{\partial^2 f}{\partial x^2}(x_0,r_0)=0$  
(for $(x_0,r_0)\in \{(x_{0,1},r_{0,1}), (x_{0,2},r_{0,2})\}$) 
only for $k=3-2\sqrt{2}$. Hence the restriction $k<3-2\sqrt{2}$ 
in the theorem on the fold bifurcation 
(we notify that $k=3-2\sqrt{2}$ is the degenerate case 
when $x_{0,1}=x_{0,2}=2-\sqrt{2}$).
\end{proof} 

The fold bifurcation with respect to $r$ ($k=0$) is presented 
in Figure \ref{fig:saddle-node}.

\subsection{Bifurcations with respect to $k$}

Let us consider now $k$ as a bifurcation parameter in the reduced
Chialvo model. Since 
$\frac{\partial^2 f }{\partial x \partial k}(x,k)= 0$, 
the fundamental condition for the occurrence of 
a flip bifurcation is not satisfied here. 
However, the analysis of a fold bifurcation versus 
the parameter $k$ is still possible, which is discussed below.

\begin{theorem}[fold bifurcation with respect to $k$]\label{thm:foldk}
Let $r> 2-\sqrt{2}-\ln\big(2\sqrt{2}-2\big)$ be fixed.
Then there is a unique point $x(r)$
and a unique parametr value $k^*$ such that $x(r)$ is the fixed point 
of the 1D Chialvo map in the interval $(0,2-\sqrt{2})$ for $k=k^*$
and $x(r)$ undergoes a fold bifurcation with respect to $k$ 
at the bifurcation value $k^*$. Moreover, the fixed point and 
the the bifurcation value are related by the equation 
\[
k^*=x(r)-\frac{x(r)}{2-x(r)}.
\]
\end{theorem}

\begin{figure}[!htbp]\vspace{2mm}
\centering{
  \includegraphics[scale=0.32,trim= 31mm 8mm 35mm 10mm]{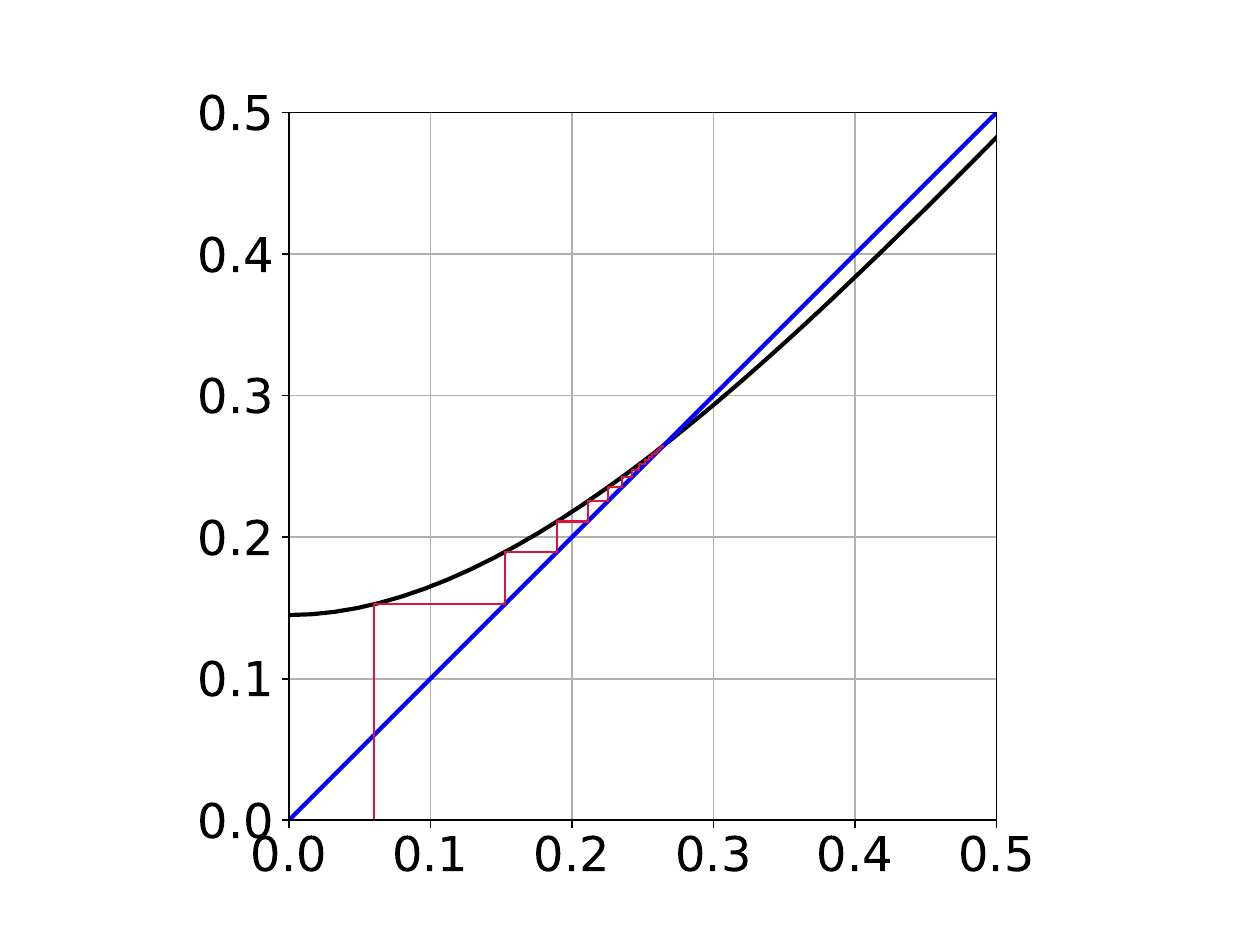}
  \includegraphics[scale=0.32,trim= 5mm 8mm 35mm 10mm]{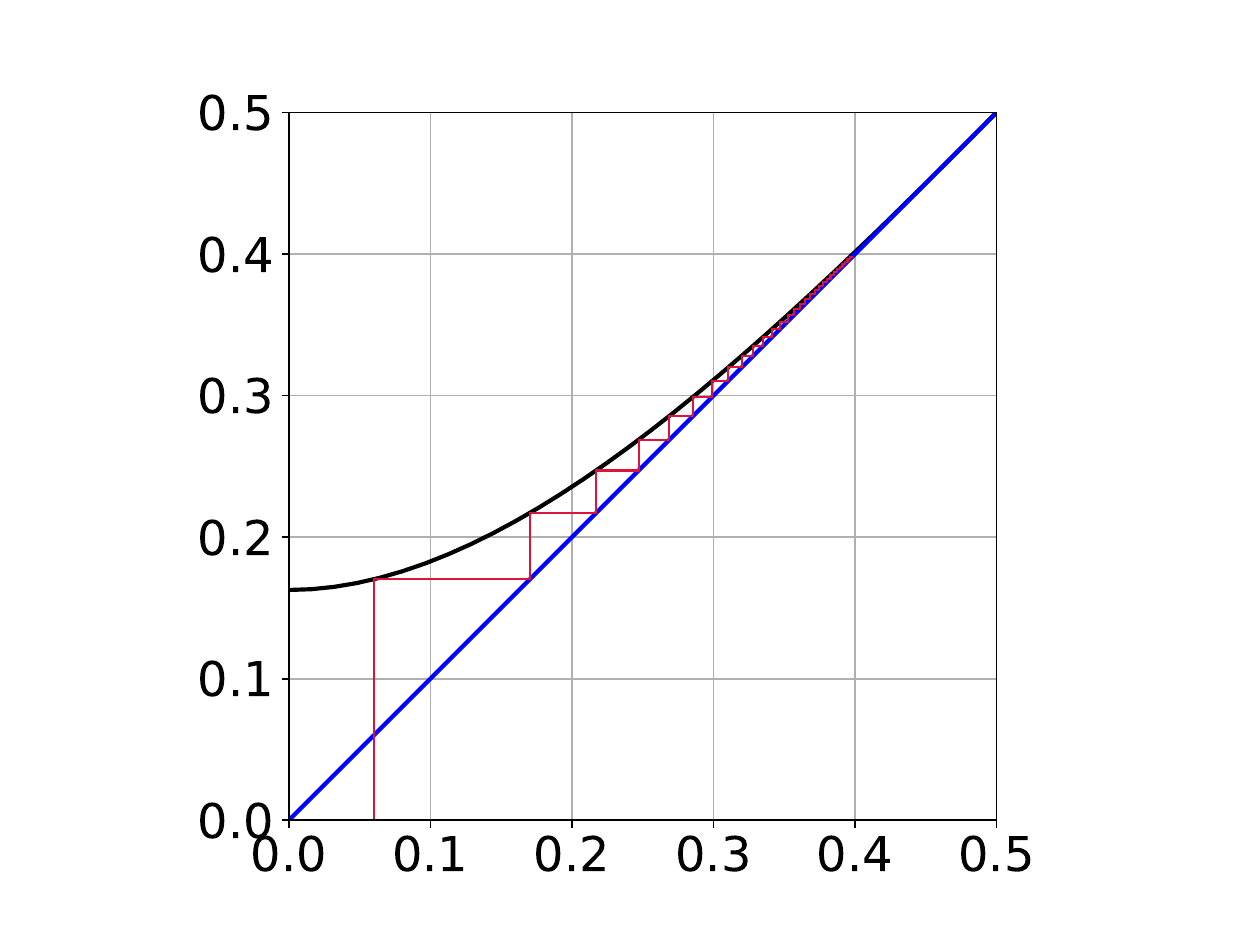}
  \includegraphics[scale=0.32,trim= 5mm 8mm 35mm 10mm]{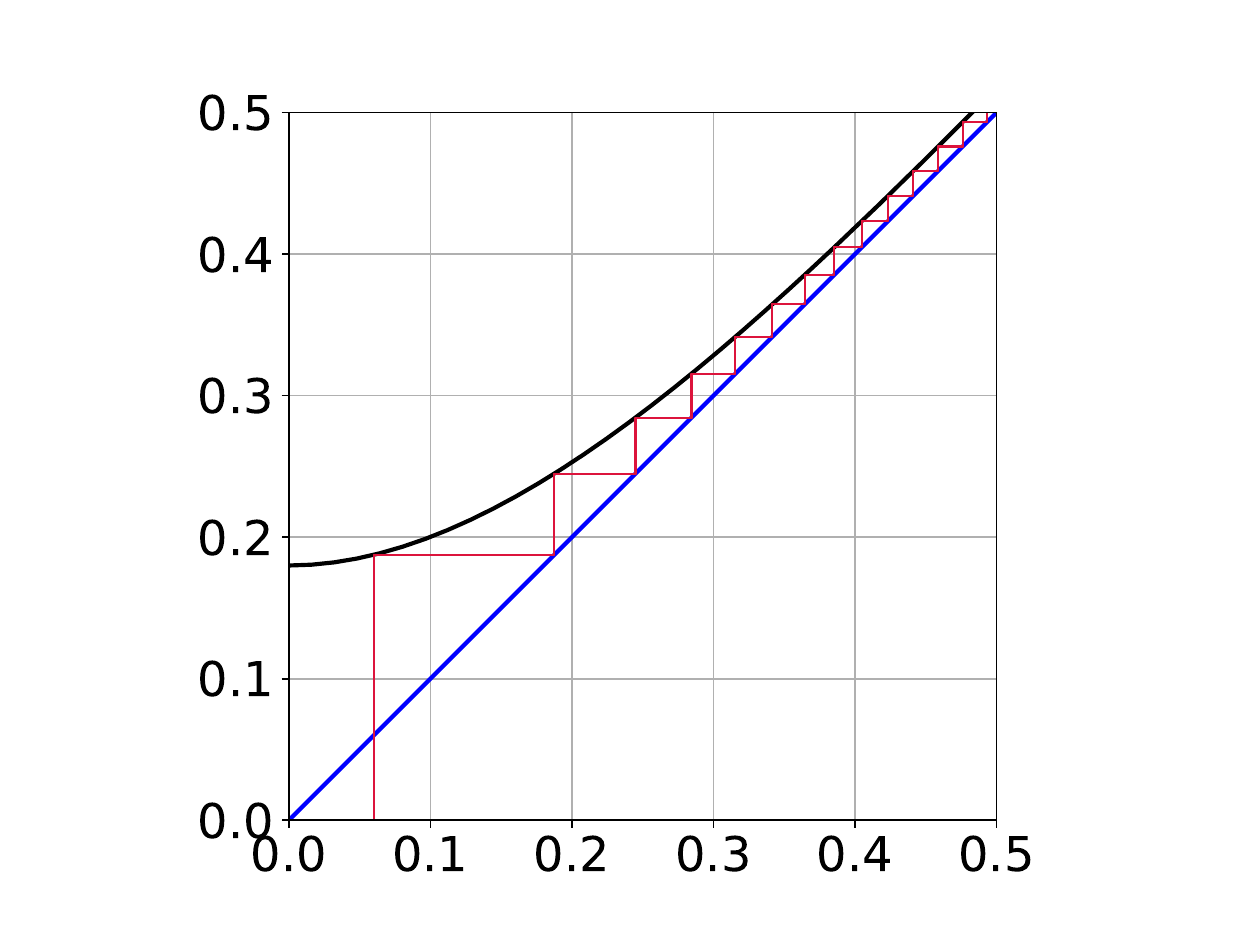}}
\caption{Fold bifurcation in the 1D Chialvo model with $k$ as 
a bifurcation parameter ($r=0.08$). Left: $k = 0.151$. Center: 
$k = 0.1627$ (numerical bifurcation value). Right: $k = 0.18$.}
\label{fig:kfoldbif}
\end{figure}

\begin{proof}
Observe that for $r> 2-\sqrt{2}-\ln\big(2\sqrt{2}-2\big)$
the multiplier equation 
\[
(2x-x^2)\exp{(r-x)}=1
\]
has a unique solution $x(r)$ in the interval $(0,2-\sqrt{2})$.
Since the fixed point and the multiplier conditions read 
as follows:
\begin{align*}
       x^2\exp{(r-x)}+k&=x,\\ 
       (2x-x^2)\exp{(r-x)}&=1,
\end{align*}
one obtains that $x(r)$ is the fixed point if and only if
$k=x-\frac{x(r)}{2-x(r)}$. It remains to check 
the fold bifurcation conditions from Theorem \ref{bif_theorem}. 
But both conditions (A.1) and (A.2) are
obviously satisfied under our assumptions,
which completes the proof.
\end{proof}

Figure~\ref{fig:kfoldbif} illustrates an example of 
the fold bifurcation with respect to $k$. 
Namely, for $r=0.8$, the fixed point
$x(r)\approx 0.4695$ undergoes a fold bifurcation, 
with bifurcation parameter $k^*=0.1627$. 

\subsection{Implications for the 2D Chialvo 
model and neural activity}\label{sec:implications}

We notify that often the knowledge about the reduced subsystem
\eqref{eq:Chialvo1DIM} gives insight into the dynamics of 
the full model in a properly chosen parameters range.
Figures~\ref{fig:BurstingPeriod4} and~\ref{fig:KillingBursting}
present the behaviour of the voltage $x$ and 
the recovery variable $y$ over simulation time of 
$n=80$ iterates for $a=0.876$ and $c=0.28$ and initial conditions
$(x_\textrm{init},y_\textrm{init})=(5,3)$. 
The trace of the voltage
$x$ is marked in red and of the recovery variable $y$ in blue. 
For $b=0$ (Figure~\ref{fig:BurstingPeriod4}) 
we observe stabilization
of $y$ at $r_1:=y\approx 2.258$ and bursting of period $4$ for 
the voltage variable. Small increase in the parameter $b$ to 
$b=0.02$ (Figure~\ref{fig:KillingBursting}) yielded, 
for the same initial conditions, stabilization of the recovery
variable at $r_2:=y\approx 1.8$ and resting of the voltage at
$x\approx 2.84$ after initial damped oscillations. 
The corresponding reduced models given by iterates of 
$f_{r_1,k=0}$ and $f_{r_2,k=0}$, respectively, display 
(except for the superattracting fixed point $x_0=0$) 
attracting period $4$ periodic orbit and non-zero stable 
fixed point $x_f\approx 2.84$, correspondingly. 
Both cobwebs present the orbit starting at $x=5$. 
In Figure \ref{fig:KillingBursting} the first $80$ iterates are
depicted, whereas Figure \ref{fig:BurstingPeriod4} shows only 
iterates between $60$ and $80$ for clarity.

\begin{figure}[!ht]
\centering{
    \includegraphics[scale=0.21, trim= 5mm 10mm 0mm 10mm]{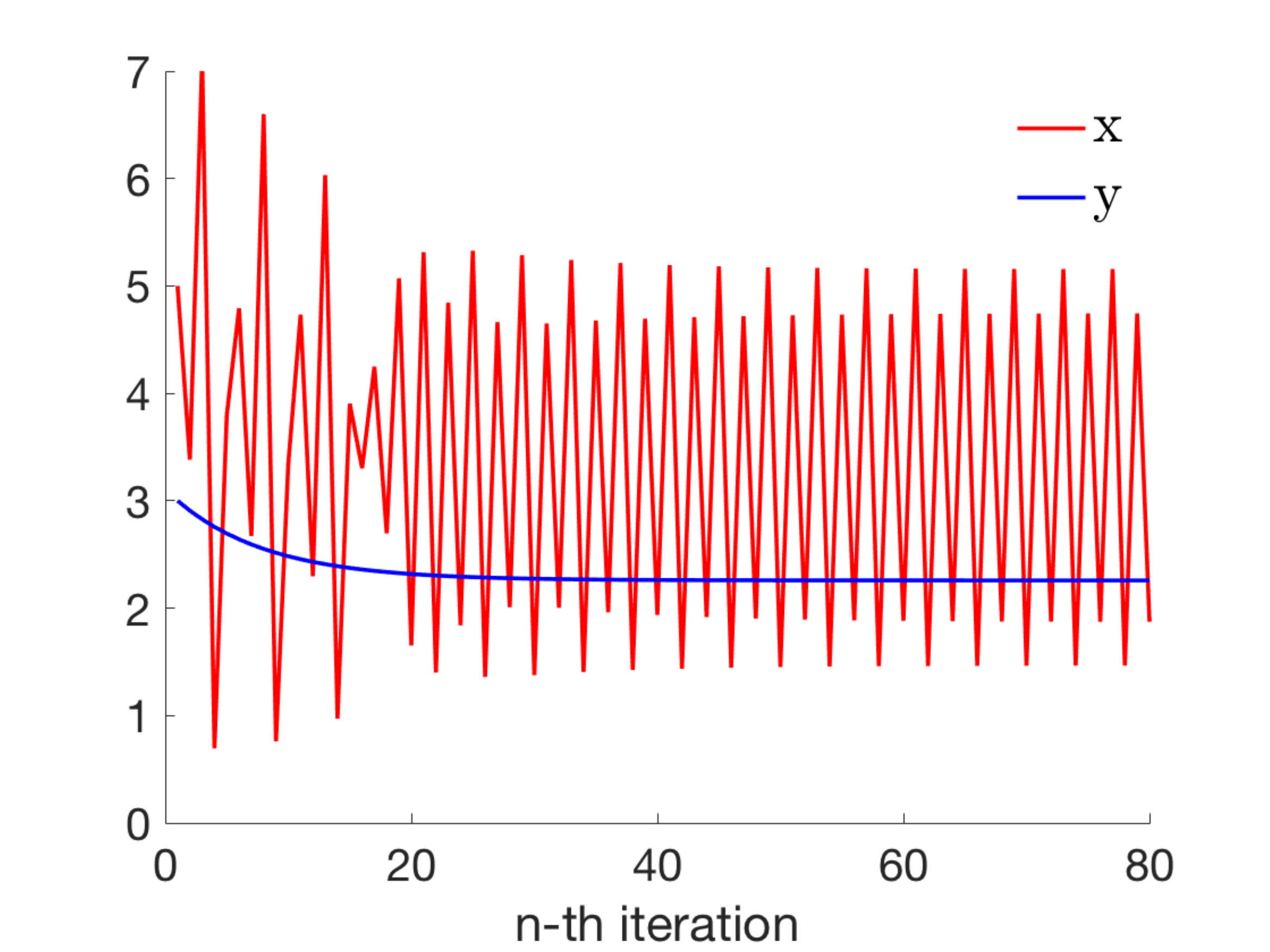}
    \includegraphics[scale=0.21, trim= 12mm 0mm 20mm 10mm]{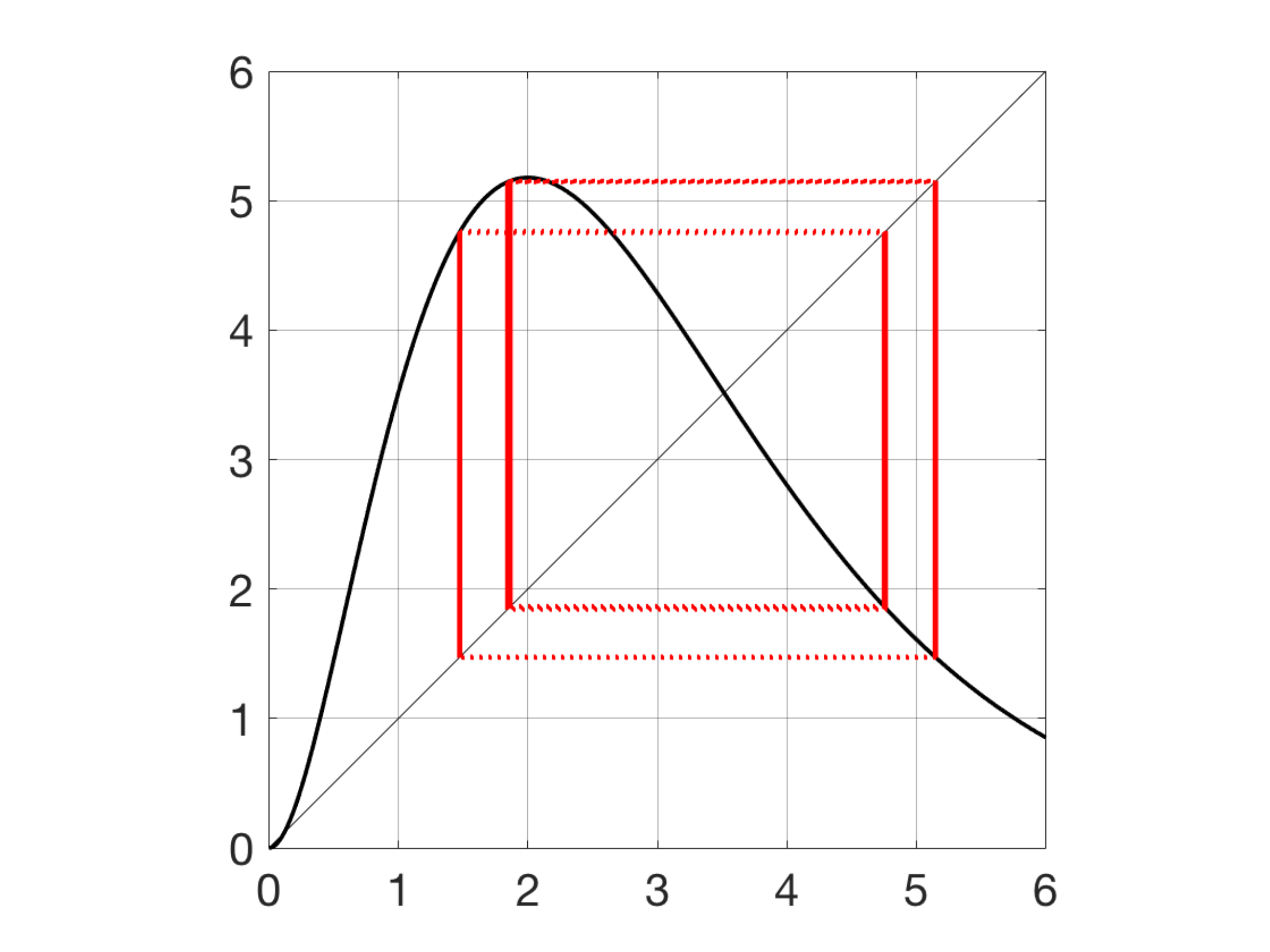}}
    \caption{Bursting of period $4$ in the 2D Chialvo model (left) and the corresponding period $4$ orbit in the reduced model (right). See text for details.}
    \label{fig:BurstingPeriod4}
\end{figure}

\begin{figure}[!ht]\vspace{-5mm}
\centering{
    \includegraphics[scale=0.21, trim= 5mm 10mm 0mm 10mm]{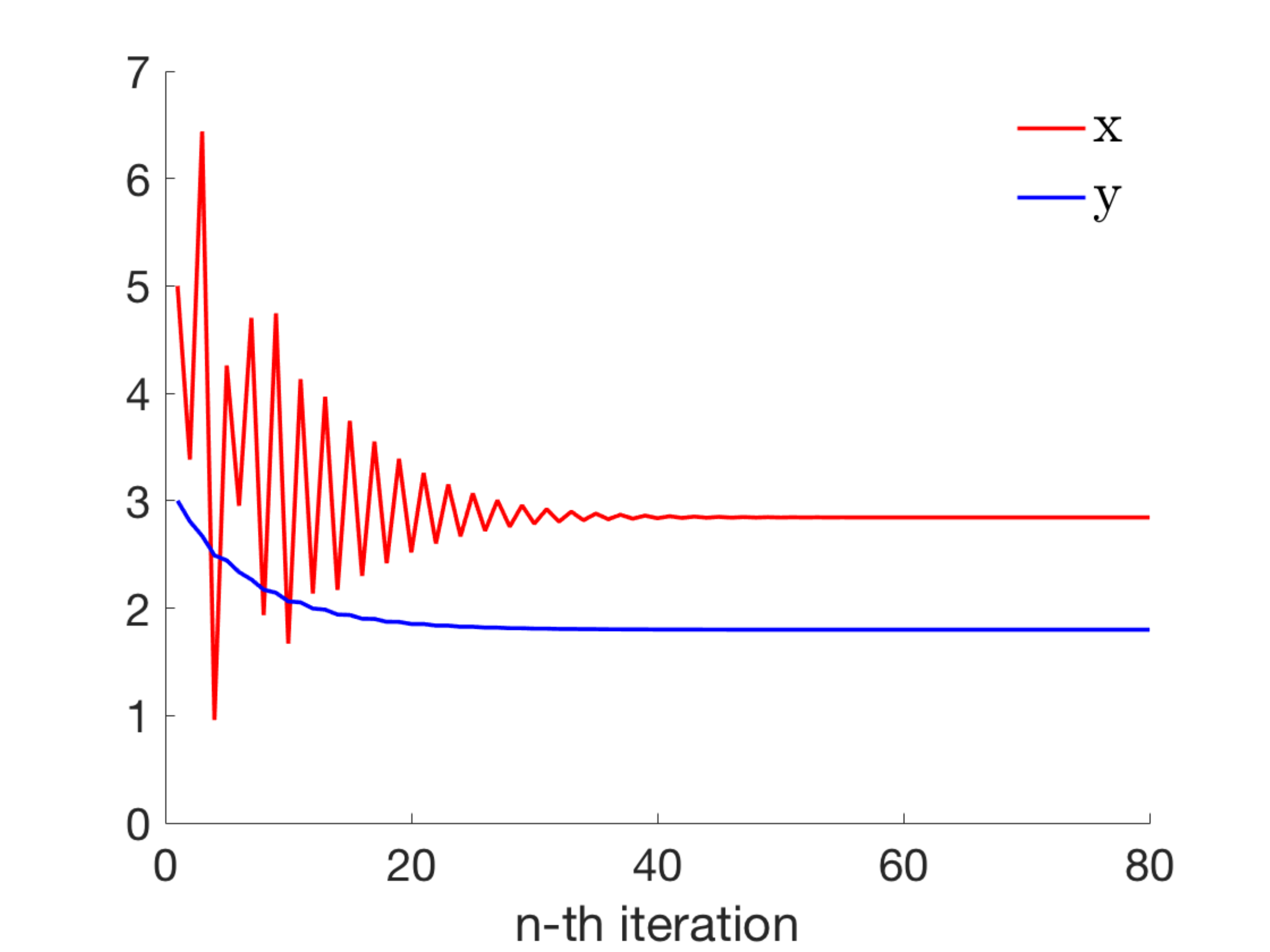}
    \includegraphics[scale=0.21, trim= 12mm 0mm 20mm 10mm]{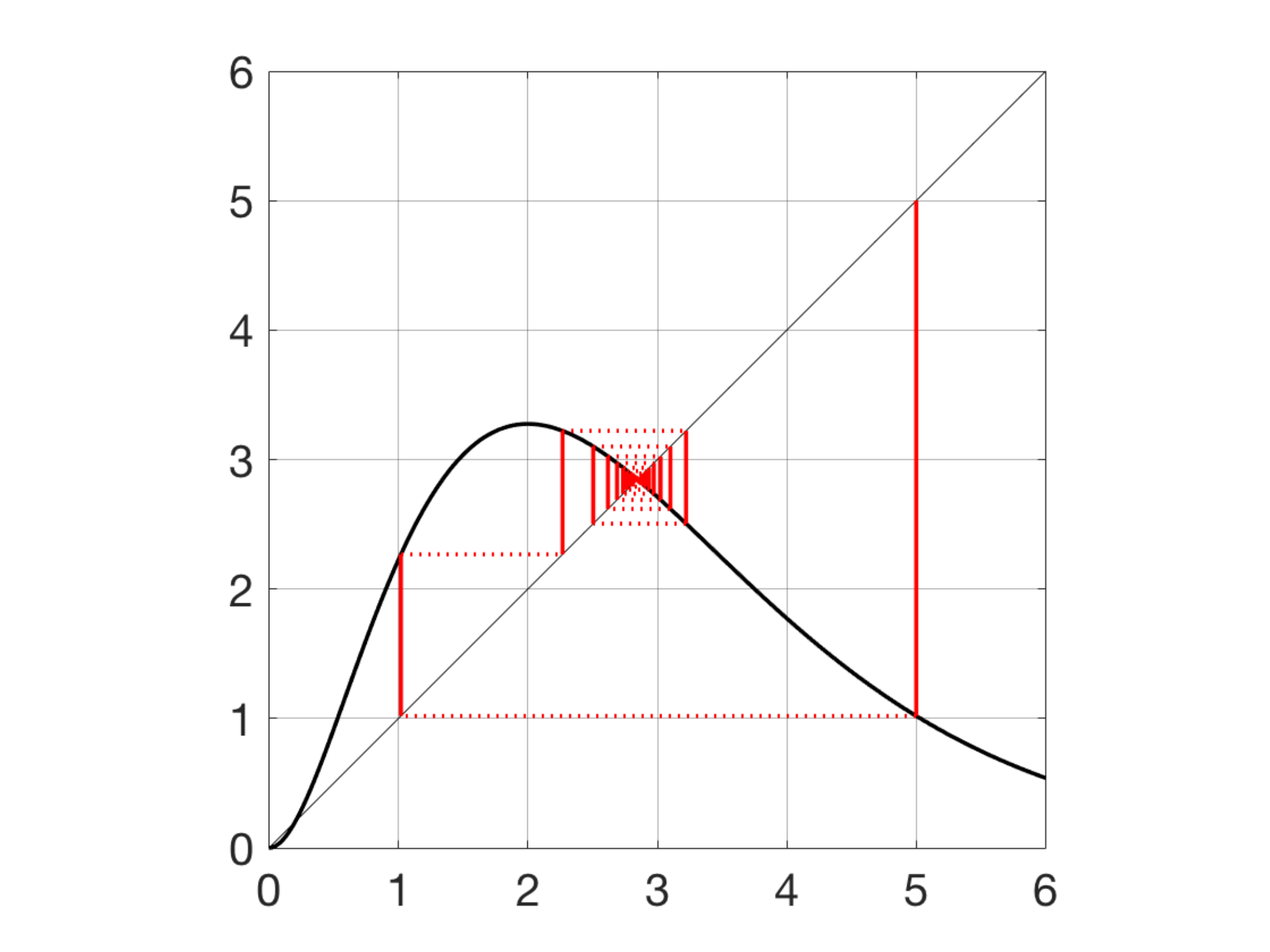}}
    \caption{Resting of the voltage in the 2D Chialvo model (left) and the corresponding stable fixed point in the reduced model (right). See text for details.}
    \label{fig:KillingBursting}
\end{figure}

Our bifurcation results (Thms~\ref{thm:flip}, \ref{thm:fold} 
and~\ref{thm:foldk}) confirm that a single neuron in 
the reduced Chialvo model is able to exhibit diverse 
activity regimes and periodic state transitions.
Since quiescence plays the fundamental role in the analysis
of neuronal activity, it seems to be crucial to describe
the basic possible scenarios that fixed states
of the model can undergo. Recall that the bifurcation parameters 
$r$ (Thms~\ref{thm:flip} and~\ref{thm:fold})
and $k$ (Thm~\ref{thm:foldk}) coincide, respectively,
with the recovery variable of the 2D Chialvo system
and the additive voltage perturbation.
Note that, in general, the~flip bifurcation, 
which usually starts the period-doubling route to chaos
(compare with Sec.~\ref{sec:periodic}), corresponds
to transition from  quiescence to regular tonic spiking,
while the fold bifurcation corresponds to appearance 
of new resting states.

However, our main motivation behind
studying the 1D Chialvo model is not to show that it well 
approximates the dynamics of the full 2D model but 
to point out that the 1D system itself represents very rich dynamics.  
Indeed, within some parameters range the behaviour 
of the 1D map might be even more complicated. 
For example, while experimenting numerically with 
the Chialvo model we have found that with the parameters set 
as $a=0.866$, $b=0.05$, $c=0.48$ and $k=0$ 
the trajectory starting from the initial conditions
$(x_\textrm{init},y_\textrm{init})=(2.8,1.5)$ 
converges numerically to the period~$4$ cycle
$\{(x_0,y_0), (x_1,y_1), (x_2,y_2), (x_3,y_3)\}$ with

\begin{center}
\begin{tabular}{ c|c|c|c|c|c|c|c } 
 $x_0$ & $y_0$ & $x_1$ & $y_1$ & 
 $x_2$ & $y_2$ & $x_3$ & $y_3$ \\ 
 \hline
 $1.4646$ & $2.2539$ & $4.7230$ & $2.3586$ & 
 $2.0970$ & $2.2864$ & $5.3144$ & $2.3552$
\end{tabular}
\end{center}

On the other hand, the corresponding 1D model 
(with $k=0$ and $x_\textrm{init}=2.8$) showed period~4 
periodic attractor for $r\in \{y_0,y_2\}$, 
period~12 periodic attractor for $r=y_3$ and 
chaotic interval attractor for $r=y_1$, 
which agrees with our numerically obtained
bifurcation diagram for the~Chialvo map $f_r$ with $k=0$ 
(not shown here, but similar to that 
from Figure~\ref{fig:bif_diagram}).
In Figure \ref{fig:NEW_Review} we present,
starting from top left in clockwise direction, 
the time-series of $x$- and $y$-variables for 
the~2D model and the corresponding trajectories of 
the reduced model (with $x_\textrm{init}=2.8$ 
after removing a transient of 970 iterates) 
with $r=y_0$, $r=y_1$ and $r=y_3$, respectively. 
This behaviour exemplifies that sometimes 
the reduced model might, counter-intuitively, 
show more complicated dynamics than some of its 
2D extensions. Therefore understanding of 
the 1D model bifurcation structure 
seems to be a relevant standalone topic. 
\begin{figure}[!htb]
\centering{
    \includegraphics[scale=0.2, 
    trim= 5mm 10mm 0mm 10mm]{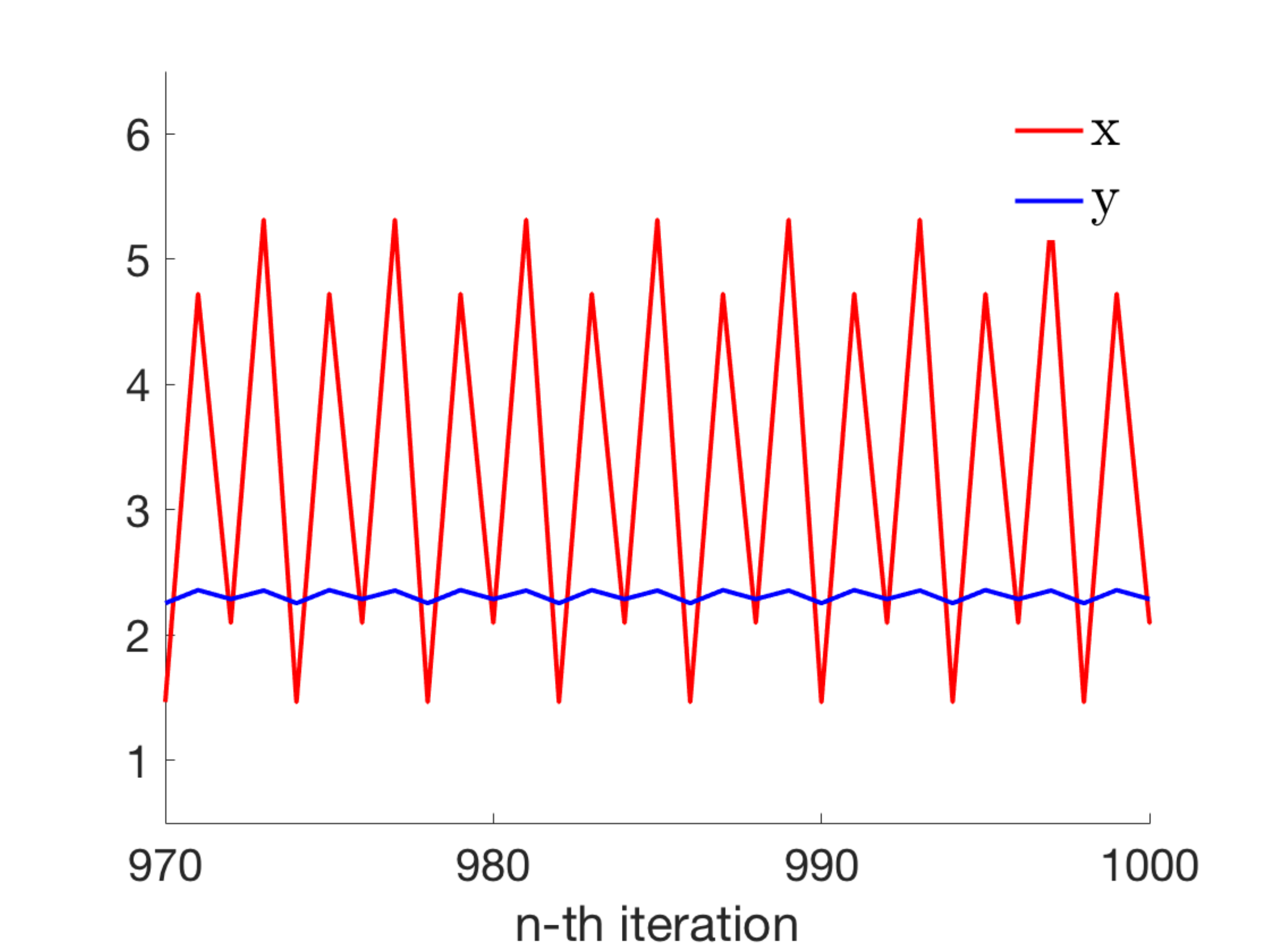}
    \includegraphics[scale=0.2, 
    trim= 13mm 0mm 20mm 10mm]{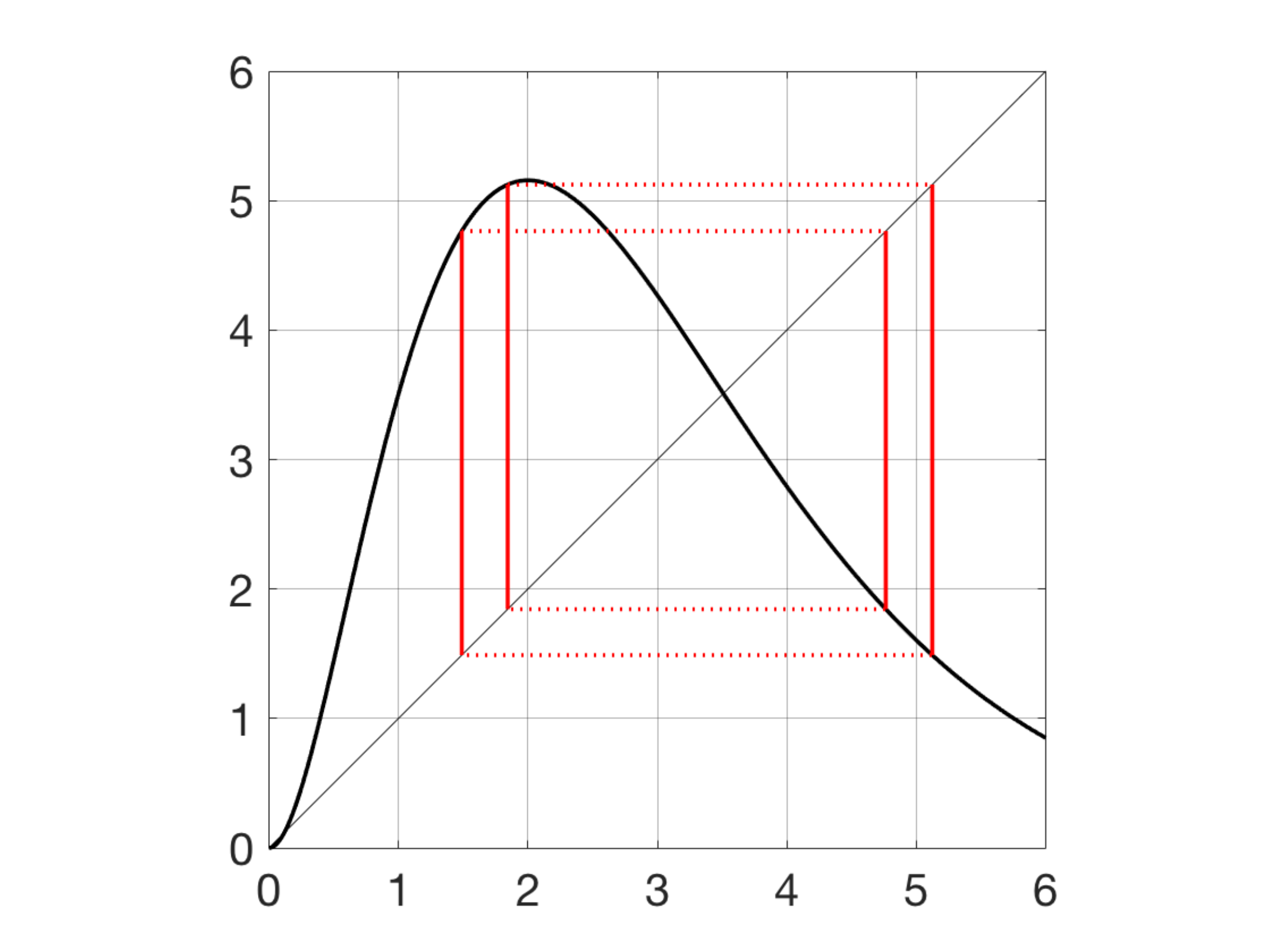}
    \vskip 0.4cm
    \includegraphics[scale=0.2, 
		trim= 12mm 0mm 20mm 10mm]{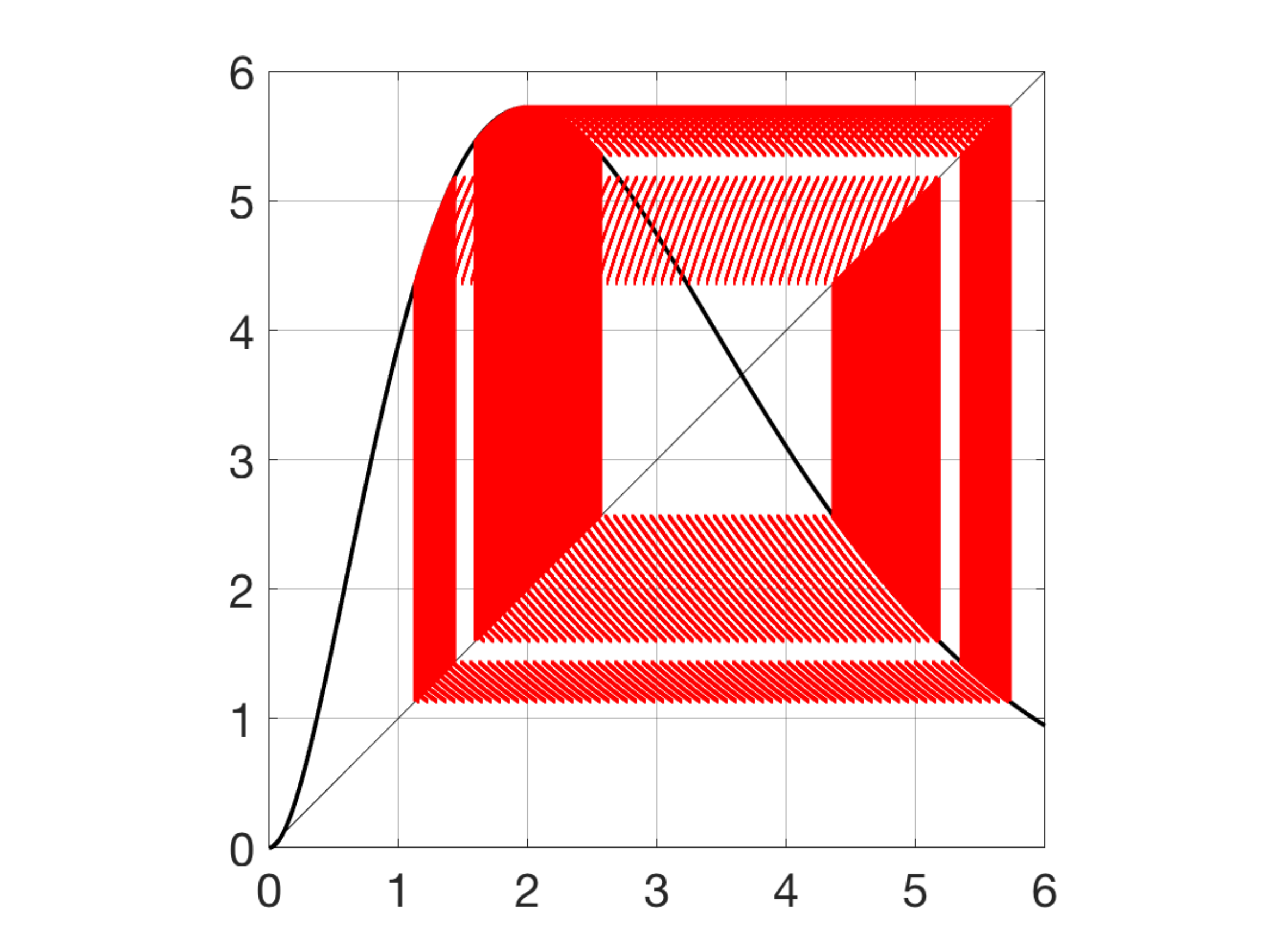}\hskip 0.6cm
    \includegraphics[scale=0.2, 
		trim= 12mm 0mm 20mm 10mm]{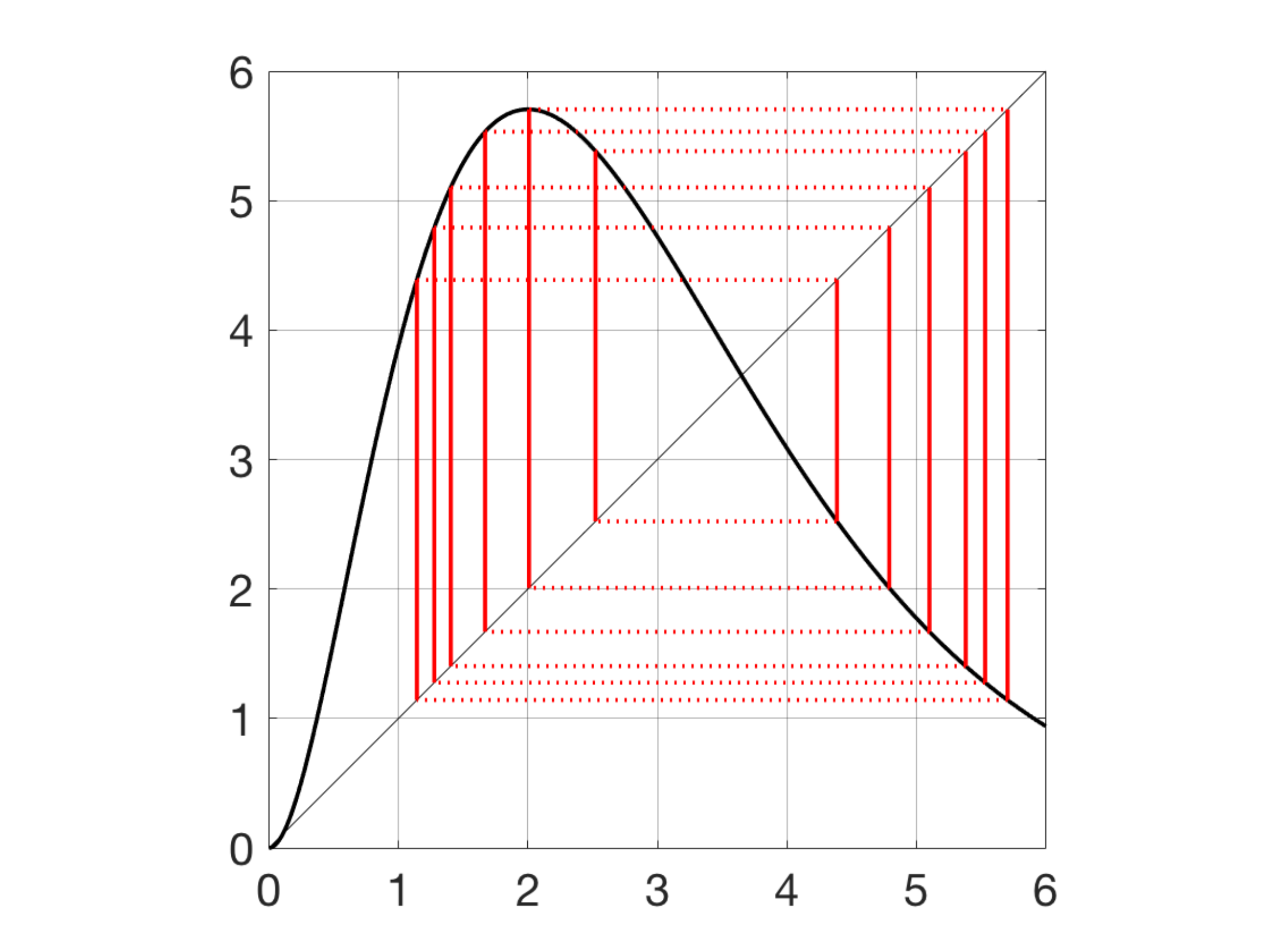}}
    \caption{Regular period 4 bursting in the 2D model 
		and the behaviour of the reduced model 
		in the corresponding range of $r$ parameter values. 
		See main text for details.}\vspace{-3mm}
    \label{fig:NEW_Review}
\end{figure}

We close this section with Figure~\ref{fig:bif_diagram} 
which shows one-parameter bifurcation
diagrams for the 1D Chialvo model, obtained in Python.
One can observe that the $r$ dependence
for the Chialvo family is similar to the parameter-dependence 
in the logistic family, whilst the $k$ dependence is similar 
to that in the Gauss family 
(with respect to the additive parameter).
\begin{figure}[!h]
\centering{
    \includegraphics[scale=0.185, 
    trim= 10mm 10mm 0mm 0mm]{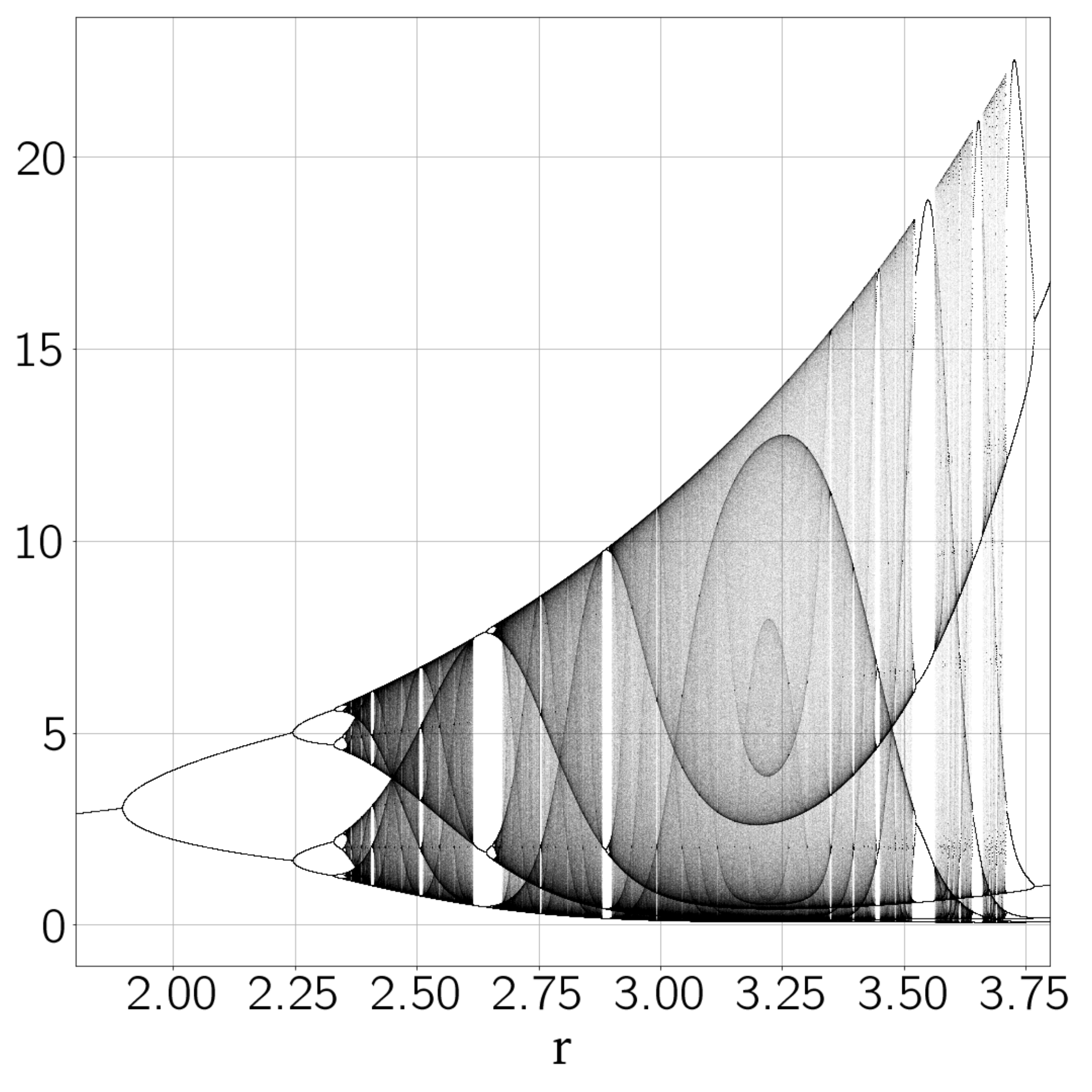}
    \hspace{9mm}
    \includegraphics[scale=0.185, 
    trim= 0mm 10mm 0mm 0mm]{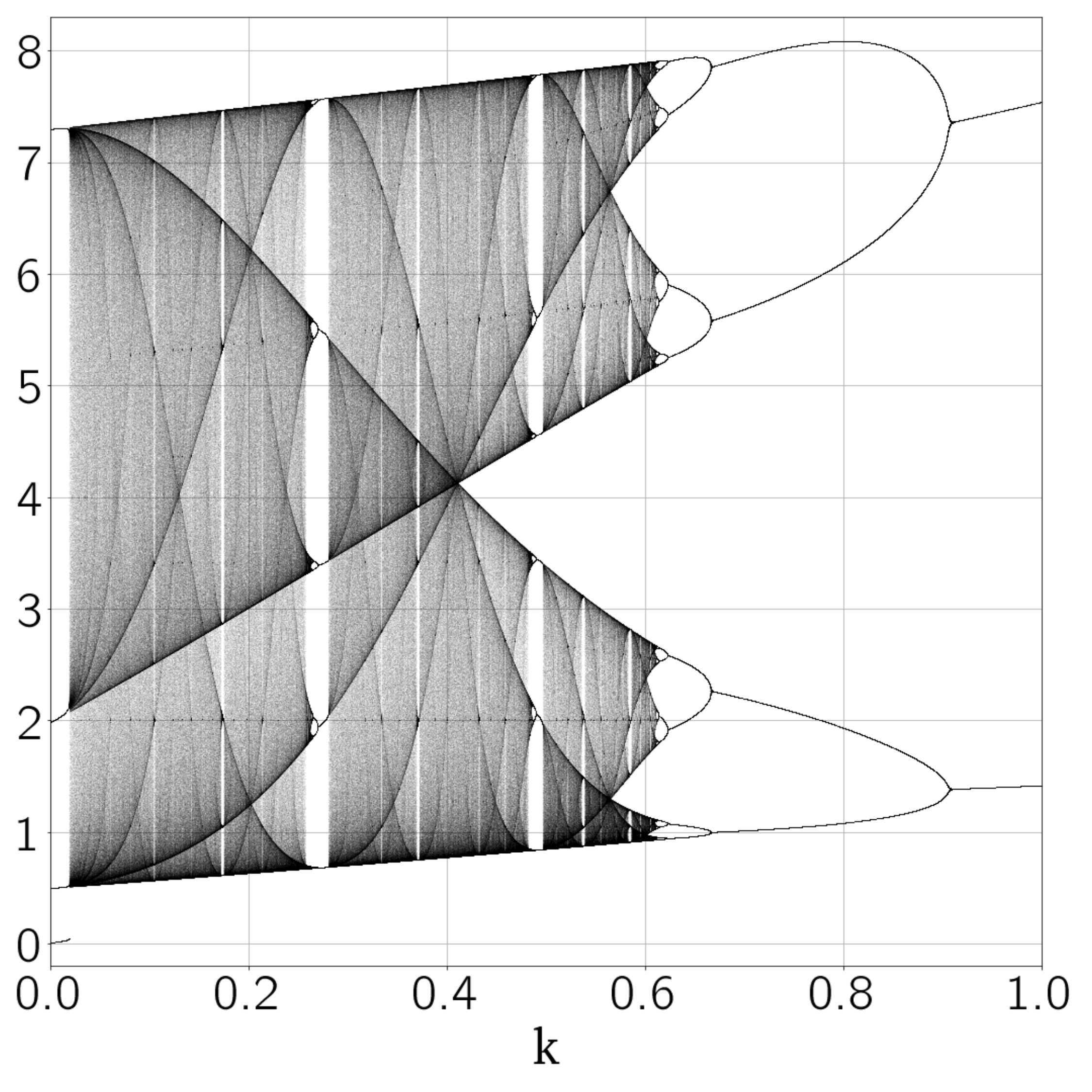}}
    \caption{Bifurcation diagrams for the Chialvo family: 
    versus $r$ with fixed
    $k=0.05$ (left) and versus $k$ with fixed $r=2.6$ (right).}
    \label{fig:bif_diagram}
\end{figure}

\section{Dynamical core of the Chialvo map}\label{sec:dyncore}
Recall that our definition of unimodality requires 
that either the left endpoint  of the interval domain $I$ 
of the map $f$ is a fixed point with the right endpoint 
of the interval as its other preimage, 
or that $I=\left[f^2(c),f(c)\right]$, where $c$ is 
the critical point (maximum) of $f$.
In this case we say that the interval $I$ is 
the \emph{dynamical core} of $f$, i.e., the invariant 
interval where all nontrivial dynamics concentrates.
An example of the dynamical core for some unimodal map
is showed in Fig.~\ref{fig:dyncore}.
\begin{figure}[!ht]
\centering{
\includegraphics[scale=0.33, 
trim = 40mm 5mm 10mm 7mm]{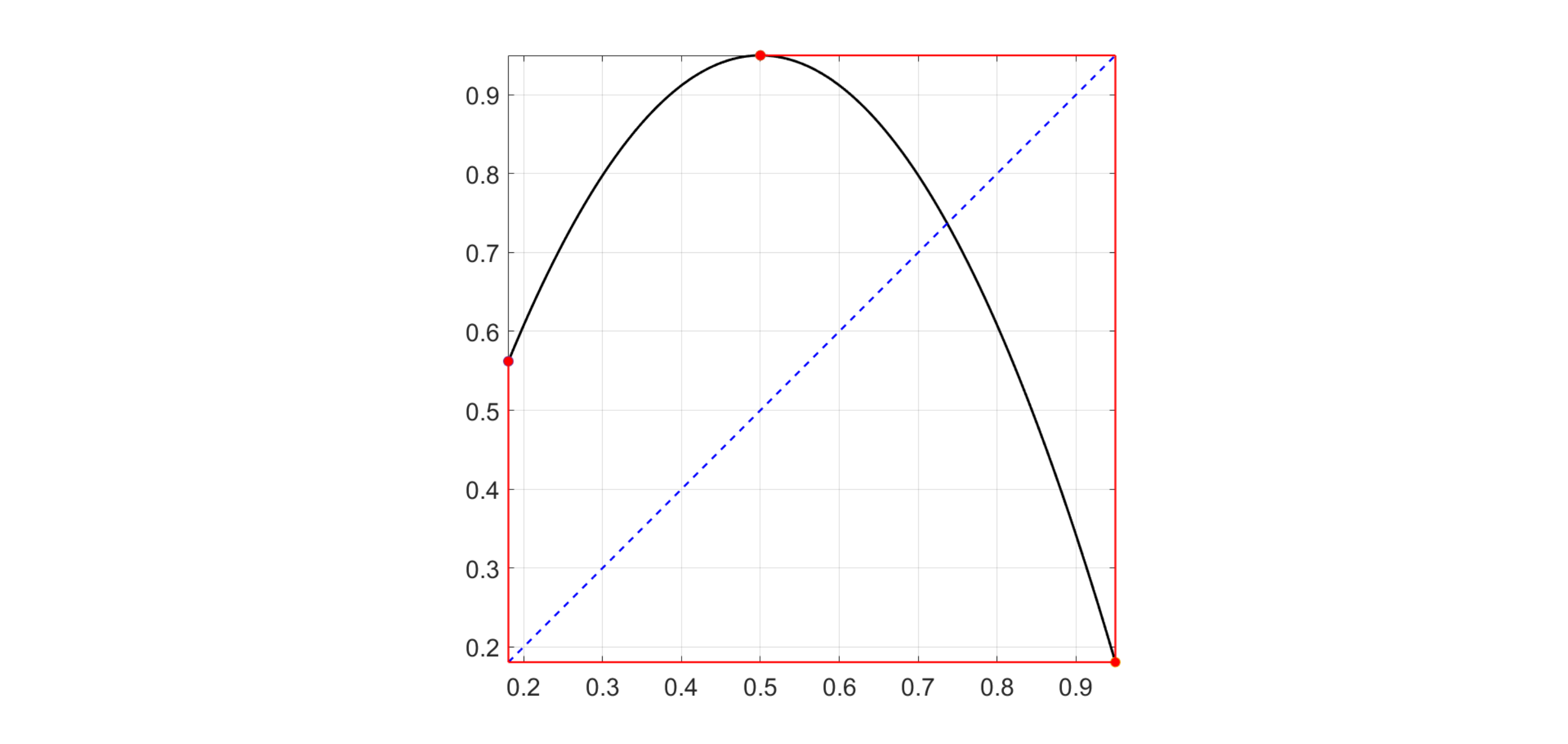}}
\caption{Interval $[f^2(c),f(c)]$ as a dynamical core 
($c=1/2$ is a critical point)}
\label{fig:dyncore}
\end{figure}

Below we discuss the form of the dynamical core $I_{r,k}$ for
$f_{r,k}(x)=x^2\exp(r-x)+k$ with respect to parameters
$k$ and $r$ in order to justify that the map $f_{r,k}$
describing the 1D Chialvo model is S-unimodal on 
an appropriate subinterval of $[0,+\infty)$.

\subsection{Case $k=0$}
Let us firstly consider the case $k=0$ and the corresponding 
map $f_{r,0}(x)=x^2\exp{(r-x)}$. 
The map $f_{r,0}$ has exactly
\begin{itemize}
\item one fixed point $x_0=0$, 
\item two fixed points: $x_0=0$ and $x_1$ satisfying $0<x_1<2$, 
\item three fixed points: $x_0$, $x_1$ and $x_2$ satisfying $x_0=0$, $0<x_1<x_2$.
\end{itemize}
In  the  first  two  instances,  the dynamics is trivial. 
Namely, every initial condition is attracted either by 
the superattracting fixed point $x_0$ or 
by the semistable fixed point $x_1$ (if exists). 
Configuration with three fixed points splits into two cases,
the second with two subcases:
\begin{enumerate}
\item $x_2\leq 2$,
\item $x_2>2$: subcase 2.1.\ $x_2$ is stable, subcase 2.2.\ $x_2$ 
is unstable.
\end{enumerate}
Of course, there is also a special case when $x_2=3$ undergoes flip
bifurcation (when $2<x_2<3$, then $x_2$ is stable and if $x_2>3$ 
then it is unstable as follows from simple calculations in the part
concerning flip bifurcation). 
The cases 1. and 2.1. are again trivial. 
Therefore for our further developments we shall be interested 
in case 2.2. Note that under this condition $x_2$ satisfies
$x_2\exp(-x_2)=\exp(-r)$ and for the derivative we have 
$\vert f^{\prime}_{r,0}(x_2)\vert =\vert 2-x_2\vert>1$. 
Therefore with the increase of $r$ the fixed point $x_2>3$ moves
further to the right and remains unstable.  
Simultaneously the fixed point $x_1$ remains to the left of $c=2$ 
and is unstable. If, moreover, $x_1<f^2_{r,0}(2)$, then 
the interval $[f^2_{r,0}(2),f_{r,k}(2)]$ can be treated as 
the dynamical core of the map $f_{r,0}$, because every initial
condition $x\in (x_1,\infty)$ is after a few iterates mapped into
$[f^2_{r,0}(2),f_{r,k}(2)]$. On the other hand, all the points in
$(0,x_1)$ are monotonically attracted to the resting state $x_0=0$. 
We are mainly interested in this case, i.e., when 
$x_1\notin [f^2_{r,0}(2), f_{r,0}(2)]$, 
which holds at least for some interval of $r$-values 
(see Lemma \ref{lem:DynCore} and Remark \ref{Remark32}).

\subsection{Case $k>0$}
Now let us consider the case $k>0$. Note that when $k\geq 2$, 
the dynamics of $f_{r,k}$ is actually trivial: 
indeed for $k\geq 2$ there is exactly one fixed point $x_0$ 
which is located on the right (decreasing) part of the graph 
(i.e. $x_0>k\geq 2$). As also $f_{r,k}^2(2)>k\geq 2$, 
the interval $[2,f_{r,k}(2)]$ is an invariant interval 
and every initial condition $x\in\mathbb{R}$ is after 
at most one iterate mapped into this interval. 
On the other hand, as $f_{r,k}$ is strictly decreasing 
on this interval, every initial condition
$x\in\mathbb{R}\setminus\{x_0\}$ is actually attracted to $x_0$ 
(if $x_0$ is attracting) or to period two periodic orbit 
(if $x_0$ is not attracting). Obviously, in this case 
the dynamical core can be set as $[f^2(2), f(2)]$. 

Therefore assume that $0<k<2$. Then (especially for small values 
of $r$) one possibility is that $f_{r,k}$ has only one fixed point
$x_0$, which is located on the left branch: $x_0\leq 2$. Note that
in this case $x_0$ is globally attracting and thus this is 
the trivial case. If $x_0$ is a unique fixed point but $x_0>2$, 
then, depending on $r$, the fixed point $x_0$ can be either stable 
or unstable,  but, as $f^2_{r,k}(2)>k=f_{r,k}(0)$, we can set 
the interval $[f^2_{r,k}(2),f_{r,k}(2)]$ as the dynamical core, 
which contains exactly one fixed point $x_0>2$. 
Lemma \ref{lem:DynCore} asserts that this always holds 
for $r$ large enough. 

 The situation when $f_{r,k}$ has exactly two fixed points is 
 very rare (it holds e.g. when there is a fixed point $x_0>0$ 
 and the other fixed point $x_1=x_2$ is just undergoing 
 the fold bifurcation). Thus the effective complementary case 
 for unique fixed point $x_0$ is the situation when there 
 are exactly three fixed points $0<x_0<x_1<x_2$ and necessarily 
 then $x_1\leq 2$. If also $x_2\leq 2$, then every 
 initial condition (except $x_1$) is attracted either to $x_0$ 
 or to $x_2$. It follows that for $0<k<2$ the non-trivial 
 generic case with three fixed points is the situation when
 $0<x_0<x_1<2<x_2$ (with $x_2$ that can be stable or unstable).  
 If $f_{r,k}^2(2)<x_1$, then as a dynamical core 
 we set $[x_0,y_0]$, where $y_0\in f_{r,k}^{-1}(x_0)$ 
 and $y_0\neq x_0$ (note that such a point $y_0$ exists 
 since $f_{r,k}(x) \rightarrow k<x_0$ as $x\to \infty$). 
 Then $f_{r,k}\colon [x_0,y_0] \to [x_0,y_0]$ matches 
 the definition of the unimodal map. On the other hand, 
 if $f_{r,k}^2(2)\geq x_1$, we can assume that the dynamical core 
 is $[f_{r,k}^2(2), f_{r,k}(2)]$ (the map $f_{r,k}$ 
 restricted to this interval is unimodal).

\subsection{Auxiliary results on dynamical core} 
This subsection contains several results concerning the
dynamical core of the 1D Chialvo map, which will be of our 
interest in the next section 
(for example in the proof of Theorem \ref{thm:ai}). 
The technical proofs of Lemmas~\ref{lem:DynCore} and \ref{lem:UnstableFP} 
are postponed to Appendix~\ref{appendix:core}.

\begin{lemma}\label{lem:DynCore} For any $0\leq k<2$  
there exists $r^*=r^*(k)$ such that for all $r\geq r^*(k)$ we have
\begin{equation}\label{LemmaDynCoreEq}
f_{r,k}^2(c)<c<f_{r,k}(c),
\end{equation}
where $f_{r,k}(x)=x^2\exp(r-x)+k$ is the 1D Chialvo map and $c=2$ 
is its critical point. Moreover, for $k\neq 0$ the value 
of $r^*(k)$ can be chosen such that, in addition, 
for $r\geq r^*(k)$ the interval $[f_{r,k}^2(c),f_{r,k}(c)]$  
contains exactly one fixed point $x_f$.  
\end{lemma}

\begin{remark}\label{Remark32} For $k=0$, demanding that
\eqref{LemmaDynCoreEq} is satisfied and simultaneously 
there are no fixed points in $[f_{r,0}^2(c),f_{r,0}(c)]$ 
to the left of $c$ requires some more specific range 
of $r$ values. Indeed, for $k=0$ the map can have 
three fixed points $x_0=0$, $0<x_1<2$ and $x_2>2$ with $x_1$ 
that actually can fall into the interval 
$[f_{r,0}^2(c),f_{r,0}(c)]$. We have verified numerically 
that for any $r\in [2.1, 2.97]$ the inequality 
\eqref{LemmaDynCoreEq} is satisfied with $f^2_{r,0}(c)>x_1$. 
In turn, for $r=2$ we have $f_{r,0}^2(c)\approx 2.1654>c=2$ 
(violating \eqref{LemmaDynCoreEq}) and for 
$r=2.98$ $f_{r,0}^2(c)<x_1$ as $f_{r,0}^2(c)\approx 0.0526$ 
and $x_1 \approx 0.0535$.  
\end{remark}

\begin{remark}
Consequently, the interval $[f_{r,k}^2(c),f_{r,k}(c)]$ 
(in the proper parameter regime, i.e., $r\geq r^*(k)$ 
for $k\in (0,2)$ and $r\in [2.1,2.97]$ for $k=0$) can 
be treated as dynamical core of $f_{r,k}$, 
which contains exactly one fixed point, namely $x_f>2$. 
\end{remark}
The next lemma assures that with $r$ large enough 
the fixed point $x_f$ mentioned above is unstable.

\begin{lemma}\label{lem:UnstableFP} For any $0\leq k<2$  
there exists $r^*(k)$ such that for all $r\geq r^*(k)$ 
the map $f_{r,k}$ has an unstable fixed point $x_f>2$. 
\end{lemma}

\begin{corollary}\label{afterLemma3_4}
For every $k\in [0,2)$ there exists an interval $[r_1,r_2]$ 
of positive $r$-values such that for any $r\in [r_1,r_2]$ 
the map $f_{r,k}$ has an attracting fixed point $x_f>2$.
\end{corollary}
The above corollary can be easily justified taking into account 
the proof of Lemma \ref{lem:UnstableFP} and the fact that for 
a S-unimodal map $f$ every fixed point $x_f$ with 
$\vert f^{\prime}(x_f)\vert \leq 1$ is attracting.

\section{Periodic versus chaotic behaviour}\label{sec:periodic}
The 1D Chialvo map \eqref{eq:Chialvo1DIM}  except for regular 
behaviour due to the existence of stable fixed points or 
periodic orbits can also exhibit chaotic behaviour. 
There are many different notions of chaos 
(see \cite{AulbachKieninger2001}), 
with the most prominent one probably due to Devaney
(\cite{devaney2003}). 
The common point of most of these definitions is some sort 
of sensitive dependence on initial conditions. 
In our approach we will consider both metric chaos associated 
with the existence of the absolutely continuous invariant 
probability measure (\emph{acip}), which, as we recall below, 
for S-unimodal maps can be identified with some strong sensitive
dependence on initial conditions, and topological chaos 
corresponding to the classical definitions of Li and Yorke, 
Block and Coppel, and Devaney.

\subsection{Periodic attractors}\label{subsec:attract}

Since the reduced Chialvo map is S-unimodal and 
its critical point is non-degenerate,
the following two theorems are immediate consequences of
Theorems \ref{thm:singer} and~\ref{thm:attractorgen}.

\begin{theorem}\label{mainChaos1} The 1D Chialvo map
\eqref{eq:Chialvo1DIM} can have at most one periodic attractor 
and this attractor attracts the critical point $c=2$. 
\end{theorem}

Let $\Gamma$ be a periodic orbit with period $n$
and let $\lambda(\Gamma)=(f^n)^{\prime}(x)$ 
for $x\in \Gamma$. It is clear that
$\vert \lambda(\Gamma)\vert \in [0,1)$ implies 
that $\Gamma$ is attracting. However, 
since $f$ is S-unimodal, 
by Theorem~\ref{thm:singer},
also neutral periodic orbits 
(i.e. such that $\vert \lambda(\Gamma)\vert=1$) 
are attracting.

\begin{theorem}\label{thm:types}
The 1D Chialvo map has a unique metric attractor
$\mathcal{A}$ which is a limit set $\omega(x)$  for
almost all initial conditions $x$.  
This attractor is either an attracting periodic orbit, 
a Cantor set of measure zero or 
a finite union of intervals with a dense orbit 
(so called interval attractor).
\end{theorem}

We say that two unimodal maps $f$ and $g$ are 
\emph{combinatorially equivalent} 
if there exists an order-preserving bijection 
$h\colon \cup_{n\ge0}f^n(c)\to\cup_{n\ge0}\,g^n(c')$
such that $h(f^n(c)) = g^n(c')$ for all $n\ge0$, where $c$ 
and $c'$ are critical points of
$f$ and $g$. In the other words, $f$ and $g$ are 
combinatorially equivalent if 
the order of their forward critical orbit is the same. 

Let us present the following general
mathematical observation, which might be known to specialists 
in one-dimensional dynamics, but probably lacks clear 
references as follows, for example, from Remark~\ref{rem:schreiber} below. 
\begin{lemma}\label{lem:comb}
Assume that $f$ and $g$ are S-unimodal maps. If $f$ has 
a period two attracting orbit and  
$g$ has a period three attracting orbit 
then $f$ and $g$ 
are not combinatorially equivalent.
\end{lemma}

\begin{proof}
Let $\mathcal{A}$ denote a periodic attractor and 
$c$ a critical point. Since $\omega(c)=\mathcal{A}$,
starting from some iteration the order of the critical point
orbit is the same as the order on the periodic attractor.
Obviously, the orders of a two and a three periodic orbit
are different, which is our claim.
\end{proof}

\begin{remark}\label{rem:schreiber}
In \cite{schreiber2007} it is stated that a map with 
an attracting equilibrium and
a map with a period two attracting orbit are 
not combinatorially equivalent,
which in general is not true. 
\end{remark}

Let us also formulate a conjecture which generalizes Lemma~\ref{lem:comb}, and which we believe is true.

\begin{conj}
Let $f$ and $g$ be S-unimodal maps, $m>k$ and $m\neq2k$.
If $f$ has an attracting periodic orbit with 
a prime period $m$ and $g$ has an attracting periodic orbit 
with a prime period $k$, then $f$ and $g$ 
are not combinatorially equivalent.
\end{conj}

The following useful result can be found 
in \cite[Th. C]{kozlovski2003}. 

\begin{theorem}[Kozlovski]\label{thm:koz}
Let $\Lambda\subset\mathbb{R}^k$ be open and 
$\{f_\lambda\}_{\lambda\in\Lambda}$ 
be an analytic family of unimodal maps with a
non-degenerate critical point and negative Schwarzian derivative. 
If there exist two maps in $\{f_\lambda\}$
that are not combinatorially equivalent, 
then there exists a dense subset 
$\Theta$ of $\Lambda$
such that  $f_\lambda$ has a stable attracting  
periodic orbit for every $\lambda\in\Theta$.
\end{theorem}

The next result describes the subset in parameter space 
corresponding to periodic attractors.

\begin{theorem}\label{thm:apo}
The set of parameters for which the 1D Chialvo map $f_{r,k}$ 
has a stable attracting  
periodic orbit forms a dense subset of
\begin{enumerate}
	\item the $r$ parameter space $\mathbb{R}$ when $k\in[0,2)$ is fixed,
	\item the $(r,k)$ parameter space $\mathbb{R}\times[0,2)$.
\end{enumerate}
\end{theorem}

\begin{proof}
Let $k\in[0,2)$ be fixed.
Clearly, for some $r$ the map $f_{r,k}$ has
a period two attracting orbit and
for some other $r$ it has a period three one.
Consequently, Lemma~\ref{lem:comb} guarantees that the family
$\{f_{r,k}\}$ contains two maps that are not combinatorially
equivalent. Alternatively, to obtain two such maps we
can use the kneading sequences from the proof of
Corollary~\ref{cor:pori}.
Since all other assumptions of Theorem~\ref{thm:koz} are obviously
satisfied, the assertion of the first statement of 
our theorem follows the Kozlovski Theorem. 
Finally, the second statement of our theorem 
follows from the first one.
\end{proof}

\subsection{Metric chaos}\label{subsec:metric}

The following theorem guarantees that 1D Chialvo maps are
strongly chaotic with positive frequency in parameter
space. As we have mentioned above strong chaos here
is associated with the existence of acip measure.
Below we present the rigorous proof in the case $k=0$
and conclusions of numerical analysis for the case $k\in(0,0.58)$.
Here and subsequently, $\delta_x$ denotes the Dirac measure at 
the point $x$ and $\supp(\mu)$ the support of a measure $\mu$.
Moreover, recall that a map $f$ is called 
a \emph{Misiurewicz map} if it has no periodic attractors 
and if critical orbits do not accumulate on critical points,
i.e. when
\begin{displaymath}
C_f \cap \omega(C_f) = \emptyset,
\end{displaymath}
where $C_f$ denotes the set of critical points of $f$.
For a parametrized family of maps $\{f_r\}$
the values of parameters for which $f_r$
are Misiurewicz maps are called \emph{Misiurewicz parameters}.
An example of such a
parameter in the logistic family
$Q(x)=\mu x(1-x)$ is $\mu=4$.

For fixed $k$, let $f_r(x):=x^2\exp{(r-x)}+k$ and let 
$I_r\subset [0,\infty)$ denote the dynamical core of $f_r$
(see Section~\ref{sec:dyncore} for details). 

\begin{theorem}\label{thm:ai} Let $k=0$ be fixed. 
There exists a positive Lebesgue measure set 
$\mathcal{R}$ such that for all $r\in\mathcal{R}$: 
\begin{itemize}
\item the map $f_r: I_r\to I_r$ admits an absolutely 
continuous invariant probability measure $\mu_r$ such that 
\begin{equation}\label{Chaos2_1}
\mu_r=\lim_{n\to\infty}\frac{1}{n}
\sum_{i=0}^{n-1}\delta_{f_r^i(x)}
\quad\text{for a.a.\ $x\in I_r$},
\end{equation} 
i.e., $\mu_r$ describes the asymptotic distribution 
of almost all orbits under $f_r$,
\item the unique metric attractor of the map $f_r$ is 
an interval attractor $\mathcal{A}_r$;
moreover, $\mathcal{A}_r=\supp(\mu_r)$ and $\mu_r$  
is equivalent to the Lebesgue measure on $\mathcal{A}_r$.
\end{itemize}
In consequence, for each $r\in \mathcal{R}$ 
the Lyapunov exponent of $f_r$ is positive, i.e.,
\begin{equation}\label{Chaos2_2}
\lim_{n\to\infty}\frac{1}{n}\log \vert \mathrm{D}f^n_r(x)\vert=\kappa_r>0 
\quad\text{for a.a.\ $x\in I_r$},
\end{equation}
where $\mathrm{D}$ denotes the derivative with respect to $x$.
\end{theorem}

\begin{proof}
The statements of the above theorem follow from some nowadays
classical results on S-unimodal maps, for example, 
from Theorem 18 and Corollary 19 in \cite{thunberg2001}, 
provided that the below-listed conditions C1--C6 are satisfied 
for some interval of parameters $r$: 
\begin{itemize}
  \item[C1.] $\{f_r\}$ is a one-parameter family of $C^2$ unimodal
  maps of an interval $I$ (in general the interval $I$ can vary 
  with the change of the parameter $r$ , i.e. $I=I_{r}$), 
  \item[C2.]  each $f_r$ has a nondegenerate critical point $c$,
  \item[C3.]  the map 
  $(x,r) \mapsto (f_r(x),D_xf_r(x),D_{xx}f_r(x))$ is $C^1$, 
  \item[C4.] each $f_r$ has a repelling fixed point on 
  the boundary of $I$, 
  \item[C5.]  there is a parameter value $r^*$ such that
  the critical orbit is mapped onto an unstable periodic cycle 
  $P^*$ in a finite number of steps (this implies
  that $f_{r^*}$ is a Misiurewicz map),
  \item[C6.]  the map moves with the parameter at $r^*$ in 
  the following technical sense
  \[\frac{d}{dr}(x(r)-f_r(c))\big\vert_{r=r^*}\neq 0,\]
  where $x(r)$ is a point of $I$ and $P(r)$ is an unstable
  periodic orbit such that
  \begin{itemize}
      \item $r\mapsto x(r)$ is differentiable,
      \item $x(r^*)=f_{r^*}(c)$,
      \item $P(r)$ moves continuously with $r$,
      \item $x(r)$ is mapped onto $P(r)$ in the same
      number of steps as $x(r^*)$ and onto the corresponding point
      of $P(r)$
  \end{itemize}
  (existence of such a point and orbit is guaranteed by the
  perturbation theory).
\end{itemize}
We are left with the task of checking all these conditions.
We will see that in fact only C5--C6 might pose a problem.
\vspace{1mm}
  
\noindent \textit{Ad} C1--C3. It is easy to see that in 
our situation the first three conditions do not require any
justification.\vspace{1mm}
  
\noindent \textit{Ad} C4. The condition C4 can also be formally
achieved by properly extending the map $f_r$ beyond its 
dynamical core $[f_r^2(c),f_r(c)]$ which does not effect 
its dynamics as all the points in $I$ will after a few iterates 
be mapped into the dynamical core. Recall that by Lemma
\ref{lem:DynCore} for each $k=0$ and $r=r(k)$ large enough 
we can set $[f_r^2(c),f_r(c)]$ as a dynamical core such that 
there are no fixed points in $[f_r^2(c),f_r(c)]$ except for 
$x_f$ located on the decreasing branch of $f$, i.e.,
$x_f>2$.\vspace{1mm}
  
\noindent \textit{Ad} C5. We will show that there exists 
a map $f_{r^*}$ such that the critical point $c$ is mapped in 
a few steps onto the unstable fixed point. Such a map is necessarily 
a Misiurewicz map. Let us only remind that Misiurewicz maps 
can be found in a number of ways and our idea of the proof 
is chosen for simplicity. We will call $r^*$ 
the \emph{Misiurewicz parameter}.

Let us firstly justify that the fixed point $x_f>2$ is unstable 
for the $r$-parameter values used in further estimation of 
the Misiurewicz parameter. 
Let $x_f$ be a fixed point of $f(x)=x^2 \exp{(r-x)}$. 
Then for $x_f \neq 0$ we have  $x_f\exp(-x_f)=\exp(-r)$.
The right hand side of this equality decreases with $r$ and 
similarly the left one is a decreasing function of $x_f$ 
for $x_f>1$ (we are interested in $x_f>2$). 
Therefore the value of the fixed point $x_f>2$ 
also increases as $r$ increases, i.e. with the increase of $r$ 
the fixed point $x_f>2$ moves further to the right. 
Since $f'(x_f)=x_f\exp({r-x_f})(2-x_f)=2-x_f$, $x_f$ 
is unstable whenever $x_f>3$. In our below estimation of 
the Misiurewicz parameter we consider $r\in (2.43,2.44)$. 
As for $r\geq 2.43$ we certainly have $x_f>3$, $x_f$ is 
unstable in this parameter range 
(taking into account the above arguments and 
our previous estimation of flip bifurcation $r$ value, 
$x_f$ remains unstable fixed point for all $r>r_0\approx 1.901$).

\begin{figure}[!htb]
\centering{
  \includegraphics[scale=0.42,trim= 50mm 8mm 25mm 10mm]{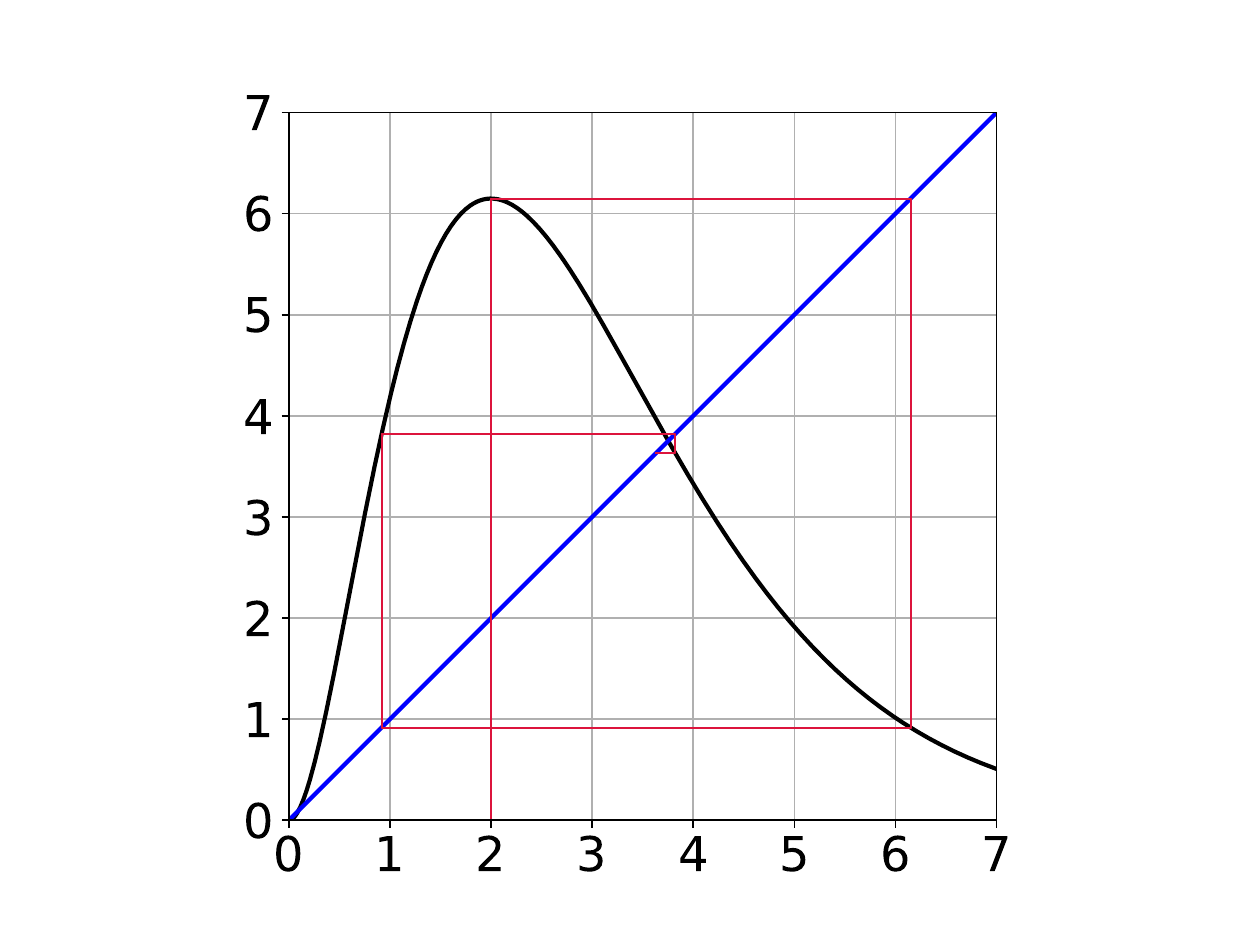}
  \includegraphics[scale=0.42,trim= 0mm 8mm 45mm 10mm]{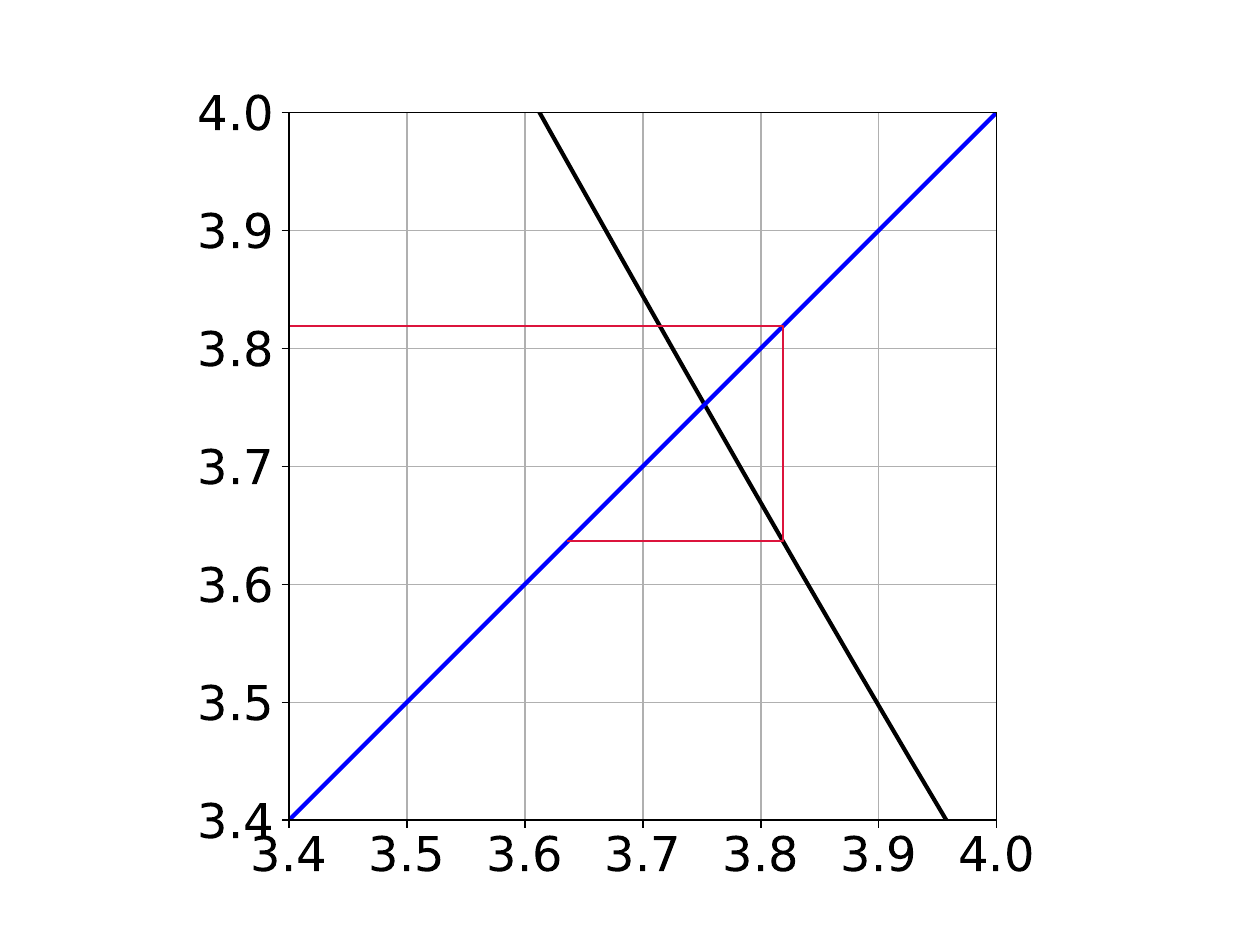}}
\caption{Lower estimation of the Misiurewicz parameter (case $k=0$):
$d(r)>0$ for $r=2.43$. Left: iterations of the critical point.
Right: close-up of cobweb around the fixed point.}
\label{fig:updown}
\end{figure}

We show the existence of the Misiurewicz parameter in the interval
$r\in (2.43,2.44)$. Namely, we prove that for the parameter value
$r^*\approx 2.436$ the critical point $c=2$ is mapped onto 
the unstable fixed point after exactly three iterations.
We make use of the cobweb diagram, instead of using numerical
approximations methods. 
Let $d(r)$ be the distance between the third iteration 
of the critical point $x_c=2$ and the corresponding 
fixed point $x_f(r)$, i.e., $d(r)=f^3(2,r)-x_f(r)$. 
Although we do not have an explicit expression for 
the fixed point, $d(r)$ represents a smooth function in $r$. 
Since $d(r)$ changes sign in the interval $2.43<r<2.44$, 
that is illustrated in Figures~\ref{fig:updown}--\ref{fig:downup}, 
by the Intermediate Value Theorem there exists $r^* \in (2.43,2.44)$
such that $d(r^*)=0$, and therefore the critical point is 
mapped onto the fixed point in three iterations. 
Note that in this way we have obtained actually 
a rigorous proof of the existence of a Misiurewicz map  
and the numerics were only used for estimation 
of the Misiurewicz parameter.\vspace{1mm} 

\begin{figure}[!htb]
\centering{
  \includegraphics[scale=0.42,trim= 50mm 8mm 25mm 11mm]{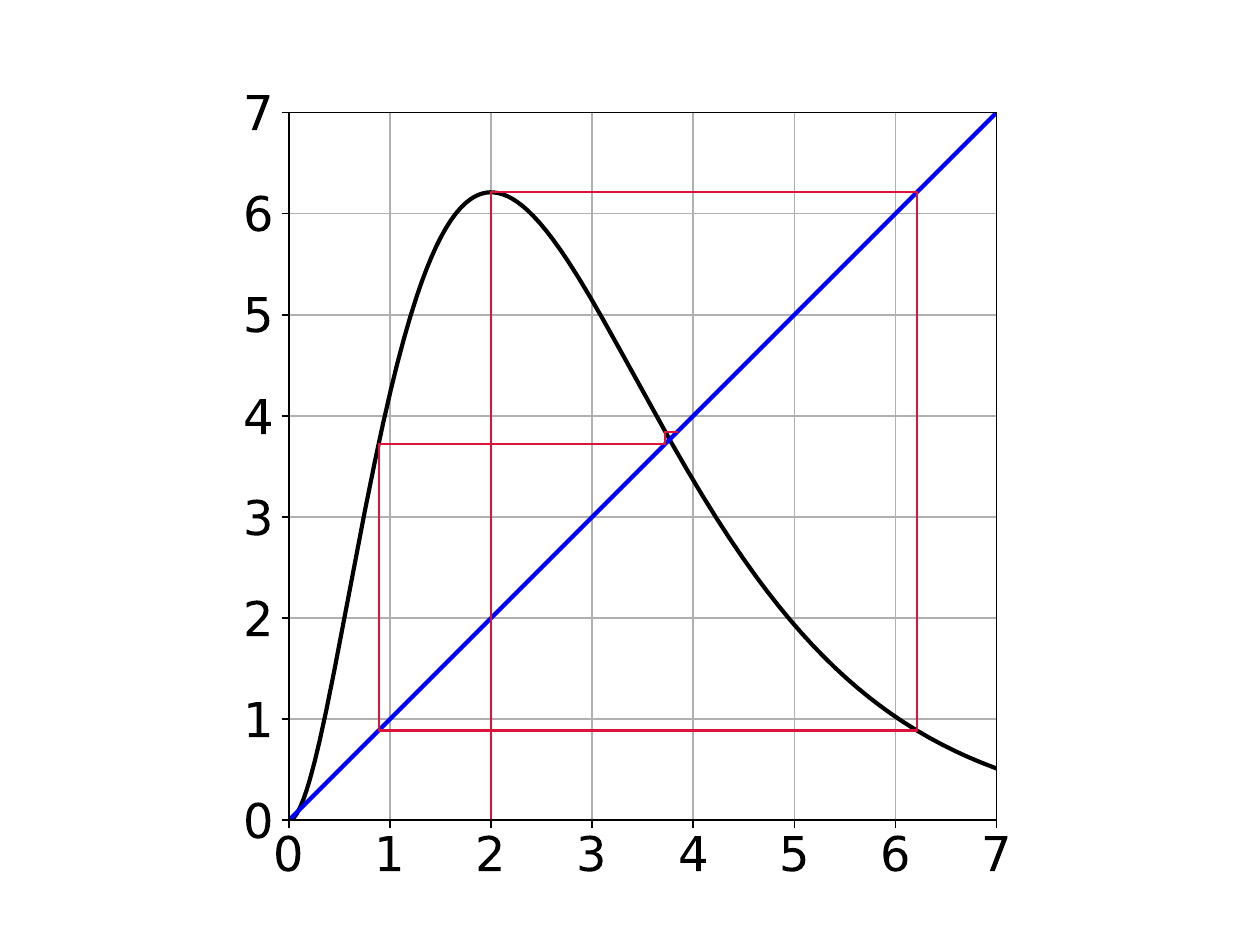}
  \includegraphics[scale=0.42,trim= 0mm 8mm 45mm 11mm]{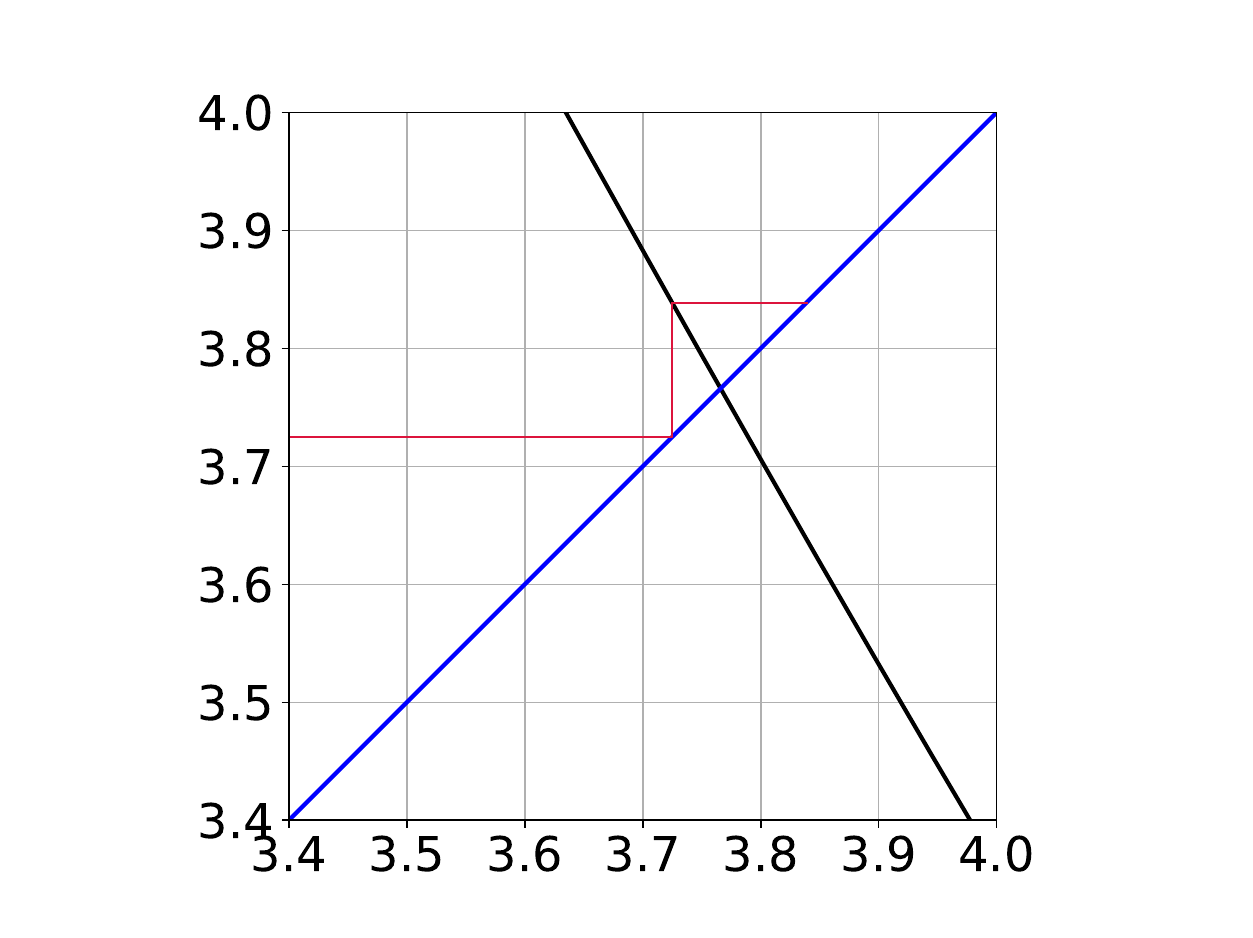}}
\caption{Upper estimation of the Misiurewicz parameter (case $k=0$):
$d(r)<0$ for $r=2.44$. Left: iterations of the critical point.
Right: close-up of cobweb around the fixed point.}
\label{fig:downup}
\end{figure}

\noindent \textit{Ad} C6. We will verify the condition 
$C6$ numerically as in \cite{thunberg2001}.
Let us define $f(x,r) = f_r(x)$, and
$f^n(x,r)=f(f^{n-1}(x,r),r)=f^n_r(x)$ for $n>1$. 
For $r$ close to $r^*$ we may define $\zeta(r)$ 
by the formulas
\[
\zeta(r^*)=f_{r^*}(c)
\quad\text{and}\quad
f^2(\zeta(r),r)=z(r),
\]
where $c$ is the critical point and $z(r)$ is an unstable fixed
point of $f_r$ (see Lemma~\ref{lem:UnstableFP}). 
This means that $\zeta(r)$ is a point close to $f_{r}(c)$ 
with the same type of forward orbit as $f_{r^*}(c)$. 
We abbreviate $f(\zeta(r),r)$ to $\zeta_1(r)$.
The condition C6 now reads 
\[
\frac{d}{dr}(\zeta(r)-f(c,r))\big\vert_{r=r^*}\neq 0. 
\]
By the definition of $\zeta$, we have:
\begin{equation}
\begin{split}\label{eq:zeta}
\frac{d\zeta}{dr} &= 
\Big(\frac{\partial f}{\partial x}(\zeta,r)\Big)^{-1} 
\bigg\{\Big(\frac{\partial f}{\partial x}(\zeta_1,r)\Big)^{-1}
\Big(\frac{dz}{dr}-\frac{\partial f}{\partial r}(\zeta_1,r)\Big)
-\frac{\partial f}{\partial r}(\zeta,r)\bigg\}\\
&=\frac{dz}{dr}\cdot
\frac{\exp(\zeta+\zeta_1-2r^*)}{\zeta\zeta_1(2-\zeta)(2-\zeta_1)}
-\frac{\zeta_1\exp(\zeta-r^*)}{\zeta(2-\zeta)(2-\zeta_1)}
-\frac{\zeta}{2-\zeta}.
\end{split}
\end{equation}
Observe that all the partial derivatives in \eqref{eq:zeta} can be 
computed explicitly. To find the value of $\frac{dz}{dr}$
we differentiate the implicit formula for fixed points
$z^2\exp(r-z)=z$, which gives $\frac{dz}{dr}=\frac{z}{z-1}$.
In consequence, to check the condition C6, we need only the previous
estimation of the Misiurewicz parameter ($r^*\approx2.436$) and the
similar estimation of the corresponding fixed point
($z(r^*)\approx3.761$). The final calculations (done in MATLAB)
are presented in the first column of Table~\ref{tab:Mis}.
As we see 
$\frac{d}{dr}(\zeta(r)-f(c,r))\big\vert_{r=r^*}\approx-3.851$,
so it is nonzero, which shows that the condition C6 is satisfied 
and completes the proof.
\end{proof} 
\begin{table}[!htb]
\caption{Data for estimation of
$\Gamma=\frac{d}{dr}(\zeta(r)-f(c,r))\big\vert_{r=r^*}$
for some $k$ in $[0,0.58]$.}
\label{tab:Mis} 
\centering
$\begin{array}{lrrrrrrr}
\toprule
\text{term} 
& k=0
& k=0.01
& k=0.1
& k=0.3
& k=0.5 
& k=0.55
& k=0.58\\
\midrule
r^*
&2.436
&2.439
&2.461
&2.535
&2.681
&2.759
&2.851\\
\midrule
z(r^*)
&3.761
&3.768
&3.830
&3.999
&4.254
&4.367
&4.491\\
\midrule
\zeta=\zeta(r^*)
&6.186
&6.215
&6.443
&7.130
&8.403
&9.095
&9.948\\
\midrule
\zeta_1=\zeta_1(r^*)
&0.900
&0.895
&0.874
&0.814
&0.731
&0.697
&0.662\\
\midrule
\frac{d\zeta}{dr}(r^*)
&2.335
&2.335
&2.383
&2.594
&3.491
&4.433
&6.426\\
\midrule
\frac{\partial f}{\partial r}(c,r^*)
&6.186
&6.205
&6.343
&6.830
&7.903
&8.545
&9.368\\
\midrule
\Gamma=
\frac{d\zeta}{dr}(r^*)-
\frac{\partial f}{\partial r}(c,r^*)
&-3.851
&-3.870
&-3.960
&-4.236
&-4.412
&-4.112
&-2.942\\
\bottomrule
\end{array}$
\end{table}

\begin{remark}
Unfortunately, when $k>0$ we are not able to prove 
rigorously that the conditions C5 and C6 are satisfied for all $k$ 
in this range. However, we can provide numerical justification of 
both the existence of the Misiurewicz parameter $r^*$ (C5) and 
the nonzero value of the term
$\Gamma=\frac{d}{dr}(\zeta(r)-f(c,r))\big\vert_{r=r^*}$ (C6)
for some values of the parameter $k$ between $0$ and $0.58$.
Recall that, by Lemma~\ref{lem:UnstableFP}, for any $0< k<2$  there
exists $\widehat{r}=\widehat{r}(k)$ such that for all $r\geq
\widehat{r}$ the map $f_{r}$ has an unstable fixed point $x_f>2$.
It occurs that for $k\in(0,0.58]$
one can proceed in much the same way as in the case $k=0$ to show
that there exists $r^*$ such that the map $f_{r^*}$ is 
a Misiurewicz map. We simply search for $r^*$ such that the third
iteration of the critical point is mapped onto the unstable fixed
point (this method does not work for $k>0.58$).
The calculation of the term $\Gamma$ are also similar
as in the case $k=0$. The formula \eqref{eq:zeta} still holds, but
$\frac{dz}{dr}=\frac{z^2}{\exp(z-r)+z^2-2z}$ for $k>0$.
The results of our numerical analysis are summarized in
Table~\ref{tab:Mis}. Note that only the last row of
Table~\ref{tab:Mis} does not display monotonic behaviour.
Finally, in Figure~\ref{fig:graph} we present plots of the Misiurewicz
parameter $r^*$ and the term $\Gamma$ from the condition C6  
as functions of the parameter $k\in[0,0.58]$ (in both cases the step
size is $0.001$). Observe the rapid growth of the value of 
$\Gamma$ near $k=0.58$.
\end{remark}

\begin{figure}[!htb]
\centering{
\includegraphics[scale=0.78, trim= 0mm 3mm 0mm 5mm]{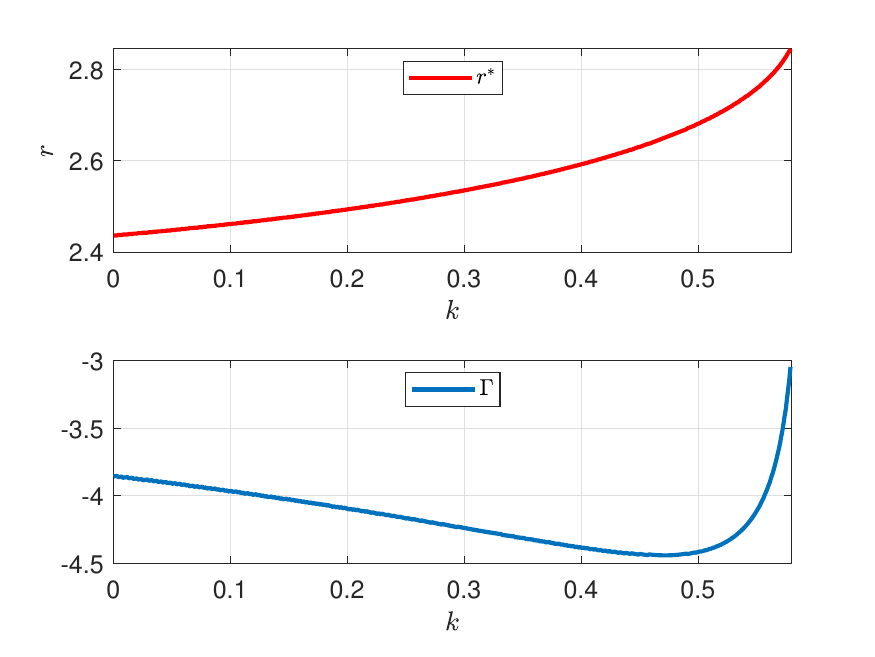}}
\caption{The Misiurewicz parameter $r^*$ (upper plot) and 
$\Gamma=\frac{d}{dr}(\zeta(r)-f(c,r))\big\vert_{r=r^*}$ (lower plot)
as functions of the parameter $k\in[0,0.58]$ (step size $0.001$).}
\label{fig:graph}
\end{figure}

We finish this subsection with a result showing that in the Chialvo
model two cases (periodic attractors and acips) occupy a set of full
measure in parameter space. Namely, the following classical theorem 
from \cite[Th. B]{AvilaetAll2003} can be applied in the analysis 
of the reduced Chialvo map. The original result is more general, as it
considers the class of quasiquadratic maps, which contains
S-unimodal maps with a non-degenerate critical point.
Recall that a one-parameter family of S-unimodal maps is called
\emph{trivial} if all maps have the same kneading sequence and 
the multipliers of any non-repelling orbits are the same for all maps.

\begin{theorem}[Avila, Lyubich, de Melo]\label{thm:ALdM}
Let $\Lambda\subset\mathbb{R}$ be open and let
$\{f_\lambda\}_{\lambda\in\Lambda}$
be a non-trivial real-analytic family of S-unimodal maps
with a non-degenerate critical point. Then for
Lebesgue-almost all parameter values $\lambda\in\Lambda$ the map
$f_\lambda$ has either a periodic attractor (such a map is called
regular) or an interval attractor supporting an acip (such a one is
called stochastic).
\end{theorem}

\begin{proposition}\label{cor:pori}
If $k\in[0,2)$ is fixed then for Lebesgue-almost 
all parameters $r$ the 1D Chialvo map $f_{r}$ 
has either a periodic attractor or an interval attractor 
supporting an acip. 
\end{proposition}

\begin{proof}
It suffices to show that the 1D Chialvo family with respect
to $r$ is non-trivial. However, when $k\in[0,2)$ our family
contains two maps with the kneading sequences of the form
$(CC\dotsc)$ (the critical point is fixed) and 
$(C1\dotsc)$ (the critical point is on the left from 
the fixed point), which is our claim.
\end{proof}
      
\subsection{Topological chaos}\label{subsec:topological}
Let us start this subsection with three classical definitions 
of topological chaos for discrete dynamical systems given 
by Li and Yorke
(\cite{LiYorke1975}), Block and Coppel (\cite{BlockCoppel1992}) 
and Devaney (\cite{devaney2003}). The relations among these
three definitions of chaos for continuous maps are comprehensively
explained in \cite{AulbachKieninger2001}. Recall that $\Sigma$ 
denotes the \emph{sequence space} on the two symbols
$\{\alpha=(a_0a_1a_2\dotsc)\mid a_i=0\text{ or }a_i=1\}$ 
with the metric
$d(\alpha,\beta)=\sum_{i=0}^\infty{\abs{a_i-b_i}}/{2^i}$. 
The \emph{shift} map $\sigma\colon\Sigma\to\Sigma$ is given by
$(a_0a_1a_2\dotsc)\to(a_1a_2a_3\dotsc)$.
From now on assume that $X$ is a compact metric space and 
$f\colon X\to X$ is continuous. 

\begin{definition}[L/Y-chaos]
 A map $f$ is called \textit{chaotic in the sense of Li and Yorke}
 (\textit{L/Y-chaotic} for short) if there exists an uncountable
 subset $S$ (called \textit{scrambled set}) of $X$ 
 with the following properties:
 \begin{enumerate}
      \item $\limsup\limits_{n\rightarrow \infty}{d(f^n(x),f^n(y))}>0$ 
			for all $x,y \in S$, $x\neq y$,
      \item $\liminf\limits_{n\rightarrow \infty}{d(f^n(x),f^n(y))}=0$ 
			for all $x,y \in S$, $x\neq y$,
      \item $\limsup\limits_{n\rightarrow \infty}{d(f^n(x),f^n(p))}>0$ 
			for all $x\in S$, $p\in X$, $p$ periodic.
 \end{enumerate}
\end{definition}

\begin{definition}[B/C-chaos]
A map $f$ is called \textit{chaotic in the sense of Block and Coppel}
(\textit{B/C-chaotic} for short) if there exists $m \in \mathbb{N}$
and a compact $f^m$-invariant subset $Y$ of $X$ such that $f^m|_{Y}$
is semi-conjugate to the shift on $\Sigma$, i.e., if there exists 
a continuous surjection $h\colon Y \to \Sigma$ satisfying 
$h \circ f^m = \sigma \circ h$ on $Y$.
\end{definition}

\begin{definition}[D-chaos]
 A map $f$ is called \textit{chaotic in the sense of Devaney} 
 (\textit{D-chaotic} for short) if there exists a compact invariant
 subset $Y$ (called a \textit{D-chaotic set}) of $X$ with 
 the following properties:
 \begin{enumerate}
      \item $f|_{Y}$ is topologically transitive,
      \item the set of periodic points $\textrm{Per}(f|_{Y})$ 
      is dense in $Y$, i.e.,  $\overline{\textrm{Per}(f|_{Y})}=Y$,
      \item $f|_{Y}$ has sensitive dependence on initial conditions.
 \end{enumerate}
\end{definition}

The following result provides the sufficient condition for 
the existence of topological chaos in the 1D Chialvo model 
in all the above senses.
\begin{theorem}\label{thm:three}
If the 1D Chialvo map $f$ satisfies the condition
\begin{equation}\label{eq:three}
f^2(c)<f^3(c)<c<f(c)\quad
\text{($c=2$ is the critical point)},
\end{equation}
then it is chaotic in the sense of Li and Yorke, 
Block and Coppel, and Devaney.
\end{theorem}

\begin{proof}
Let $x_0=c$ and $x_i=f(x_{i-1})$ for $i=1,2,3$.
By the Intermediate Value Theorem, the map 
$f\colon[x_0,x_1]\to[x_2,x_1]$ has a unique fixed point
$z\in(x_0,x_1)$. Observe that
\[
f\big([x_2,x_0]\big)=[x_3,x_1]
\supset[z,x_1]
\quad\text{and}\quad
f\big([x_0,z]\big)=[z,x_1]
\]
and, in consequence,
\[
[x_2,z]=[x_2,x_0]\cup[x_0,z]\subset
f^2\big([x_2,x_0]\big)\cap f^2\big([x_0,z]\big).
\]
Hence $f^2$ is turbulent, i.e., there are nonempty closed
subintervals $J$, $K$ of $[x_2,x_1]$ with disjoint interiors
such that $J\cup K\subset f^2(J)\cap f^2(K)$.
By \cite[Prop. 3.3]{AulbachKieninger2001}, the map
$f\colon[x_2,x_1]\to[x_2,x_1]$ is B/C-chaotic.
Moreover, by \cite[Th. 4.1]{AulbachKieninger2001},
this is equivalent to being D-chaotic and,
by \cite[Th. 4.2]{AulbachKieninger2001}, this implies
being L/Y-chaotic.
\end{proof}

\begin{remark}
It is worth pointing out that according to the above definitions 
of chaos the map is chaotic if it displays chaotic behaviour on 
a nonempty compact invariant subset of
the dynamical core (not necessarily on the whole domain interval)
and this set can be small both in the sense of measure and category.
Notice also that the above proof works for an arbitrary 
unimodal map, whose critical point satisfies \eqref{eq:three}.
\end{remark}

It occurs that the sufficient condition from
Theorem~\ref{thm:three} is actually satisfied
for some range of parameters in the Chialvo model.

\begin{proposition}\label{prop:topchaos}
For $k=0$ and $2.6\leq r \leq 2.9$ the 1D Chialvo map  
$f_r$ satisfies \eqref{eq:three} and, in consequence, 
is chaotic in the sense of Li and Yorke, 
Block and Coppel, and Devaney.
\end{proposition}

The technical proof of Proposition~\ref{prop:topchaos}
is postponed to Appendix~\ref{appendix:chaos}.

\begin{figure}[!ht]
\centering{
\includegraphics[scale=0.324, 
trim = 25mm 5mm 15mm 8mm]{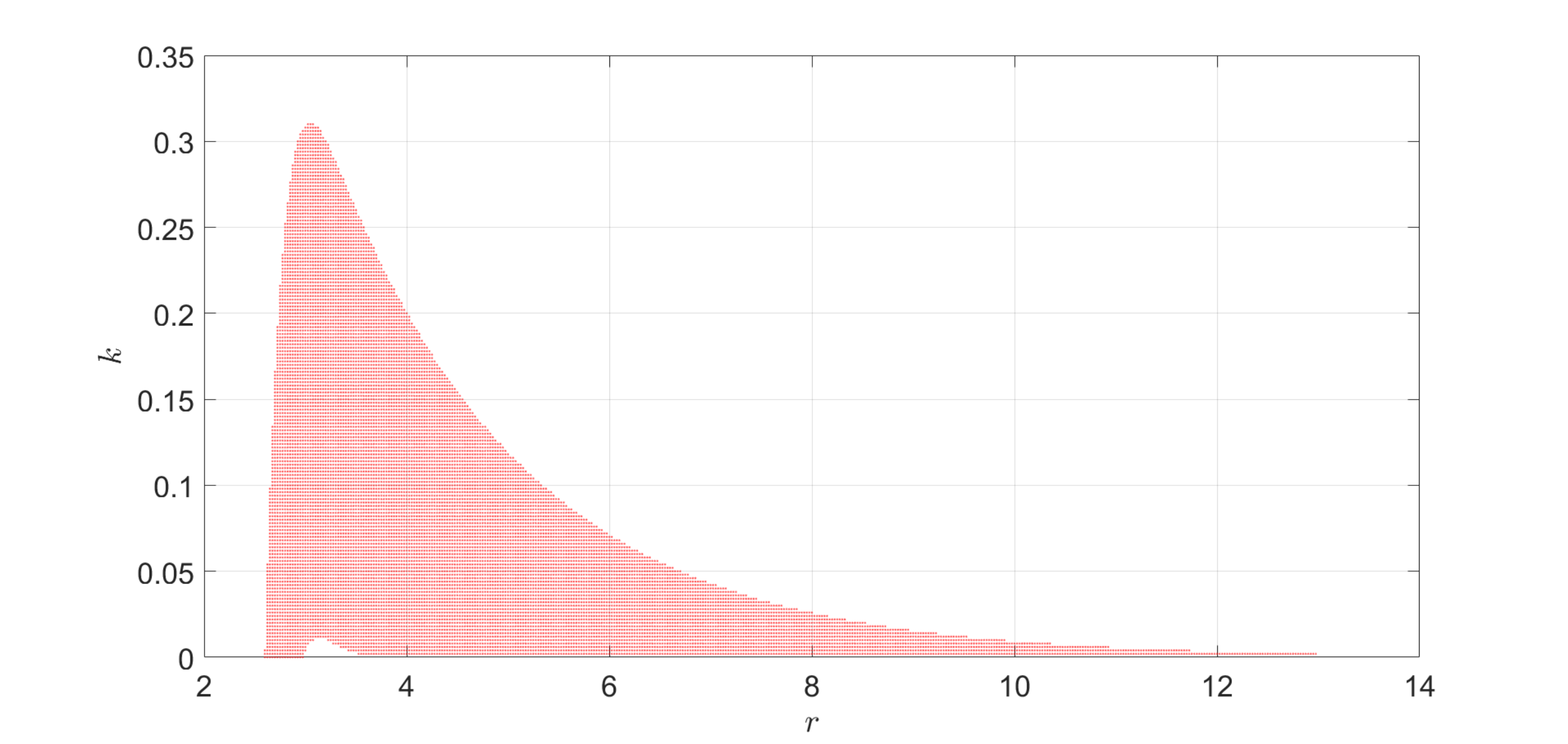}}
\caption{Topological chaos values (red dots) 
for the 1D Chialvo model in $(r,k)$ parameter plane.}
\label{fig:rvsk}
\end{figure}

\begin{remark}
When $k>0$, by means of numerical simulations, we are able to
observe a dependence between values of $r$ and $k$ for which 
the condition \eqref{eq:three} from Theorem~\ref{thm:three} is
satisfied. However, it is difficult to obtain an interval 
for $r$ where all the inequalities are satisfied independently 
of the values of $k$. Figure~\ref{fig:rvsk} below depicts 
the result of a numerical simulation in MATLAB. 
We plot values of $r$ and $k$ for which the condition
\eqref{eq:three} is satisfied. The intervals of study for both
parameters are: $2\leq r\leq 14$ with step size $0.025$ and 
$0\leq k\leq 0.35$ with step size $0.002$. The short of 
the program and the table with computed parameter values can 
be found in an open repository \cite{dataset}. 
\end{remark}

\subsection{Conclusions for the neuron activity analysis}
It should be emphasized that the standalone 1D model 
\eqref{eq:Chialvo1DIM} is able to reproduce 
a wide range of interesting neuronal behaviours
(similarly to its 2D counterpart). 
In Figure \ref{fig:1DIMDynamics} we show two examples of 
voltage trajectories. The trajectory with initial condition
$x_\textrm{init}=2.2$ and parameter values $k=0.1$ and $r=2.45$
displays `chaotic' spiking with subthreshold oscillations
(depicted on the left), which can be also classified as
(non-regular) MM(B)Os (mixed-mode (bursting) oscillations), 
i.e., spikes (or bursts) interspersed with small 
(`subthreshold') oscillations. 
Note that Theorem~\ref{thm:ai} guarantees that 
for $k=0$ the set of parameters $r$ for which the model 
exhibits `strongly chaotic' behaviour (chaotic spiking/ bursting,
irregular MMOs, etc.) is `observable' (has
a positive Lebesgue measure). However, the estimation of its
actual size is a completely different and much harder task
(see for instance \cite{Luzzatto2006}).
Moreover, for $k=0$ `weak chaotic' behaviour of the model 
is even more clearly `visible' as, according to
Proposition~\ref{prop:topchaos}, it occurs for the set of
parameters $r$ containing a `quite big' interval.

On the other hand, for the same value of $r$ and $k$ increased 
to $k=0.2$ the neuron with initial voltage value
$x_\textrm{init}=2.25$ exhibits, after initial adaptation, 
regular MMOs (depicted on the right). In this case, both spikes
and subthreshold (non-spikes, i.e. low amplitude) 
oscillations occur periodically. 
Observe that, according to Theorem~\ref{thm:apo}, 
periodicity and asymptotic periodicity (periodic
spiking/bursting, phasic spiking/bursting, regular MMOs, etc.)
are the most `common' types of behaviour of 
the~model under consideration.

\begin{figure}[!htb]
\centering{
    \includegraphics[scale=0.21, 
    trim= 15mm 0mm 5mm 0mm]{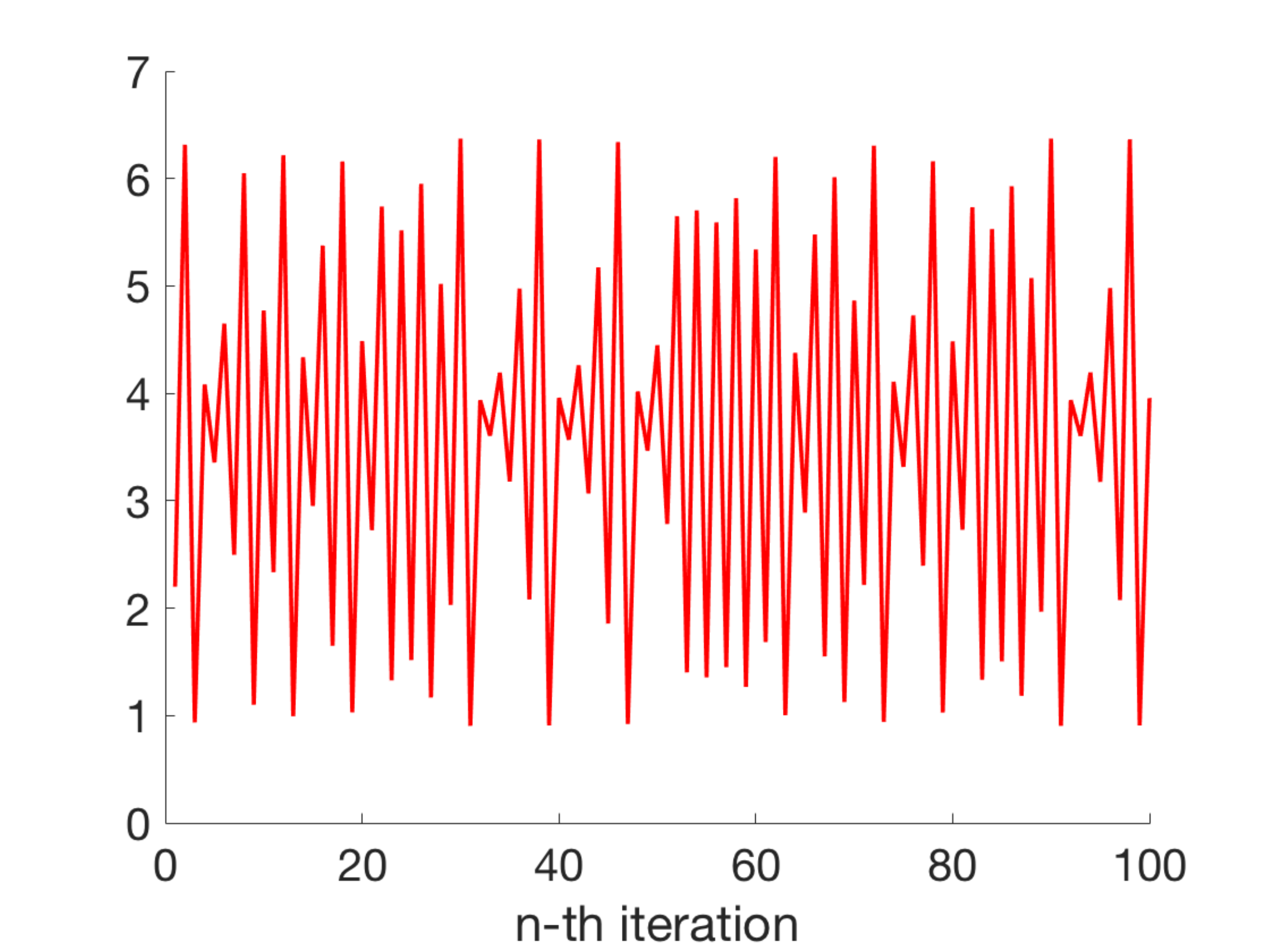}
    \includegraphics[scale=0.21, 
    trim= 10mm 0mm 10mm 0mm]{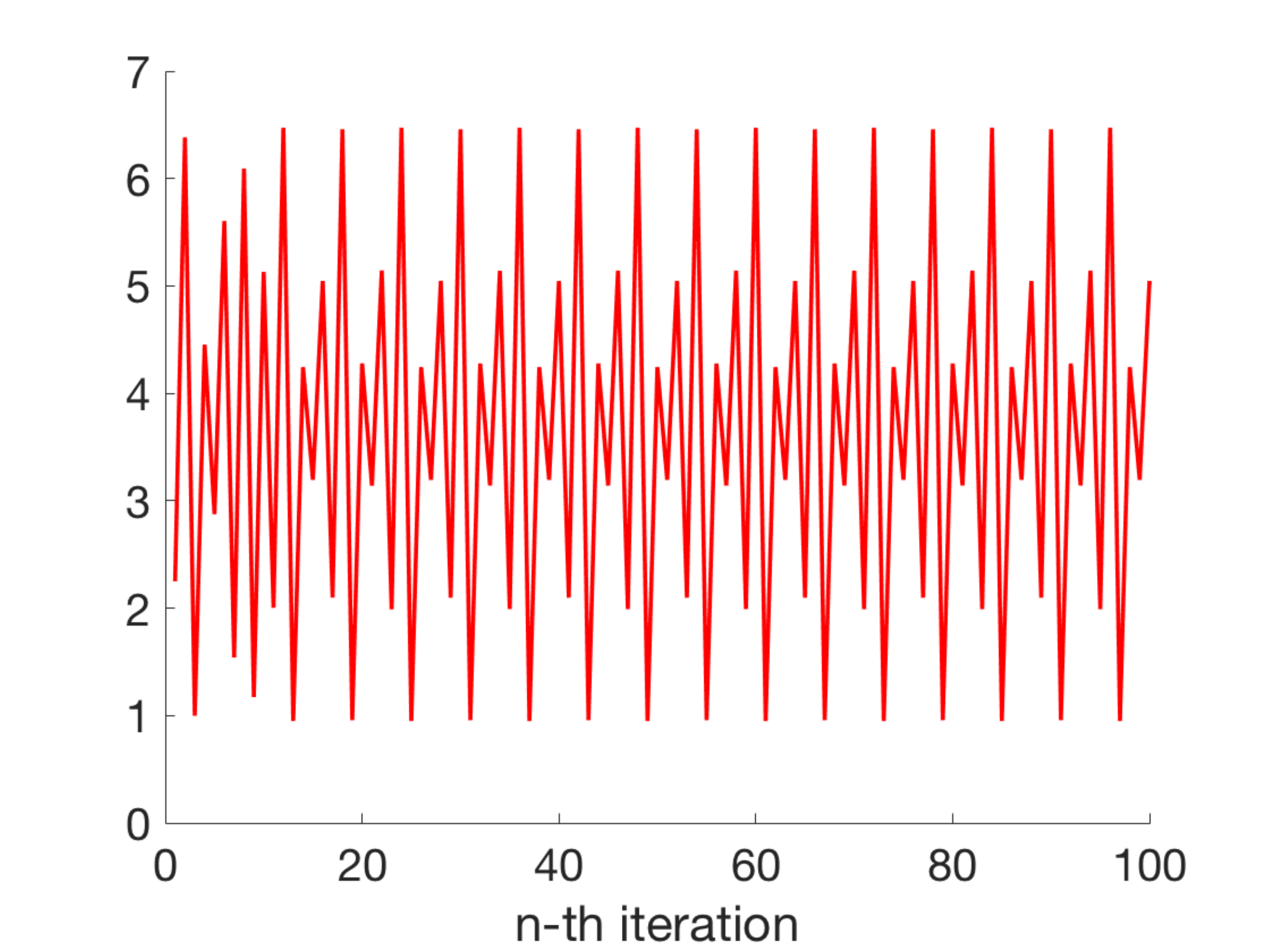}}
    \caption{Chaotic spiking with subthreshold oscillations 
    (left) and regular mixed-mode oscillations after initial
    adaptation (right)
    in 1D Chialvo model (for parameter values see main text).}
    \label{fig:1DIMDynamics}
\end{figure}

It is worth notifying that MM(B)Os can play an important role in
neuron's information processing but for non-discrete models can
occur only in higher dimensions (continuous models require at
least three dimensions and for hybrid models, with continuous
subthreshold dynamics, starting from two dimensions,
see \cite{wild2}). Actually, during our numerical experiments 
with the 1D Chialvo model we have also encountered,
e.g. regular tonic spiking and phasic spiking.

Moreover, observe that, due to Theorem~\ref{thm:types},
if we restrict ourselves to the dynamical core, then
the reduced Chialvo model cannot directly exhibit bistable
behaviour (coexistence of periodic and chaotic attractors) 
for fixed values of parameters $r$ and $k$, 
because the 1D Chialvo map admits a unique metric attractor,
which in general is periodic, solenoidal or interval,
but with full probability periodic or interval
(see Prop~\ref{cor:pori}). However, chaotic
behaviour in the dynamical core can, of course, coexist
with an attracting fixed point outside the dynamical core.
Moreover, the substitute of bistability effect, i.e.,
interweaving of periodic and chaotic attractors
in the 1D model can be also observed while varying 
the parameter $r$ corresponding to 
the second (recovery) variable of the 2D model.
This is formally assured by Theorems~\ref{thm:apo} 
and~\ref{thm:ai}.

Various oscillations modes in the reduced Chialvo model, 
as we have seen, are closely related to the existence 
of periodic attractors. 
Restricting the domain of the initial conditions 
to the dynamical core, we know that there is 
at most one attracting periodic orbit 
since the map is $S$-unimodal. 
Hence if the $1D$ model exhibits periodic oscillations, 
then this (asymptotically) periodic spike pattern is unique 
among observed oscillations modes 
(for fixed parameters $r$ and $k$). 
The arrangement of smaller and larger oscillations (spikes) 
and their amplitudes depend on the itinerary 
of this attracting periodic orbit. 
In particular, period-two attracting orbit 
(with points on different sides 
of the critical point $c$) corresponds 
to tonic regular spiking 
(when each spike has the same amplitude and length 
of each interspike-interval is the same). 
However, for many parameter values the system exhibits 
also chaotic-like (not regular) oscillations 
since there are ranges of parameters when the map is chaotic. 
Therefore determining values of parameters corresponding 
to various oscillations modes firstly demands providing 
exact bifurcation diagram of 
the map $f_{r,k}$ in the $(r,k)$-plane. 
This, however, would be a very challenging task 
since even with one parameter fixed the system undergoes 
cascades of rapid bifurcations. 
This is exemplified in Figure~\ref{fig:bif_diagram}. 
In fact, left panel of this figure allows to conclude, 
approximately, values of $r$ (for $k=0.05$ fixed) corresponding 
to resting and periodic oscillations with period $2$, $4$, $8$ 
(and perhaps more) as well as ranges of chaotic 
(or very large period) oscillations. 
The second panel of this figure provides analogous information 
for varying $k$ and $r=2.6$ fixed. 
Note, however, that in all the cases, 
small change of the parameter might cause onset 
of chaos or quite sudden disappearance of chaotic behaviour. 
Approximate regions in $(r,k)$-parameter space 
when chaotic oscillations might occur might be deduced 
from the red area depicted in Figure~\ref{fig:rvsk}.

\subsection{Remarks on the 2D Chialvo model}

Although, this work is concerned with 1D reduction of the Chialvo model, 
for completeness let us discuss some properties of the system 
\eqref{eq:Chialvo1}-\eqref{eq:Chialvo2} 
which might indicate interesting pathways for further studies.

\begin{figure}[!b]
\centering{
    \includegraphics[scale=0.4, 
    trim= 30mm 0mm 0mm 0mm]{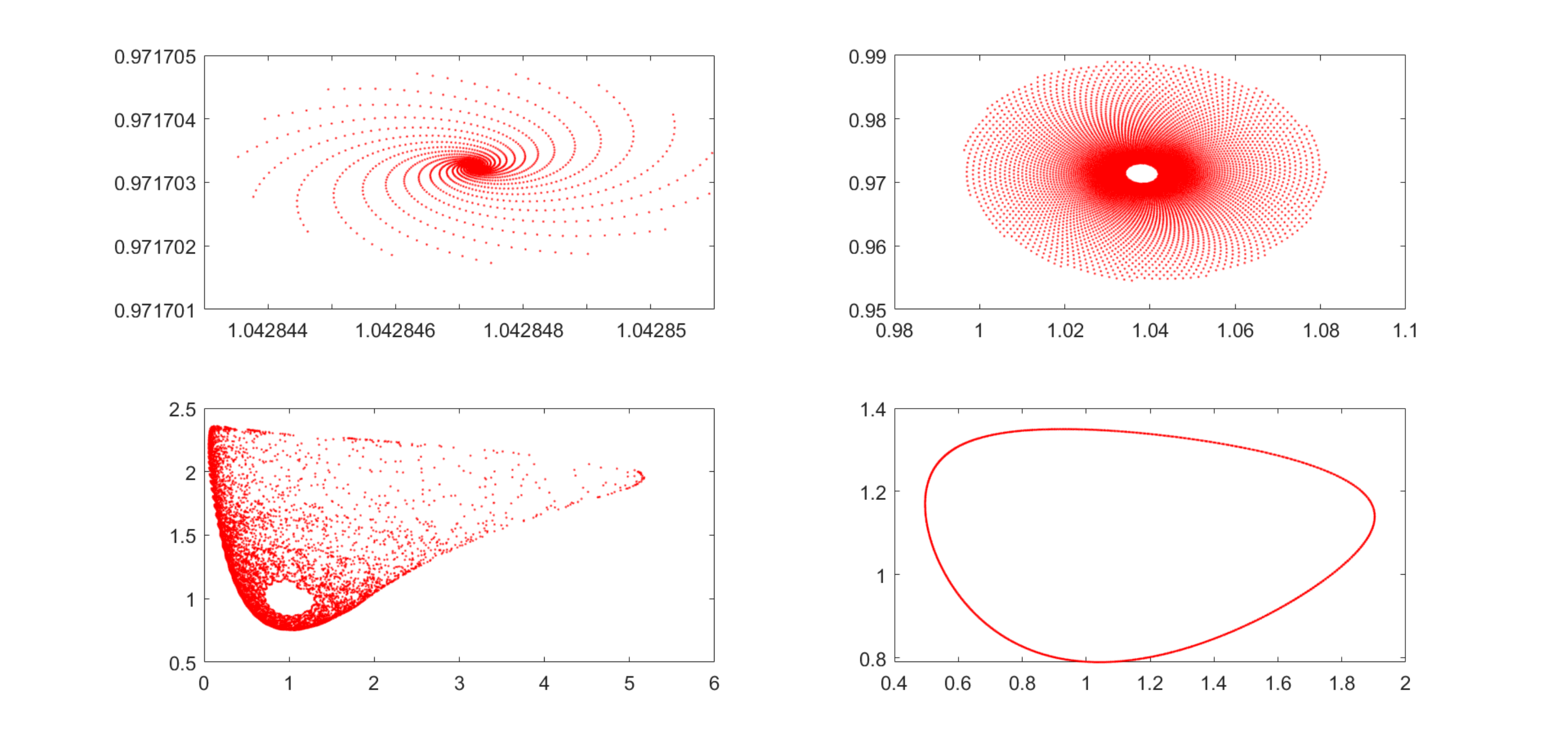}}
    \caption{Numerical attractors in the 2D Chialvo model
    for the fixed values of $a=0.89$, $c=0.28$, $k=0.03$
    and four different values of $b:0.166,\,0.1668,\,
    0.169,\,0.1738$ (clockwise from upper left corner)}
    \label{fig:attractors}
\end{figure} 

The work of Chialvo \cite{chialvo1995} discusses the case 
of $k=0$ and $k$ of small positive value keeping fixed values 
of parameters $a$ and $c$ and only varying $b$ when needed. 
The parameter $k$ is treated as the bifurcation parameter. 
As mentioned in Introduction, for $k=0$, 
the point $(x_{\textrm{f}\,0},y_{\textrm{f}\,0}):=(0, c/(1-a))$ 
is always a stable fixed point of the system. 
For $k\neq 0$ the system \eqref{eq:Chialvo1}--\eqref{eq:Chialvo2} 
can have up to three fixed points.  

Chialvo \cite{chialvo1995} treats only the case 
when the phase portrait has exactly one equilibrium point 
and the values of parameters $a$ and $c$ are fixed as $a=0.89$ and $c=0.28$. 
Nevertheless, even in this prescribed regime of parameters 
the system features rich behaviour which is directly matched 
with activity properties of neurons. With this choice of parameters $a$ and $c$, 
small values of $k$ (e.g. $k=0.02$) and $b=0.6$ the unique fixed point 
is globally attracting and the system is \emph{quiescent-excitable} 
since it can generate an action potential when properly stimulated 
but is not able to maintain oscillatory spiking. Increasing $k$ 
(while keeping other parameters fixed) causes fixed point loses stability 
and oscillatory behaviour appears (bifurcation from quiescent-excitable to oscillatory solution).  
Obviously, the $2D$ model is also able to exhibit chaotic-like behavior 
which Chialvo \cite{chialvo1995} observed e.g. for $k=0.03$ and $b=0.18$ 
(and $a=0.89$, $c=0.28$ as stated). 
In this case the model displays chaotic mixed-mode oscillations 
(bursting oscillations with large spikes often followed 
by oscillations of smaller amplitude, see Figure~10 therein).
 
However, as we already mentioned, the system \eqref{eq:Chialvo1}--\eqref{eq:Chialvo2} 
can have up to three fixed points.  Their  existence, stability and bifurcations 
were studied in \cite{Jing2006}. In particular, the authors provided explicit conditions 
when the $2D$ Chialvo system has a given number of fixed points and theorems 
establishing the existence of fold, flip bifurcation and Hopf bifurcations. 
They have also proved the existence of chaos in the sense of Marotto. 
These analytical studies are then complemented by numerical simulations 
which in addition reveal e.g. various period-doubling bifurcations 
and different routes to chaos, sudden disappearance of chaos 
and strange chaotic and nonchaotic attractors. 
In the analytical part they treat $k$ as  a bifurcation parameter 
but interestingly allow also negative values of $k$. 
On the other hand, in the numerical part they consider also non-autonomous 
Chialvo model where the external stimulus is periodic 
and the amplitude or frequency of this input can be a bifurcation parameter.

In a recent short communication paper \cite{NewPaperOnChialvo}, 
the authors, contrary to other studies, consider  
$(a,b)$-parameter space where $c$ and $k$ are fixed. 
These studies are numerical but illustrate the existence of interesting structures.  
In particular, they suggest Neimark–Sacker bifurcation 
of the fixed point and mode locking behaviour. 
The location of parameter regions corresponding Neimark–Sacker bifurcation 
and the birth of Arnold tongues are determined analytically. 
They have also numerically identified comb-shaped periodic regions 
(corresponding to period-incrementing bifurcations) 
and shrimp-shaped structures immersed in large chaotic regions.

Finally, the recent preprint \cite{chialvo2D} undertakes the analysis 
of the 2D Chialvo model combining rigorous numerical methods 
and a topological approach based on the Conley index and Morse decompositions 
which allows to split considered range of parameters into classes of equivalent dynamics. 
However, again due to the computational complexity of this $4$-parameters dependant system, 
the authors fix $a = 0.89$ and $c = 0.28$ (similarly as in Chialvo paper \cite{chialvo1995}) 
and explore some rectangular area of $b>0$ and $k>0$. 
Their study is enriched with the analysis of chain-recurrence properties 
of sets and additionally incorporates machine learning 
for identifying parameter ranges that yield chaotic behaviour.

\begin{figure}[!ht]
\centering{
    \includegraphics[scale=0.4, 
    trim= 30mm 0mm 0mm 0mm]{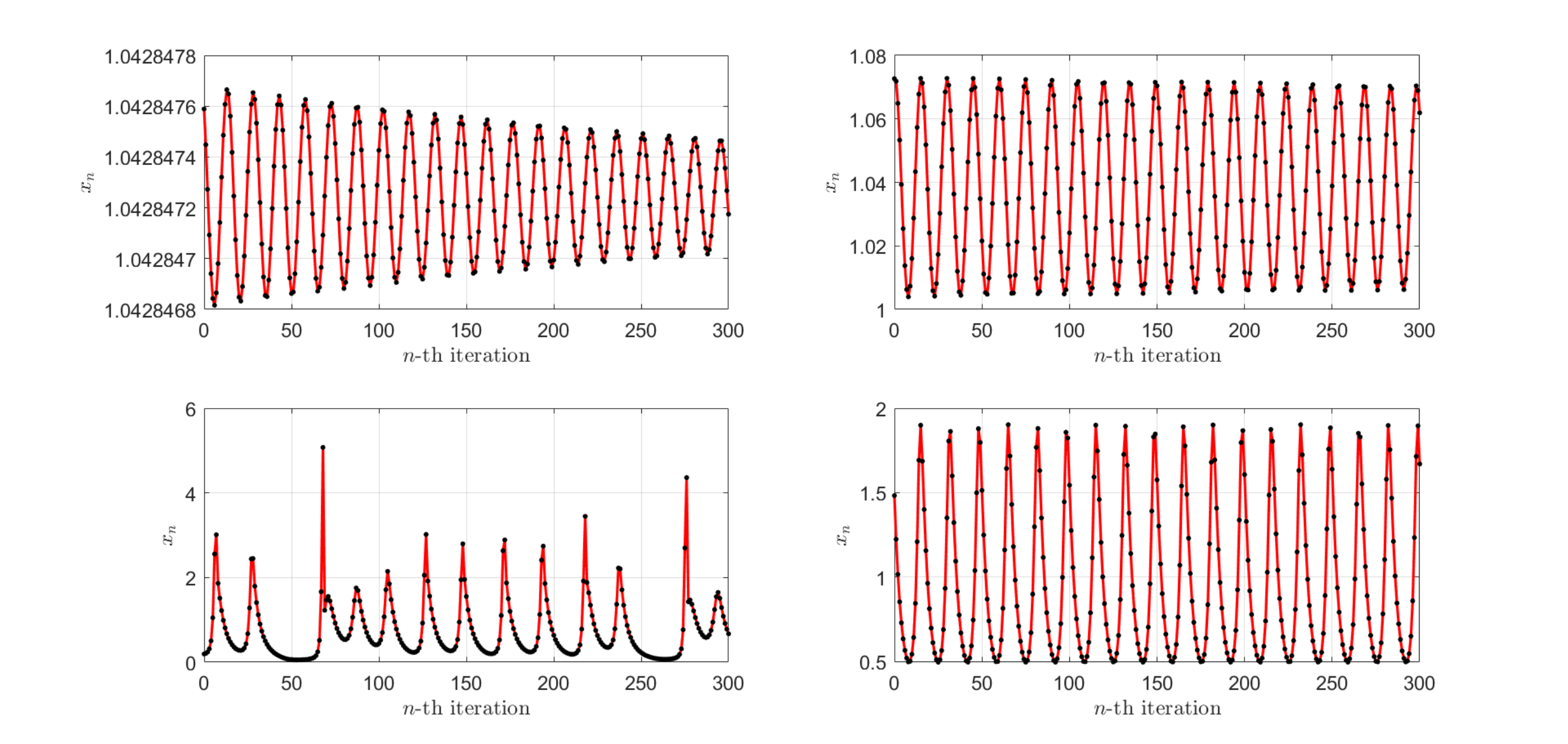}}
    \caption{Time series of the voltage $x$ 
		for the numerical attractors of the 2D Chialvo model
    from Fig.~\ref{fig:attractors}}
    \label{fig:tseriesCh}
\end{figure}

Although the current work is concerned with 1D Chialvo model, 
we have performed numerous simulations of the 2D model. 
We have decided to include only some of them in this work. 
This is due to the fact that the model is extremely sensitive to any, 
sometimes even arbitrary small, changes of parameters 
that can lead to very different behaviours. Indeed, in Fig. \ref{fig:attractors} 
we present numerically obtained representative phase portraits 
for $a=0.89$, $c=0.28$, $k=0.03$ and four different values 
of $b$ ($0.166$, $0.1668$, $0.169$, $0.1738$, respectively). 
Each picture presents a typical trajectory of the system after 
removing the transient part. Thus these pictures show various types 
of possible attractors for the 2D model. We see that increasing $b$ 
in such a small range of values, while keeping all the other parameters fixed, 
causes shifting the dynamics from the one corresponding to the attracting focus, 
through periodic limit cycle, to chaotic oscillations 
where both the interspike intervals and the number of small oscillations 
between consecutive spikes vary. 
Looking at Fig. \ref{fig:attractors} we can suspect that 
the system firstly undergoes bifurcation when 
the globally attracting fixed point loses its stability 
and the stable limit cycle appears and soon after that 
another bifurcation when asymptotically periodic dynamics 
is replaced by chaotic one. Corresponding voltage traces 
are presented in Fig. \ref{fig:tseriesCh}. 

Our brief investigation of the 2D model and the mentioned works 
\cite{chialvo1995,Jing2006,NewPaperOnChialvo,chialvo2D} provide insightful 
and valuable observations. However detailed dynamical picture 
of the 2D Chialvo model is still very incomplete. 
This is mainly due to the fact that discrete systems, 
even in low-dimenions such as the Chialvo model, yield more challenges 
in their analytical (and numerical) analysis. 
Indeed, in planar continuous systems (i.e. systems defined by ODEs) 
one has e.g. Poincar\'{e}-Bendixson Theorem which allows to draw conclusions 
on the asymptotic behaviour of trajectories 
(by describing possible types of invariant limit sets). 
In particular, in 2D continuous systems chaotic dynamics is excluded, 
whereas in discrete settings one can observe chaos even in 1 dimension 
(i.e on the interval or on the circle). The second reason is the nature 
of the Chialvo model itself, by which we mean, and as we demonstrated, 
that the model is very sensitive to even very small change of parameters 
as it happens that in the very small region of parameter space 
it might undergo not one, but a couple of bifurcations. 
This in turn makes validation of the numerical analysis 
of the model challenging and emphasis 
the need for rigorous analytical studies which, 
however, due to the above mentioned reasons are especially hard.    
Therefore our aim was to show that even the reduced Chialvo model 
can be interesting for the mathematical analysis and not necessarily 
less plausible than its full version.

\section*{Discussion} 
We have performed detailed analysis of the reduced Chialvo model
\eqref{eq:Chialvo1DIM} which describes the evolution of 
the membrane voltage when recovery variable is frozen, 
i.e., becomes a parameter. Let us briefly summarize 
these results here.

Firstly, we established that the model \eqref{eq:Chialvo1DIM} 
can be seen as iterations of the S-unimodal map $f_{r}$, 
which opened the way to prove its various properties with 
the use of the powerful theory of S-unimodal maps 
(and, more general, maps with negative Schwarzian derivative). 
In Theorems \ref{thm:flip}, \ref{thm:fold} and \ref{thm:foldk} 
we rigorously showed that the model undergoes flip and 
fold bifurcations of fixed points when $r$ is 
a bifurcation parameter and a fold bifurcation 
with respect to $k$ parameter, providing formulas for 
the bifurcation parameter values and fixed points involved. 
In particular, due to the fact that Schwarzian derivative 
is negative, all the flip bifurcations are supercritical. 

Next, discussing different configurations of the map 
with respect to the number and location of its fixed points 
and precising what should be treated as dynamical core of the map 
we justified that the map can be in fact viewed as an S-unimodal map.
Moreover, Lemmas \ref{lem:DynCore} and \ref{lem:UnstableFP} 
asserts that for wide range of $(r,k)$ parameter space in 
the dynamical core of the map there is only one fixed point 
which is unstable. This instability of the fixed point played role 
in justifying the existence of Misiurewicz maps in the family
$\{f_r\}$ (with $k$ fixed) and showing that for positive Lebesgue
measure set of parameters these maps admit an absolutely continuous
invariant probability measure (acip) with positive Lyapunov exponent
for almost all initial conditions (see Theorem \ref{thm:ai} and
following discussion). In all these cases, the map has an interval
attractor which is the support of the acip. This situation is often
referred to as metric chaos. The effective complementary case 
(see Proposition \ref{cor:pori}) is when the map $f_r$  
has a periodic attractor which is then necessarily unique 
(Theorem \ref{mainChaos1}). This fact has some implications for 
the numerics, i.e. in the regime of the existence of stable periodic
orbit, its uniqueness and the fact that it attracts almost 
all initial conditions, guarantees that by randomly choosing initial
point $x_0$ and iterating it under the one-dimensional Chialvo map 
we can reconstruct this periodic attractor. Abundance of 
this periodic case is captured in Theorem \ref{thm:apo}. 
On the other hand, Theorem \ref{thm:three} gives 
the combinatorial condition for the map to be chaotic 
in the sense of Li and Yorke, Block and Coppel, and Devaney.
Proposition \ref{prop:topchaos} and further numerical analyses 
estimate the subset of $(r,k)$ parameter space where 
this topological chaos occurs. 

As argued in some references and mentioned in Introduction, 
the equation \eqref{eq:Chialvo1DIM} sometimes qualitatively
approximates the dynamics of voltage-variable in the full model
\eqref{eq:Chialvo1}--\eqref{eq:Chialvo2}, especially when 
it can be viewed as a slow-fast system with the slow variable $y$,
virtually frozen or acting as a slowly varying bifurcation parameter
of the fast subsystem. However, the form of the full model
\eqref{eq:Chialvo1}--\eqref{eq:Chialvo2} does not allow for 
an easy and explicit separation of time-scales unless the $a$, $b$ 
and $c$ parameter values are specifically chosen. 
Therefore, although the slow-fast nature of the system 
for some parameters' values is reflected in numerical simulations 
(see e.g. \cite{chialvo1995}), rigorous application of singular
perturbation theory for this model and, perhaps, revealing
implications of our finding for the full model, 
can be a separate task for future research. 
Another approach can rely on using the theory of discrete-time
nonautonomous dynamical systems (see e.g. \cite{potzsche}) 
for inferring the $x$-dynamics in the full system.
However, as revealed e.g. 
by the numerical experiments related to Figure
\ref{fig:NEW_Review}, sometimes the behaviour 
of the reduced model might be more complicated 
than that of the 2D model within the corresponding area 
of the phase plane. This of course depends on the parameter
choice in the 2D model but also, together with the variety 
of activity modes which the model is able to display 
(as illustrated e.g. in Figures \ref{fig:BurstingPeriod4},
\ref{fig:KillingBursting} and \ref{fig:1DIMDynamics}), 
serves in favour of the argument that 1D model might be 
treated as the independent model and that, perhaps, 
it would be worth considering other variants of 
the 2D model with the same equation for the membrane voltage 
but different types of the recovery dynamics. In any case, 
the starting point for creating such new variants of 
the Chialvo model would be to fully understand the evolution 
of the membrane voltage given by \eqref{eq:Chialvo1DIM}, 
which this paper aims at. 

Moreover, as pointed also in
\cite{ibarz2011}, the Chialvo model does not fit exactly 
into the scheme of slow-fast neuron models with 
$x$ being the fast variable and $y$ the slow variable 
(in fact, $y$ might even represent the fast, 
not the slow, variable), because the impact of $y$ 
onto the voltage evolution equation is not additive 
but multiplicative and thus changing $y$ does not simply 
shifts the map $f$ upward or downward. 
Therefore understanding the bifurcation structure of 
the map $f_{r,k}$, especially with respect to $r$ parameter, 
is very important and was another motivation behind 
this work. 

On the other hand, providing complete description of 
the dynamics of the full model
\eqref{eq:Chialvo1}--\eqref{eq:Chialvo2}
(corresponding to different parameters configuration in 
the $(a,b,c,d, k)$ space, or its $(b,k)$ subspace), 
seems challenging. The works \cite{chialvo1995} 
and \cite{Jing2006} gave some initial results in 
this direction but this extensive task demands 
other supporting methods, for instance, 
topological and computational dynamics tools 
(identifying invariant sets and their Conley
indices) and might be a base for future works. 
From that point of view, 
presented here, rather complete, dynamical description of 
the somehow independent 1D model seems to be very 
natural and useful.

\section*{Acknowledgments}
Frank Llovera Trujillo and Justyna Signerska-Rynkowska 
were supported by NCN (National Science Centre, Poland) 
grant no.~2019/35/D/ST1/02253. 
Frank Llovera Trujillo was also supported 
by Polonium International Doctoral Fellowships 
(1st edition) scholarship awarded by 
Gda\'nsk University of Technology within 
``Initiative of Excellence - Research University'' program.  
Justyna Signerska-Rynkowska also acknowledges 
the support of Dioscuri program initiated 
by the Max Planck Society, jointly managed with 
the National Science Centre (Poland), and mutually funded 
by the Polish Ministry of Science and Higher Education 
and the German Federal Ministry of Education and Research.

\subsection*{Conflict of interest}

The authors declare no potential conflict of interests.

\subsection*{ORCID}
Frank Llovera Trujillo\,\orcidlink{0000-0001-5979-3584}
\url{https://orcid.org/0000-0001-5979-3584} \\
Justyna Signerska-Rynkowska\,\orcidlink{0000-0002-9704-0425} \url{https://orcid.org/0000-0002-9704-0425}\\
Piotr Bart\l omiejczyk\,\orcidlink{0000-0001-5779-4428} \url{https://orcid.org/0000-0001-5779-4428}

\bibliography{main}

\appendix 

\section{Proofs of results on
the dynamical core}\label{appendix:core}

\begin{proof}[Proof of Lemma~\ref{lem:DynCore}]
The second inequality $c<f_{r,k}(c)$ is straightforward 
and moreover if we demand only this inequality 
it is enough to set $r^*(k)=r^*(0)$ for any $k\geq 0$. 
In order to prove the first inequality we compute
$f_{r,k}^2(c)=(4\exp(r-2)+k)^2\exp\{r-4\exp(r-2)-k\}+k$, 
which can be further expanded as
\[ 
f_{r,k}^2(c)  = \left\{
16\exp(3r-4-k)+
8k\exp(2r-2-k)+ 
k^2\exp(r-k)\right\}\exp\{-4\exp(r-2)\}+k. 
\]
Since for large $r$ (with fixed $k\in [0,2)$) 
\[
\exp\{-4\exp(r-2)\}\leq 
\min{\{\exp(-4r+4+k), \exp(-3r+2+k), \exp(-2r+k)\}}, 
\]
we have 
\[
f_{r,k}^2(c) \leq\left\{16+8k+k^2\right\}\exp(-r) +k<2 
\]
for sufficiently large $r$, which proves \eqref{LemmaDynCoreEq}.  

Note that the second statement about the unique fixed point in 
the interval $[f_{r,k}^2(c),f_{r,k}(c)]$ means that there are 
no fixed points $x_f$ on the increasing part of the graph 
of $f_{r,k}$ with $x_f\geq f_{r,k}^2(c)$ (as necessarily 
the interval $[f_{r,k}^2(c),f_{r,k}(c)]$ contains 
the fixed point $x_f>2$).  Thus let $0<k<2$. 
In this case for $r$ large enough and $k\leq x\leq 2$ 
the following sequence of inequalities holds:
\[ 
\exp(r-x)\geq \exp(r-2)>\frac{1}{k}-\frac{k}{4}\geq \frac{1}{x}-\frac{k}{x^2}=\frac{x-k}{x^2}, 
\]
which implies
\[
f_{r,k}(x)=x^2\exp(r-x)+k>x.
\]
Consequently, for large  $r$ enough there are no fixed points in
$[k,2]$. However, $x^2\exp(r-x)+k >x$ for all $0\leq x<k$,
independently of $r$, yielding that for $k>0$ and sufficiently 
large $r$ there are no fixed points in $[0,2]$ and completing 
the proof. 
\end{proof}

\begin{proof}[Proof of Lemma~\ref{lem:UnstableFP}]
Since it is clear that for any fixed $0\leq k<2$ and $r=r(k)$ 
large enough there is a fixed point located in 
the right subinterval $[c,f_{r,k}^2(c)]$, it remains to show 
that the fixed point $x_f>2$ is unstable for all $r$ large enough. 

Let $0\leq k<2$ be fixed and assume that $x_f>2$ is a fixed point 
of $f=f_{r,k}$, where $r$ is some parameter value. 
We have to show that 
$\vert f^{\prime}(x_f)\vert = x_f(x_f-2)\exp(r-x_f)>1$ 
for $r$ large enough (note that the value of $x_f$ 
also changes with $r$). Since the fixed point $x_f$ satisfies
$x_f=x_f^2\exp(r-x_f)+k$, we only need to show that
$(x_f-k)(x_f-2)/(x_f)>1$  (note that $x_f>2>k$). 
This last inequality is satisfied when
\begin{equation}\label{eqpom1}
x_f<\frac{3+k-\sqrt{(k-1)^2+8}}{2}
\end{equation}
or
\begin{equation}\label{eqpom2}
x_f>\frac{3+k+\sqrt{(k-1)^2+8}}{2}.
\end{equation}
The inequality \eqref{eqpom1} must be rejected as we are interested 
in $x_f>2$ and this contradicts \eqref{eqpom1}. 
Instead, let us focus on the inequality \eqref{eqpom2}. 
From Theorem \ref{thm:flip} on the period doubling bifurcation 
it follows that there exists a parameter value $r_1$ such that
$x_f=x_{f,r_1}>2$ is an unstable fixed point of $f=f_{r_1,k}$. 
Thus necessarily $x_{f,r_1}$ satisfies \eqref{eqpom2}. 
Let $r_2>r_1$. Then $x_{f,r_2}>x_{f,r_1}$, where $x_{f,r_2}$ is 
a fixed point of $f_{r_2,k}$, located in the subinterval
$[2,f_{r_2,k}(2)]$. Indeed, if $x_f>2$ is a fixed point 
of $f_{r,k}$ then $\exp(r)=(x_f^{-1}-k x_{f}^{-2})\exp(x_f)$. 
The left hand side of this equation is an increasing function 
of $r$ and the function $(x^{-1}-k x^{-2})\exp(x)$ is 
an increasing function of $x$ ($k$ is kept constant). 
Thus increasing $r$, moves the fixed point $x_f>2$ further to 
the right. It follows that $x_{f,r_2}$ also satisfies 
\eqref{eqpom2} and thus the fixed point $x_{f,r_2}$ 
is also unstable. 
\end{proof}

\section{Proof of the condition
for topological chaos}\label{appendix:chaos}

\begin{proof}[Proof of Proposition~\ref{prop:topchaos}]
To show that
\[
f_r^2(2)<f_r^3(2)<2<f_r(2)
\quad\text{for $2.6\leq r \leq 2.9$},
\]
we need only consider three cases:\vspace{1mm}

\noindent\emph{Case I: \fbox{$2<f_r(2)$.}}
One can check that $2<f_r(2)$ if{f} $r>2-\ln(2)\approx1.3$. 
But we assume $r\geq 2.6$. 
\vspace{1mm}

\noindent\emph{Case II: \fbox{$f_r^3(2)<2$.}}
An easy computation shows that 
\[
f_r^3(2)=256\exp(7r-8-8\exp(r-2)-16\exp(3r-4-4\exp(r-2))).
\]
Let us define $h(r)=f_r^3(2)-2$. Then 
\[
\frac{dh}{dr}=
256s(r)\exp(7r-8-8\exp(r-2)-16\exp(3r-4-4\exp(r-2))),
\]
where 
$s(r)=7-8\exp(r-2)+64\exp(4r-6-4\exp(r-2))-48\exp(3r-4-4\exp(r-2))$.
Assume that $r\ge2.6$. Since $\exp(r-2)\ge\exp(0.6)>1.8$, we have
\begin{itemize}
    \item $7-8\exp(r-2)<-7.4$,
    \item $64\exp(4r-6-4\exp(r-2))<64\exp(-2.8)<3.9$.
\end{itemize}
Hence $s(r)<0$ for $r\ge2.6$ and, in consequence, $h$ is decreasing
in this range. As $h(2.6)\approx -0.027<0$,
we obtain $h(r)<0$ for $r\ge2.6$, which is our claim in this case.
\vspace{1mm}

\noindent\emph{Case III: \fbox{$f_r^2(2)<f_r^3(2)$.}}
Similar considerations applied to the function 
$g(r)=f_r^2(2)-f_r^3(2)$ on $[2.6,2.9]$ give 
$g(r)<0$ for $2.6\leq r \leq 2.9$,
and the proof is complete.
\end{proof}

\end{document}